\documentclass[11pt]{amsart}

\usepackage[top=27 mm, bottom=20 mm, left=25 mm, right=25mm]{geometry}

\usepackage[all]{xy}

\usepackage{mathrsfs}
\usepackage{amsmath,amsthm, amssymb, amsgen,amsxtra,amsfonts, amsbsy} 
\usepackage{xcolor}
\usepackage{verbatim}
\usepackage{ragged2e}

\usepackage{graphicx}
\usepackage{caption,rotating}
\usepackage{color}
\usepackage{subcaption}

\usepackage{tikz}
\usetikzlibrary{trees}
\usetikzlibrary{decorations.pathmorphing}
\usetikzlibrary{decorations.markings}
\tikzset{
    triple/.style args={[#1] in [#2] in [#3]}{
        #1,preaction={preaction={draw,#3},draw,#2}
    }
} 

\tikzstyle{nodal}=[circle,draw,fill=black,inner sep=0pt, minimum width=4pt]

\usepackage{multirow}

\usepackage{times}

\usepackage[font=small,labelfont=small]{caption}

\usepackage{enumitem}

\usepackage{mathtools,stackengine}

\usepackage{pdflscape}
\usepackage{array}

\usepackage{float}
\restylefloat{table}
\usepackage{adjustbox}

\definecolor{uclablue}{rgb}{0.33, 0.41, 0.58}
\definecolor{ultramarine}{rgb}{0.07, 0.04, 0.56}
\definecolor{tropicalrainforest}{rgb}{0.0, 0.46, 0.37}
\definecolor{tealgreen}{rgb}{0.0, 0.51, 0.5}
\definecolor{sapphire}{rgb}{0.03, 0.15, 0.4} 
\definecolor{st.patricksblue}{rgb}{0.14, 0.16, 0.48} 
\definecolor{royalblue(traditional)}{rgb}{0.0, 0.14, 0.4} 
\definecolor{sacramentostategreen}{rgb}{0.0, 0.34, 0.25} 
\definecolor{myrtle}{rgb}{0.13, 0.26, 0.12}
\definecolor{forestgreen(traditional)}{rgb}{0.0, 0.27, 0.13}
\definecolor{tealblue}{rgb}{0.21, 0.46, 0.53}
 \definecolor{mediumpersianblue}{rgb}{0.0, 0.4, 0.65}
  \definecolor{smalt(darkpowderblue)}{rgb}{0.0, 0.2, 0.6}
  \definecolor{pakistangreen}{rgb}{0.0, 0.4, 0.0}
  	\definecolor{phthaloblue}{rgb}{0.0, 0.06, 0.54}
   \definecolor{dartmouthgreen}{rgb}{0.05, 0.5, 0.06}
   \definecolor{lincolngreen}{rgb}{0.11, 0.35, 0.02}
   	\definecolor{pinegreen}{rgb}{0.0, 0.47, 0.44}

\usepackage{setspace}

\usepackage[
           colorlinks=true,
            urlcolor=pinegreen,       
            filecolor=pinegreen,     
            linkcolor=phthaloblue,       
            citecolor=pinegreen,         
            pdftitle= {Title},
            pdfauthor={Author},
            pdfsubject={Subject},
            pdfkeywords={Key},
            pagebackref,
            bookmarks=false,
            bookmarksopen=true]
            {hyperref}

\theoremstyle{definition}
\newtheorem{Definition}{Definition}[section]

\newtheorem{Remark}[Definition]{Remark}
\newtheorem{Remark/Notation}[Definition]{Remark/Notation}
\newtheorem{Notation}[Definition]{Notation}
\newtheorem{Strategy}[Definition]{Strategy}
\newtheorem{Discussion}[Definition]{Discussion}

\newtheorem*{Acknowledgements}{Acknowledgements}

\theoremstyle{plain}

\newtheorem*{Main Theoremx}{Main Theorem}
\newtheorem{Proposition}[Definition]{Proposition}

\newtheorem{Lemma}[Definition]{Lemma}
\newtheorem{Corollary}[Definition]{Corollary}

\newtheoremstyle{voiditstyle}{3pt}{3pt}{\itshape}{\parindent}%
{\bfseries}{.}{ }{\thmnote{#3}}%
\theoremstyle{voiditstyle}

\newtheoremstyle{voidromstyle}{3pt}{3pt}{\rm}{\parindent}%
{\bfseries}{.}{ }{\thmnote{#3}}%
\theoremstyle{voidromstyle}

\numberwithin{equation}{section}

\newcommand{\bbP}{{\mathbb{P}}}

\newcommand{\Aut}{\operatorname{Aut}}

\newcommand{\Spec}{\operatorname{Spec}}

\newcommand{\Pic}{\operatorname{Pic}}

\newcommand{\Bl}{\operatorname{Bl}}

\newcommand{\Char}{\operatorname{char}}
\newcommand{\Stab}{\operatorname{Stab}}

\newcommand{\bsm}{\left(\begin{smallmatrix}}
\newcommand{\esm}{\end{smallmatrix}\right)}

\newcommand{\beq}{\begin{equation}}
\newcommand{\eeq}{\end{equation}}

\usepackage[square, numbers, comma, sort&compress]{natbib}

\usepackage{etoolbox}
\usepackage{tocvsec2}
\usepackage{dashrule}

\usepackage[none]{hyphenat}

\makeatletter

\def\@tocline#1#2#3#4#5#6#7{\relax
  \ifnum #1>\c@tocdepth
  \else
    \par \addpenalty\@secpenalty\addvspace{#2}%
    \begingroup \hyphenpenalty\@M
    \@ifempty{#4}{%
      \@tempdima\csname r@tocindent\number#1\endcsname\relax
    }{%
      \@tempdima#4\relax
    }%
    \parindent\z@
    \leftskip#3\relax
    \advance\leftskip\@tempdima\relax
    \rightskip\@pnumwidth plus4em
    \parfillskip-\@pnumwidth
    #5\leavevmode
    \hskip-\@tempdima
      \ifcase #1
        \or
        \or \hskip 2em
        \or \hskip 2.5em
        \else \hskip 3em
      \fi
      #6\nobreak\relax
    \dotfill
    \hbox to\@pnumwidth{\@tocpagenum{#7}}\par
    \nobreak
    \endgroup
  \fi
}

\renewcommand{\tocsubsection}[3]{%
  \indentlabel{%
    \makebox[\dimexpr5pc-2em\relax][l]{%
      \@ifnotempty{#2}{%
        \@ifempty{#1}{}{#1\space}#2.%
      }%
    }%
  }%
  #3%
}

\makeatother

\renewcommand{\arraystretch}{1.15}

\begin{document}

\title[Rational (quasi-)elliptic surfaces with global vector fields in odd characteristic]{Rational (quasi-)elliptic surfaces with global vector fields\\ in odd characteristic}

\author{Claudia Stadlmayr}
\address{Univerit{\'e} de Neuch{\^a}tel, Institut de math{\'e}matiques, Rue Emile-Argand 11, 2000 Neuch{\^a}tel, Switzerland}
\curraddr{}
\email{\underline{claudia.stadlmayr@unine.ch}}

\date{\today}

\maketitle

\begin{abstract}
We classify rational elliptic and quasi-elliptic
surfaces with global vector fields over arbitrary algebraically closed fields of
characteristic $p \geq 0$ different from $2$. For every such surface, we determine the multiple and reducible fibers, the identity component of the automorphism
scheme, and the
moduli. As a corollary, we deduce that rational (quasi-)elliptic surfaces with global vector fields are Jacobian if $p \neq 3,5$ and we describe all counterexamples in small characteristics.
\end{abstract}

\tableofcontents

\section{Introduction}

Blowing up a point of the projective plane is among the simplest
operations in algebraic geometry; iterating it is not. Over an algebraically closed
field, a smooth del Pezzo surface is either $\mathbb P^1 \times \mathbb P^1$ or the blow-up of
$\mathbb P^2$ in at most eight points in general position, i.e. no three are
collinear, no six lie on a conic, and, in the eight-point case, the
points do not lie on a cubic with one of them being the singular point
(see \cite[Chapter 8]{Dolgachev-classical}). Already their classical automorphism
groups show a rigidity threshold that depends on the number of points in
$\mathbb P^2$ that have been blown up: the blow-up of three non-collinear
points is toric with connected automorphism group $\mathbb G_m^2$, whereas
the automorphism group of a smooth del Pezzo surface of degree at most $5$,
equivalently the blow-up of at least four points in general position, is finite
(see \cite{DolgachevMartinOddAut} or \cite{WeakDelPezzoGlobalVectorFields}).

The apparently elementary problem of blowing up points of $\mathbb P^2$
becomes subtler in three directions of possible generalization: one may increase
the number of points, impose special incidence conditions on them, or pass to a
ground field of positive characteristic.

For $X_r$ the blow-up of $\mathbb P^2$ in $r$ points, we have
$K_{X_r}^2=9-r$. For $r \leq 8$ in general position, $-K_{X_r}$ is ample and
the $(-2)$-classes in $K_{X_r}^{\perp}$ form a finite root system of type $E_r$
(see e.g. \cite{RDPDelPezzoGlobalVectorFields}). At $r=9$, the canonical class
is isotropic and $K_{X_9}^{\perp}$ becomes affine of type $\widetilde E_8$.
For $r \geq 10$, the lattice $K_{X_r}^{\perp}$ has signature $(1,r-1)$ and
hence becomes indefinite; for $r=10$, this is the hyperbolic $E_{10}$ lattice.
In this range, infinitely many negative curves or automorphisms of positive
entropy may occur \cite{CantatDolgachev,ArtebaniLafaceCox,McMullenDynamicsBlowups}.
Thus, the ninth blow-up is the transition from the finite, negative-definite
root systems associated with del Pezzo surfaces to the affine case.

Moreover, even for a fixed number $r$ of blown-up points, the geometry heavily
depends on their incidences. Passing from points in general position to points
in almost general position generalizes classical del Pezzo surfaces to
weak del Pezzo surfaces: $-K_X$ stays nef and big, but may have degree
zero on $(-2)$-curves. Such curves arise, for example, from three collinear
points, six points on a conic, or infinitely near points. Their ADE-configurations
are contracted by the anti-canonical model to rational double points
\cite{WhichRDPsOccurOnDelPezzoSurfaces,RDPDelPezzoGlobalVectorFields}.
Special incidence can also force the existence of a pencil or of a
positive-dimensional stabilizer in $\operatorname{PGL}_3$.

Furthermore, the incidence geometry itself can depend on the characteristic. For
example, in characteristic $2$, the blow-up of the seven points
$\mathbb P^2(\mathbb F_2)$ yields a configuration of seven disjoint $(-2)$-curves,
i.e. a $7A_1$-configuration, respectively seven $A_1$ rational double points on
the anti-canonical model. This is a configuration that occurs on del Pezzo
surfaces only in characteristic $2$
\cite{KeelMcKernan,CasciniTanakaKawamata,KawakamiNagaokaNeu-Pathologies,WhichRDPsOccurOnDelPezzoSurfaces}.
Small characteristic also permits quasi-elliptic fibrations in $p=2,3$,
non-taut rational double points in $p=2,3,5$, non-reduced automorphism schemes,
and infinitesimal automorphisms of a singular model that do not lift to its
minimal resolution
\cite{CossecDolgachevLiedtke,WeakDelPezzoGlobalVectorFields,RDPDelPezzoGlobalVectorFields}.

These three ways of generalizing the geometry of classical del Pezzo surfaces
all yield interesting \linebreak phenomena at the nine-point threshold, where we leave the
world of Fano varieties. Under suitable incidence conditions on the nine
blown-up points, one obtains a so-called Halphen pencil
\cite{Halphen-original}. A Halphen pencil of index $m$ is a pencil of curves of
degree $3m$ with nine, possibly infinitely near, base points of multiplicity $m$.
Resolving the base points gives a rational surface with a relatively minimal
genus-one fibration. It is elliptic if the geometric generic fiber is smooth,
quasi-elliptic, which can occur only in $p=2,3$, if the geometric generic fiber
is cuspidal, and Jacobian if it has a section. For nine distinct points
$p_1,\ldots,p_9$ on a smooth cubic $C$, the pencil exists exactly when
$\mathcal O_C(3H-p_1-\cdots-p_9)$ has finite order $m$ in $\operatorname{Pic}^0(C)$;
then $|-mK_X|$ is a Halphen pencil of index $m$. Here $m=1$ is the Jacobian case,
whereas $m>1$ gives a non-Jacobian fibration with one multiple fiber of
multiplicity $m$ \cite{CossecDolgachev,CantatDolgachev}.
\footnote{The relevance of nine blow-ups, as opposed to other numbers, is not confined to
classical birational geometry. It also appears as a boundary case when studying splinters \cite{KrahVialSplinters} or phantoms \cite{KrahPhantom, LoomingPhantom}, just to name a few other perspectives.}

In this article, we study, scheme-theoretically and
infinitesimally, the automorphisms of rational genus-one surfaces, which can be realized as
blow-ups of $\mathbb P^2$ in nine points. By Matsumura–Oort
\cite{MatsumuraOort}, the \linebreak automorphism functor of a proper variety is
represented by a group scheme locally of finite type, and its tangent space at
the identity is the space of global vector fields. In positive characteristic,
finite non-reduced group schemes such as $\mu_p$ and $\alpha_p$ may have
non-zero tangent space while having no non-trivial geometric points. Thus,
classifying surfaces with global vector fields is genuinely finer than
classifying surfaces with infinite abstract automorphism group
\cite{MatsumuraOort,Martin}. We classify all rational (quasi-)elliptic surfaces
with global vector fields, describe their connected components of the identity
explicitly, and determine the types of reducible and multiple fibers that occur.

Singular fibers and Mordell–Weil groups of rational (quasi-)elliptic surfaces
are well studied \cite{ExtremalChar0,ExtremalCharpII, EllipticChar3, HarbourneLang, QuasiEllipticChar2,QuasiEllipticChar3}.
Dolgachev and Martin developed a structure theory for the \linebreak automorphism groups
of rational genus-one surfaces in all characteristics, without assuming a
section \cite{DolgachevMartin}. Their description shows that every such surface
with global vector fields is the blow-up of a weak del Pezzo surface of degree
$1$ with global vector fields \cite[Theorem 4.2 and Remark 4.4]{DolgachevMartin};
this is our starting point. The other ingredients are the RDP configurations in
degree $1$ \cite{WhichRDPsOccurOnDelPezzoSurfaces}, the weak del Pezzo surfaces
of degree $1$ with global vector fields together with their automorphism schemes
\cite{WeakDelPezzoGlobalVectorFields}, and, in odd characteristic, the explicit
anti-canonical equations and liftable infinitesimal automorphisms
\cite{RDPDelPezzoGlobalVectorFields}, combined with Blanchard’s Lemma.

We use the notation for the types $1A$–$1I$ of weak del Pezzo surfaces
$\widetilde X$ of degree $1$ with non-trivial global vector fields from
\cite[Table 6]{WeakDelPezzoGlobalVectorFields}. If a fiber is denoted by
$m\Sigma$, then $m$ is its multiplicity and $\Sigma$ is the Kodaira–-N{\'e}ron type
of its primitive fiber divisor.

\begin{Main Theoremx}
\label{thm: class jaco und non-jaco, char all}
\label{MainTheorem}
Let  $f\colon\widetilde Z\to \mathbb P^1$
be a rational (quasi-)elliptic surface over an algebraically closed field $k$ of characteristic $\Char(k)=p\neq 2$. Then, $h^0(\widetilde Z,T_{\widetilde Z})\neq0$
if and only if exactly one of the following holds.
\begin{enumerate}[label=\textup{(\Alph*)},ref=\Alph*]
\item\label{Main theorem Jacobian}\label{main theorem Jacobian part}
The fibration $f$ is Jacobian and $\widetilde Z$ is one of the surfaces
in Table \ref{Table Jacobian classification characteristic not two}.

\item\label{Main theorem non-Jacobian}\label{main theorem nonJacobian part}
The fibration $f$ is non-Jacobian and $\widetilde Z$ is one of the
surfaces in Table \ref{Table non-Jacobian classification characteristic not two}.
\end{enumerate}
All cases listed in the two tables occur and can be realized as blow-ups of weak del Pezzo surfaces $\widetilde{X}$ of degree $1$ with $h^0(\widetilde X,T_{\widetilde X})\neq0$ of type $\Gamma$ . 
\end{Main Theoremx}

\vspace{-1cm}
\begin{table}[H]
\centering
\renewcommand{\arraystretch}{1.28}
\setlength{\tabcolsep}{4.5pt}
\begin{adjustbox}{center,max width=\textwidth}
\begin{tabular}{|c|c|c|c|c|c|c|}
\hline
\bf{N\textordmasculine}
& $\boldsymbol{\Gamma}$
& \begin{tabular}{c}\textbf{Fiber}\\[-1mm]\textbf{configuration}\end{tabular}
& $\boldsymbol{\Aut_{\widetilde Z}^0}$
& $\boldsymbol{h^0(\widetilde Z,T_{\widetilde Z})}$
& \textbf{Moduli}
& $\boldsymbol{p}$
\\
\hline\hline
$1_{\rm e}$
& \hyperref[Tab1A]{$1A$}
& ${\rm I}_0^*+{\rm I}_0^*$
& $\mathbb G_m$
& $1$
& $1$-dimensional
& $p\neq2$
\\
\hline
$2_{\rm e}$
& \hyperref[Tab1B]{$1B$}
& ${\rm IV}^*+{\rm IV}$
& $\mathbb G_m$
& $1$
& unique
& $p\neq2,3$
\\
\hline
$2_{\rm q}$
& \hyperref[Tab1B]{$1B$}
& ${\rm IV}^*+{\rm IV}$
& $\mathbb G_m$
& $1$
& unique
& $p=3$
\\
\hline
$3_{\rm e}$
& \hyperref[Tab1C]{$1C$}
& ${\rm III}^*+{\rm III}$
& $\mathbb G_m$
& $1$
& unique
& $p\neq2$
\\
\hline
$4_{\rm e}$
& \hyperref[Tab1D]{$1D$}
& ${\rm II}^*+{\rm II}$
& $\mathbb G_m$
& $1$
& unique
& $p\neq2,3$
\\
\hline
$5_{\rm e}$
& \hyperref[Tab1E]{$1E$}
& ${\rm I}_3^*+{\rm II}$
& $\mu_3$
& $1$
& unique
& $p=3$
\\
\hline
$6_{\rm e}$
& \hyperref[Tab1F]{$1F$}
& ${\rm III}^*+{\rm II}$
& $\mu_3$
& $1$
& unique
& $p=3$
\\
\hline
$7_{\rm e}$
& \hyperref[Tab1G]{$1G$}
& ${\rm I}_9+{\rm II}$
& $\mu_3$
& $1$
& unique
& $p=3$
\\
\hline
$8_{\rm e}$
& \hyperref[Tab1H]{$1H$}
& ${\rm II}^*$
& $\mathbb G_a$
& $1$
& unique
& $p=3$
\\
\hline
$8_{\rm q}$
& \hyperref[Tab1I]{$1I$}
& ${\rm II}^*$
& $\mathbb G_a\rtimes\mathbb G_m$
& $2$
& unique
& $p=3$
\\
\hline
\end{tabular}
\end{adjustbox}
\caption{Jacobian rational (quasi-)elliptic surfaces with global vector fields in
$\Char(k)=p \neq 2$.}
\label{Table Jacobian classification characteristic not two}
\label{Table main theorem Jacobian}
\end{table}
\vspace{2cm}
\pagebreak

\begin{table}[H]
\centering
\renewcommand{\arraystretch}{1.25}
\setlength{\tabcolsep}{3.2pt}
\begin{adjustbox}{center,max width=\textwidth}
\begin{tabular}{|c|c|c|c|c|c|c|c|c|}
\hline
\bf{N\textordmasculine}
&  $\boldsymbol{\Gamma}$
& \textbf{Index $\boldsymbol{m}$}
& \begin{tabular}{c}\textbf{Multiple}\\[-1mm]\textbf{fiber}\end{tabular}
& \begin{tabular}{c}\textbf{Reducible}\\[-1mm]\textbf{fibers}\end{tabular}
& $\boldsymbol{\Aut_{\widetilde Z}^0}$
& $\boldsymbol{h^0(\widetilde Z,T_{\widetilde Z})}$
& \textbf{Moduli}
& $\boldsymbol{p}$
\\
\hline\hline
$1$
& \begin{tabular}{c}
\hyperref[cor 1B p=3]{$1B$},
\hyperref[cor: 1E char3 nonJac revised]{$1E$},\\[-1mm]
\hyperref[cor 1F uniqueness(1) solo]{$1F(1)$},
\hyperref[cor 1G uniqueness(1) II]{$1G(1)$},\\[-1mm]
\hyperref[cor 1I uniqueness II]{$1I(2)$}
\end{tabular}
& $3$
& $3{\rm II}$
& ${\rm II}^*$
& $\mu_3$
& $1$
& unique
& $3$
\\
\hline
$2$
& \begin{tabular}{c}
\hyperref[cor: 1C non-Jac]{$1C$},
\hyperref[cor 1F uniquenes(2,3)(2)]{$1F(2)$}
\end{tabular}
& $3$
& $3{\rm III}^*$
& ${\rm III}^*+{\rm I}_2$
& $\mu_3$
& $1$
& unique
& $3$
\\
\hline
$3$
& \hyperref[cor: 1D non-Jac]{$1D$}
& $5$
& $5{\rm II}^*$
& ${\rm II}^*$
& $\mu_5$
& $1$
& unique
& $5$
\\
\hline
$4$
& \hyperref[cor 1F uniqueness(2,3)(3) 1dim]{$1F(3)$}
& $3$
& $3{\rm III}^*$
& ${\rm III}^*$
& $\mu_3$
& $1$
& $1$-dimensional
& $3$
\\
\hline
$5$
& \hyperref[cor 1G torsion I9]{$1G(2)$}
& \begin{tabular}{c}$m>1$\\[-1mm]$(m,3)=1$\end{tabular}
& $m{\rm I}_9$
& ${\rm I}_9$
& $\mu_3$
& $1$
& \begin{tabular}{c}
unique if $m=2$;\\[-1mm]
$\varphi(m)/2$ classes if $m>2$
\end{tabular}
& $3$
\\
\hline
$6_{\rm e}$
& \hyperref[cor: 1H char3 nonJac revised]{$1H$}
& $3$
& $3{\rm II}^*$
& ${\rm II}^*$
& $\mathbb G_a$
& $1$
& $1$-dimensional
& $3$
\\
\hline
$6_{\rm q}$
& \hyperref[cor 1I uniqueness IIstar]{$1I(1)$}
& $3$
& $3{\rm II}^*$
& ${\rm II}^*$
& $\mathbb G_a$
& $1$
& unique
& $3$
\\
\hline
\end{tabular}
\end{adjustbox}
\caption{Non-Jacobian rational (quasi-)elliptic surfaces
$\widetilde Z$ with global vector fields in $\Char(k)=p \neq 2$.}
\label{Table non-Jacobian classification characteristic not two}
\label{Table main theorem nonJacobian}
\label{Table: class jaco und non-jaco, char all}
\end{table}

\vspace{1cm}

The two above tables have slightly different meanings, depending on whether $f:\widetilde{Z} \to \mathbb{P}^1$ admits a section or not. In the Jacobian case,
there is a bijection between
isomorphism classes of weak del Pezzo surfaces of degree $1$ and
Jacobian rational (quasi-)elliptic surfaces that also identifies their connected components of the automorphism group schemes. Thus the type $\Gamma$ in
Table \ref{Table Jacobian classification characteristic not two} is
intrinsic. In the non-Jacobian case, by contrast, the column $\Gamma$
records a chosen contraction of a $(-1)$-curve and need not be intrinsic
to $\widetilde Z$.
Whenever such an alternative contraction is established in the
case-by-case analysis in Section \ref{section strategy proof main theorem}, the corresponding presentations are combined in
one row of Table \ref{Table non-Jacobian classification characteristic not two}.
The column ``reducible fibers'' lists all reducible
fibers, including the reduced multiple fiber whenever it is reducible.
The entries $6_{\rm e}$ and $6_{\rm q}$ have the same displayed fiber
types and the same connected automorphism scheme, but are not
isomorphic since the former surfaces are elliptic, whereas the latter surface
is quasi-elliptic. Finally, in {N\textordmasculine} $5$, $\varphi$ denotes Euler's
totient function; for fixed $m$, the isomorphism classes are
parametrized by primitive $m$-th roots of unity modulo inversion.

Several immediate consequences of the Main Theorem are worth recording.

\begin{Corollary}
\label{cor: nonJacobian characteristic restriction}
\label{cor: class jaco und non-jaco, char neq 2,3}
Let $\widetilde Z$ be a non-Jacobian rational (quasi-)elliptic surface
over $k$ of odd characteristic.
\begin{enumerate}[label=\textup{(\alph*)}]
\item If $p\notin\{3,5\}$, then $h^0(\widetilde Z,T_{\widetilde Z})=0.$
\item If $p=5$, then
$h^0(\widetilde Z,T_{\widetilde Z})\neq0$ if and only if
$\widetilde Z$ is the unique surface in {N\textordmasculine} $3$ of
Table \ref{Table non-Jacobian classification characteristic not two}.
\item If $p=3$, then
$h^0(\widetilde Z,T_{\widetilde Z})\neq0$ if and only if
$\widetilde Z$ occurs in one of {N\textordmasculine}s $1$, $2$, $4$, $5$, $6_{\rm e}$, or $6_{\rm q}$ of Table \ref{Table non-Jacobian classification characteristic not two}.
\end{enumerate}
\end{Corollary}

In particular, if $p\notin\{2,3,5\}$, every rational
(quasi-)elliptic surface with a global vector field is Jacobian; its
connected automorphism scheme is $\mathbb G_m$ and its space of global
vector fields is one-dimensional. This applies, in particular, in
characteristic zero.

\begin{Corollary}
\label{cor: reduced automorphism schemes rational genus one}
Let $k$ be an algebraically closed field of characteristic $p\neq2$.
Then the following hold.
\begin{enumerate}[label=\textup{(\alph*)}]
\item All Jacobian rational (quasi-)elliptic surfaces over $k$ have
reduced automorphism schemes if and only if $p\neq3$. 
\item All non-Jacobian rational (quasi-)elliptic surfaces over $k$ have
reduced automorphism schemes if and only if $p\notin\{3,5\}$. 
\end{enumerate}
\end{Corollary}

\begin{Corollary}
\label{cor: tangent dimension bound rational genus one}
Let $\widetilde Z$ be a rational (quasi-)elliptic surface over an
algebraically closed field of characteristic $p\neq2$. Then $h^0(\widetilde Z,T_{\widetilde Z})\leq2$.
Equality holds if and only if $p=3$ and $\widetilde Z$ is the unique
Jacobian surface {N\textordmasculine} $8_{\rm q}$ in Table \ref{Table main theorem Jacobian}.
\end{Corollary}

The article is organized as follows.
Section \ref{section generalities rational genus one automorphisms}
collects the basic facts on rational genus-one fibrations, Halphen
pencils, weak del Pezzo surfaces of degree $1$, and automorphism
schemes. Section \ref{section consequences vector fields} establishes
the reductions used in the proofs of parts
\textup{(\ref{Main theorem Jacobian})} and
\textup{(\ref{Main theorem non-Jacobian})} of the Main Theorem:
every rational (quasi-)elliptic surface $\widetilde Z$ with global
vector fields, equivalently with
$\Aut_{\widetilde Z}^0\neq\{1\}$, can be realized as the blow-up of a
weak del Pezzo surface $\widetilde X$ of degree $1$ with global vector
fields (with center the unique base point of
$|-K_{\widetilde X}|$ in the Jacobian case, and with center a torsion point on the identity component of an
anti-canonical curve in the non-Jacobian case).

After proving part \textup{(\ref{Main theorem Jacobian})} of the Main
Theorem directly, we formulate in Section
\ref{section strategy proof main theorem} a strategy for the proof of
part \textup{(\ref{Main theorem non-Jacobian})}. It treats separately
each type of degree $1$ weak del Pezzo surface with global vector
fields and determines its \emph{admissible} blow-up centers, that is,
the points for which the resulting surface is rational
(quasi-)elliptic and still admits global vector fields. The principal
ingredients are Blanchard's Lemma and the description of the
automorphism scheme of a blow-up in terms of stabilizer subgroup
schemes; see Subsection
\ref{subsection automorphism schemes and blow-ups}.

The discussion attached to each case $1A$--$1I$ contains diagrams
comparing the relevant surfaces $\widetilde Z$, $\widetilde X$, and
$\widetilde Y$. These diagrams display the negative curves and the
reducible fibers, as well as the alternative weak del Pezzo
presentations used to identify repeated cases in the classification.

\begin{Remark}
We restrict throughout this article to characteristic different from
$2$. The geometric reduction to Halphen centers and
scheme-theoretic stabilizers remains valid in characteristic $2$, but
its effective application requires explicit Weierstra{\ss} equations
together with the subgroup schemes of $\Aut_X^0$ whose actions lift to
the minimal resolution. The characteristic $3$ analysis of the types
$1F$ and $1G$ already shows that determining the liftable action is a
substantial part of the argument. In characteristic $2$, non-taut
rational double points and non-liftable infinitesimal automorphisms
occur more extensively, and the connected automorphism schemes of the
relevant RDP del Pezzo surfaces were not determined in
\cite{RDPDelPezzoGlobalVectorFields}. The characteristic $2$ classification of genus-one fibrations will be carried out by the author in a sequel to this article.
\end{Remark}

\begin{Acknowledgements}
The author thanks Gebhard Martin for helpful comments on a first version of this article. The author's research is funded by the Swiss National Science Foundation Grant \emph{Del Pezzo surfaces over perfect fields (10003572)}.
\end{Acknowledgements}

\section{Generalities: rational (quasi-)elliptic surfaces and automorphisms}
\label{section generalities rational genus one automorphisms}
\label{section Halphen theory}
\label{section: automorphism schemes: from weak del Pezzo to rational (quasi)ell}

\subsection{Rational genus-one fibrations}
\label{subsection rational genus one fibrations generalities}
Let us recall some notions and general facts about rational (quasi-)elliptic surfaces (sometimes also called \emph{genus-one fibration} if one does not want to distinguish between the elliptic and the quasi-elliptic
cases), following for example \cite[Chapter 4]{CossecDolgachevLiedtke}.

\begin{Definition}
\label{def rational genus one surface}
A projective surface $\widetilde Z$ is called a \emph{rational
(quasi-)elliptic surface} if it is smooth and rational and admits a morphism
$f\colon\widetilde Z\to\mathbb P^1$ such that
\begin{itemize}
\item $f$ is surjective and
$f_*\mathcal O_{\widetilde Z}\cong\mathcal O_{\mathbb P^1}$,
\item the generic fiber is a regular, geometrically integral curve of
arithmetic genus $1$, and
\item no fiber contains a $(-1)$-curve.
\end{itemize}
We call $f$ \emph{elliptic} if its generic fiber is smooth, and
\emph{quasi-elliptic} otherwise. We call it \emph{Jacobian} if it admits a
section and \emph{non-Jacobian} otherwise.
\end{Definition}

If $f$ is quasi-elliptic, then its geometric generic fiber is a cuspidal
rational curve, and this can occur only in characteristics $2$ and $3$.

Now, let $f: \widetilde{Z} \to \mathbb{P}^1$ be a rational (quasi-)elliptic surface. 
 We denote the fiber of $f$ over $P \in \mathbb{P}^1$ by $F_P$.
 The greatest common divisor $m_P$ of the multiplicities of the components of $F_P$ is called the \emph{multiplicity} of the fiber $F_P$ and we set $\overline{F}_P \coloneqq \frac{1}{m_P}F_P$. The fibers with $m_P > 1$ are called \emph{multiple fibers}. With this notation, and since $f$ has no wild fibers by \cite[Proposition 5.6.1]{CossecDolgachev}, the canonical bundle formula (see for example \cite[Theorem 7.15.]{Badescu}) yields
 \begin{equation}\label{eqn: canbundle}
 \omega_{\widetilde{Z}} \hspace{2mm} \cong \hspace{2mm} f^* \mathcal{O}_{\mathbb{P}^1}(-1) \otimes \mathcal{O}_{\widetilde{Z}}\left(\sum_{P \in \mathbb{P}^1}(m_P - 1)\overline{F}_P \right).
 \end{equation}
 Since $\widetilde{Z}$ is rational, for any $n>0$ we have $h^0( \widetilde{Z}, \omega_{\widetilde{Z}}^{\otimes n})= p_n(\widetilde{Z}) = p_n(\mathbb{P}^2) =0$, hence, by Equation \ref{eqn: canbundle}, there is at most one multiple fiber and we denote its multiplicity by $m$. Hence, the class of a fiber of $f$ in the Picard group is $-K_{\widetilde{Z}}$ if there are no multiple fibers, and $-mK_{\widetilde{Z}}$ if there is a multiple fiber. Thus, by adjunction, all $(-2)$-curves on $\widetilde{Z}$ have to lie in fibers of $f$.

 For $l > 0$, an \emph{$l$-section} of $f$ is an irreducible curve $C$ on $\widetilde{Z}$ such that $C.F_P = l$, where $F_P$ is a general fiber of $f$. The \emph{index} of $f$ is the minimal $l$ for which there exists an $l$-section. By \cite[Chapter V, Proposition 5.6.1.(vi)]{CossecDolgachev}, the index of a rational (quasi-)elliptic surface is equal to the maximal multiplicity of its fibers. In particular, it is $1$ if and only if $f$ does not admit a multiple fiber. In this case, $f$ admits a section and is called \emph{Jacobian} (see Definition 3.4 in \cite{WhichRDPsOccurOnDelPezzoSurfaces}). If $f$ admits a multiple fiber, then the index of $f$ is equal to the multiplicity $m$ of this unique multiple fiber. Moreover, by adjunction, the $(-1)$-curves on $\widetilde{Z}$ are exactly the $m$-sections of $f$.

\subsection{Halphen pencils and weak del Pezzo surfaces of degree $1$}
\label{subsection Halphen pencils and degree one weak del Pezzo surfaces}

For a rational (quasi-)elliptic surface $f: \widetilde{Z} \to \mathbb{P}^1$, we used in \cite[Subsection 3.2]{WhichRDPsOccurOnDelPezzoSurfaces} that in the \emph{Jacobian} case, $\widetilde{Z}=\widetilde{Y}$ is a blow-up of a weak del Pezzo surface of degree $1$. This connection persists also in the \emph{non-Jacobian} case as we are going to recall in the following (we refer the reader to \cite[Chapter 7]{Badescu}, \cite[Chapter V]{CossecDolgachev}, \cite[Chapter 4]{CossecDolgachevLiedtke} \cite{Dolgachev-on-Halphen}, \cite[Section 2.]{CantatDolgachev}, or \cite{Halphen-original} for further details):

 Let $\widetilde{X}$ be the contraction of any $m$-section of $f$. Since $K_{\widetilde{Z}}^2 = 0$ and $- K_{\widetilde{Z}}$ is nef, the anti-canonical divisor $- K_{\widetilde{X}}$ of $\widetilde{X}$ is nef with $K_{\widetilde{X}}^2 = 1$, so $\widetilde{X}$ is a weak del Pezzo surface of degree $1$.
 Hence, we can choose a realization of ${\widetilde{X}}$ as a blow-up of $\mathbb{P}^2$ in $8$ points and obtain a birational morphism $\pi: \widetilde{Z} \to \widetilde{X} \to \mathbb{P}^2$ which realizes $\widetilde{Z}$ as the blow-up of $9$ (possibly infinitely near) points. This yields the classical description of $f$ as the resolution of the base points of a \emph{Halphen pencil of index $m$}, which is a pencil of curves of degree $3m$, all of which have multiplicity $m$ at the $9$ points prescribed by $\pi$. The image of the multiple fiber of $f$ under $\pi$ is a cubic through these $9$ points, taken with multiplicity $m$.

 If $\widetilde{X}$ is a weak del Pezzo surface of degree $1$, then every curve in $|-K_{\widetilde{X}}|$ is of arithmetic genus one and marked with the base point $\widetilde{Q}$ of $|-K_{\widetilde{X}}|$. Note that this base point is a smooth point on every member of $|-K_{\widetilde{X}}|$, since $K_{\widetilde{X}}^2 = 1$. Hence, for every curve $C \in |-K_{\widetilde{X}}|$, the smooth locus of the irreducible component of $C$ that contains this base point becomes a group scheme (with $\widetilde{Q}$ as neutral element). We call this group scheme the \emph{identity component} of $C$ and denote it by $C^0$.

 The description of the previous paragraph has the following important consequence, which is explained in \cite[Theorem 4.2]{DolgachevMartin}:
 \begin{Proposition} \label{prop: torsionpoint}
 Let $f: \widetilde{Z} \to \mathbb{P}^1$ be a rational (quasi-)elliptic surface of index $m$. Then, $\widetilde{Z}$ is the blow-up of a weak del Pezzo surface $\widetilde{X}$ in a point $\widetilde{P}$ that satisfies the following conditions:
 \begin{enumerate}
     \item If $m = 1$, then $\widetilde{P}=\widetilde{Q}$ is the base point of $|-K_{\widetilde{X}}|$.
     \item If $m > 1$, then $\widetilde{P}\neq \widetilde{Q}$ lies on a unique curve $C \in |-K_{\widetilde{X}}|$. More precisely, $\widetilde{P}$ is a point of exact order $m$ on $C^0$. 
 \end{enumerate}
 \end{Proposition}

Using this, we can extend the connection between weak del Pezzo surfaces and Jacobian rational genus-one fibrations used in \cite{WhichRDPsOccurOnDelPezzoSurfaces} with non-Jacobian fibrations $f : \widetilde{Z} \to \mathbb{P}^1$ of index $m$ and Halphen pencils: 

Let $f_1, f_2$ be equations of cubics in $\mathbb{P}^2$ such that $8$ of their (possibly infinitely near) $9$ points of intersection lie in almost general position (see  in \cite[Subsection 2.1]{WeakDelPezzoGlobalVectorFields}). Let
$h$ be the equation of a curve of degree $3m$ in $\mathbb{P}^2$ with multiplicity $m$ at the above $8$ points in almost general position in $\{f_1=0=f_2\}$ and with one more point of multiplicity $m$ at a point $P_9$ different from the $9$ base points of the cubic pencil spanned by $f_1$ and $f_2$. Then, there exists precisely one curve in this cubic pencil passing through $P_9$. Let $g$ be the cubic equation of this curve. By Proposition \ref{prop: torsionpoint}, the point $P_9 \in \{g = 0\}$ is a point of exact order $m$ on $\{g = 0\}$. Moreover, the Halphen pencil corresponding to $\widetilde{Z}$ is spanned by $g^m$ and $h$, and $\widetilde{Z}$ is obtained by blowing up $\mathbb{P}^2$ in the $8$ common base points of the cubic pencil and the Halphen pencil, and $P_9$.
Note that simply blowing up the $9$ base points of the cubic pencil $\langle f_1, f_2 \rangle$ yields a Jacobian rational (quasi-)elliptic surface $\widetilde{Y}$.
The situation is summarized in the following Diagram \ref{Diagram connection3 plus pencils}. 

\vspace{5mm}

\begin{center}
{\large{
\begin{equation}
\hspace{1.7cm}
\begin{gathered} \hspace{-8mm}
 \label{Diagram connection3 plus pencils}
{\Small{\xymatrix{
   & \widetilde{Z} \ar@[gray][ldddd]  \ar@<0pt>@{}[ldddd]_{\tiny{\begin{tabular}{c} \vspace{-10mm}\textcolor{gray}{fibration}\end{tabular}}}  \ar[rrdd] \ar@<-26pt>@{}[rrdd]^[@]{{\tiny{\begin{tabular}{c} \hspace{-2mm}blow up $m$-tors. point \end{tabular}}}}   \ar@<-39.5pt>@{}[rrdd]^[@]{{\tiny{\begin{tabular}{c} \hspace{-1mm}contract $m$-section \end{tabular}}}} & & & & \widetilde{Y} \ar@[gray][rdddd] \ar@<1pt>@{}[rdddd]^{\tiny{\begin{tabular}{c} \vspace{-10mm}\textcolor{gray}{fibration}\end{tabular}}} \ar[lldd]  \ar@<34pt>@{}[lldd]_{\rotatebox{36,6}{\tiny{\begin{tabular}{c}  \hspace{3mm}blow up base point resp. \end{tabular}}}}   \ar@<-22pt>@{}[lldd]^{\rotatebox{36,6}{\tiny{\begin{tabular}{c} 
   contract section \end{tabular}}}}
 \ar[rd] \ar@<-14pt>@{}[rd]^{\rotatebox{329}{\tiny{\begin{tabular}{c} 
   Weierstra{\ss}\\model \end{tabular}}}} & & & \\
   & & & & & & Y \ar@/^1.9pc/@[gray][ddd] \ar[lldd] \ar@<32pt>@{}[lldd]_{\rotatebox{39,5}{\tiny{\begin{tabular}{c}  blow up base point resp. \end{tabular}}}}   \ar@<-23pt>@{}[lldd]^{\rotatebox{39,5}{\tiny{\begin{tabular}{c} 
   contract section \end{tabular}}}} \\
   & & & \widetilde{X} \ar@/^1.2pc/@{-->}@[gray][ddrrr] \ar@/_1.2pc/@{-->}@[gray][ddlll] \ar[dd] \ar@<3pt>@{}[dd]_{\tiny{\begin{tabular}{c} blow up $8$ \\ base points  \end{tabular}}} \ar@<-2pt>@{}[dd]^{\tiny{\begin{tabular}{c} $\pi$  \end{tabular}}} \ar[rd]  \ar@<-31pt>@{}[rd]^{\rotatebox{330}{\tiny{\begin{tabular}{c} 
  \hspace{-2mm} anti-can. mod. \end{tabular}}}} & & & & & \\
   & &  & & X \ar@{-->}[rrd] \ar@{^{(}->}[rrr]^{\tiny{\begin{tabular}{c} 
   \hspace{1.5mm}sextic hypersurface \end{tabular}}} & & &  \mathbb{P}(1,1,2,3)  \ar@{->>}[ld] \ar@<26pt>@{}[ld]_{\rotatebox{29}{\tiny{\begin{tabular}{c}  projection \end{tabular}}}}\\
   \mathbb{P}^1 & &  & \mathbb{P}^2 \ar@{-->}@[gray][lll]^{\tiny{\begin{tabular}{c} 
   \textcolor{gray}{Halphen pencil} \end{tabular}}} \ar@{-->}@[gray][rrr]_{\tiny{\begin{tabular}{c} 
   \hspace{-2mm}\textcolor{gray}{cubic pencil} \end{tabular}}} & & &   \mathbb{P}^1 & & & \\
   \text{\SMALL{$[g^m(x): h(x)]$}} & & & \text{\SMALL{$x = [x_0:x_1:x_2]$}} \ar@{|->}@[gray][rrr] \ar@{|->}@[gray][lll] & & & \text{\SMALL{$[f_1(x): f_2(x)]$}}
} }}
\end{gathered}
\end{equation}
}}
\end{center}

\vspace{7mm}

\begin{Notation} \label{Notation all objects Ztilde Xtilde Ytilde, ... points}
The following notations will be fixed throughout the rest of this article (if not explicitly stated otherwise): 
\begin{itemize}
    \item $\widetilde{Z}$ a non-Jacobian rational (quasi-)elliptic surface of index $m>1$
    \item $\widetilde{Y}$ a Jacobian rational (quasi-)elliptic surface with Weierstra{\ss} model $Y$
    \item $\widetilde{X}$ a weak del Pezzo surface of degree $1$ with corresponding anti-canonical model $\widetilde{X} \to X \subseteq \mathbb{P}(1,1,2,3)$
    \item $\widetilde{Q}$ the base point of $|-K_{\widetilde{X}}|$ with image $Q$ in $X$
    \item $\widetilde{Y}$ \emph{arises from} $\widetilde{X}$ if $\widetilde{Y} \cong {\Bl}_{\widetilde{Q}}(\widetilde{X})$.
    \item $C$ a curve in $|-K_{\widetilde{X}}|$ with identity component $C^0 \ni \widetilde{Q}$ 
    \item $\widetilde{P} \in \widetilde{X}$ a point of exact order $m>1$ on $C^0$ for $C \in |-K_{\widetilde{X}}|$ through $\widetilde{P}$
    \item $P$ the image of $\widetilde{P}$ in $X$
    \item $\widetilde Z$ \emph{arises from} $\widetilde{X}$ if there exists a $\widetilde{P} \in \widetilde{X}$ such that $\widetilde{Z} \cong {\rm Bl}_{\widetilde{P}}(\widetilde{X})$.
\end{itemize}
\end{Notation}

\begin{Remark} \label{remark: identifikation Z von unterschiedlichen X}
Note that, on $\widetilde{Y}$, all $(-1)$-curves are translates of each other, hence $\widetilde{Y}$ uniquely determines $\widetilde{X}$. This is not the case for $\widetilde{Z}$, so that $\widetilde{Z}$ can arise from non-isomorphic $\widetilde{X}$'s. We will use this to our advantage later in the proof of the non-Jacobian part of the Main Theorem.
\end{Remark}

\subsection{Kodaira--N\'eron types and reducible fibers}
\label{subsection reducible fibers}
We have seen above that a rational (quasi-)elliptic surface has at most one
multiple fiber: none in the Jacobian case, and exactly one, of multiplicity
equal to the index $m>1$, in the non-Jacobian case. In either case, it is
natural to ask what the remaining fibers look like. Since the generic fiber is
regular, geometrically integral, of arithmetic genus $1$, every fiber is either
a smooth elliptic curve, an irreducible rational curve with a single node or
cusp, or reducible.

By the canonical bundle formula, $K_{\widetilde Z}$ is supported on fibers; in
particular, adjunction shows that every component of a reducible fiber is a
$(-2)$-curve, and conversely every connected configuration of $(-2)$-curves is
contained in a single fiber. The possible fiber types were classified by
Kodaira and N\'eron. For the reader's convenience, we summarize the dictionary
between fiber types, affine Dynkin diagrams, and the rational double points on
the Weierstra{\ss} model $Y$, equivalently on the RDP del Pezzo surface $X$
(compare \cite{WhichRDPsOccurOnDelPezzoSurfaces} and
\cite[Chapter V]{CossecDolgachev}), in Table \ref{Table KodairaNeron2nd and RDPs mit tilde} below.

\begin{table}[h]
\renewcommand{\arraystretch}{1.35}
    \centering
    $ \begin{array}{|c||c|c|c|c|c|c|c|c|c|}
    \hline
        \text{Kodaira--N\'eron type} &  {\rm I}_0 & {\rm I}_n  & {\rm II} & {\rm III} & {\rm IV}  & {\rm I}_n^* & {\rm IV}^* & {\rm III}^* & {\rm II}^* \\
        \hline 
         \begin{tabular}{c}
         dual graph
         \end{tabular}
         & \widetilde{A}_0 & \widetilde{A}_{n-1} & \widetilde{A}_0 & \widetilde{A}_1 & \widetilde{A}_2  & \widetilde{D}_{4+n} & \widetilde{E}_6 & \widetilde{E}_7 & \widetilde{E}_8 \\
         \hline 
\begin{tabular}{c}
corresponding \\
rational double point
\end{tabular}
       & \text{smooth} & A_{n-1} & \text{smooth} & A_1 & A_2  & D_{4+n} & E_6 & E_7 & E_8 \\
       \hline
    \end{array} $
    \caption{Kodaira--N\'eron types and dual graphs} 
    \label{Table KodairaNeron2nd and RDPs mit tilde}
\end{table}

\begin{Notation} \label{notation: is of type Kodaira-Neron}
    For $C\in|-K_{\widetilde X}|$, we say $C$ \emph{is of type} $\Sigma$ if its strict
transform on $\widetilde Y=\Bl_{\widetilde{Q}}(\widetilde X)$ is a fiber of
Kodaira--N\'eron type $\Sigma$.
\end{Notation}

Whether a fiber of a Jacobian fibration is singular, and of which type, can be determined using Tate's
algorithm \cite{TateAlgo} (see also \cite[Subsection 4.1]{WhichRDPsOccurOnDelPezzoSurfaces}).

The number of reducible fibers is constrained by the Picard number of a
rational (quasi-)elliptic surface. We record the following bound, which will be used repeatedly to show that a given configuration of $(-2)$-curves already
accounts for all reducible fibers.

\begin{Lemma} \label{lemma (-2)curves on Ztilde}
    Let $f: \widetilde{Z} \to \mathbb{P}^1$ be a not necessarily Jacobian rational (quasi-)elliptic fibration. Let $e_P$ be the number of irreducible components of a fiber $F_P$ over $P \in \mathbb{P}^1$. Then, 
    $$
    \sum_{P \in \mathbb{P}^1} (e_P-1) \leq 8.
    $$
\end{Lemma}

\proof
By \cite[Theorem 6.6.]{LiuLorenziniRaynaud-alt,LiuLorenziniRaynaud-Corrigendum}, it suffices to prove the statement if $f$ has a section $\sigma$. Then, $\langle \sigma, F_P \rangle^{\perp} \subseteq {\rm Pic}(\widetilde{Z})$ has rank $8$. For every fiber $F_P$, the irreducible components of $F_P$ that do not meet $\sigma$ span a negative definite lattice of rank $(e_P -1)$ in $\langle \sigma, F_P \rangle^{\perp}$. Hence $\sum_{P \in \mathbb{P}^1} (e_P-1) \leq {\rm rank }(\langle \sigma, F_P \rangle^{\perp}) =8$.
\qed

\subsection{Automorphism schemes, Blanchard's Lemma, and blow-ups}
\label{subsection automorphism schemes and blow-ups}

For a $k$-scheme $V$, we denote by $\Aut(V)$ its classical group of
automorphisms. If, in addition, $V$ is proper, this group is the group of
$k$-valued points of the automorphism scheme $\Aut_V$, which represents the
automorphism functor of $V$ (see \cite{MatsumuraOort}). We denote by
$\Aut_V^0$ its connected component of the identity. As is well known, the
tangent space of $\Aut_V$ at the identity is
naturally identified with $H^0(V,T_V)$, the space of global vector fields on
$V$. Hence classifying rational (quasi-)elliptic surfaces with global vector
fields is equivalent to classifying those with non-trivial connected component
$\Aut_V^0$ of the automorphism scheme; in each case we will fully determine
this identity component.

The main tool for our study of automorphism schemes is the following lemma of
Blanchard (see \cite[Theorem 7.2.1]{Brion1}).

\begin{Lemma}[Blanchard's Lemma]
\label{lemma Blanchard generalities}
Let $g\colon V\to W$ be a morphism of proper schemes such that
$g_*\mathcal O_V\cong\mathcal O_W$. Then the action of $\Aut_V^0$ descends to
$W$ and induces a homomorphism $g_*\colon\Aut_V^0\longrightarrow\Aut_W^0$.
If $g$ is birational, this homomorphism is a closed immersion.
\end{Lemma}

If $g$ is the blow-up of a surface with at worst rational double points in a
closed point $P$, the image of $g_*$ can be described explicitly as a
stabilizer subgroup scheme (see \cite[Proposition 4.2]{RDPDelPezzoGlobalVectorFields} for a proof;
see also \cite[Lemma 2.11]{WeakDelPezzoGlobalVectorFields}, \cite{Neuman}, and
\cite{Martin}).

\begin{Lemma}
\label{lemma blowup stabilizer}
Let $W$ be a normal surface with at worst rational double points, and let
$g\colon V\to W$ be the blow-up of $W$ in a closed point $P$. Then
\[
\Aut_V^0\cong\bigl(\Stab_{\Aut_W}(P)\bigr)^0,
\]
where $\Stab_{\Aut_W^0}(P)$ denotes the stabilizer subgroup scheme of $P$.
\end{Lemma}

\section{Consequences for vector fields of rational (quasi-)elliptic surfaces}
\label{section consequences vector fields}

Applying these results to the description of Jacobian, respectively
non-Jacobian, rational (quasi-)elliptic surfaces as blow-ups of weak del Pezzo
surfaces $\widetilde{X}$ of degree $1$ already yields the strategy for the classification of such surfaces by
their connected automorphism schemes carried out in this article, since we have complete knowledge of $\Aut_{\widetilde{X}}^0$ by \cite{WeakDelPezzoGlobalVectorFields}.

\subsection{Strategies in the Jacobian and non-Jacobian case}
\label{subsection strategies Jacobian and non-Jacobian}

\begin{Corollary}[Jacobian case]
\label{cor: approach for classification jacobian}
There is a bijection of isomorphism classes
 $$
\begin{array}{r c l}
\left\{ 
\begin{tabular}{c}
Weak del Pezzo surfaces \\
of degree $1$
\end{tabular}
\right\} / \cong 
&  \longleftrightarrow 
& \left\{ 
\begin{tabular}{c}
Jacobian rational \\
(quasi-)elliptic surfaces
\end{tabular}
\right\} / \cong \vspace{3mm}
\\
 \widetilde{X}   & \longmapsto 
 & \hspace{-2mm}\begin{tabular}{l}
 \noindent Blow-up of $\widetilde{X}$ in the \\
 base point $\widetilde{Q}$ of $|-K_{\widetilde{X}}|$
 \end{tabular} \vspace{2mm}
 \\
\text{Contraction of a section}  & \reflectbox{$\longmapsto$} &  \widetilde{Y}
\end{array}
$$
which preserves the subsets of surfaces with global vector fields and identifies the identity components of the automorphism schemes of these surfaces.
\end{Corollary}

\begin{proof}
First, we observe that the map from the left-hand side to the right-hand side is well-defined, since every isomorphism between two weak del Pezzo surfaces of degree $1$ identifies the unique base points of the anti-canonical system and hence lifts to the blow-up.

Similarly, since any two sections of a Jacobian rational (quasi-)elliptic surface can be mapped to each other by a suitable translation, the isomorphism class of the weak del Pezzo surface resulting from the contraction of a section does not depend on the choice of a section. Hence the map from the right-hand side to the left-hand side is well-defined.
Finally, if $\widetilde{X}$ is the contraction of a Jacobian rational (quasi-)elliptic surface $\widetilde{Y}$, then $\Aut_{\widetilde{X}}^0 \cong \Aut_{\widetilde{Y}}^0$ by Lemmata \ref{lemma Blanchard generalities} and \ref{lemma blowup stabilizer}. 
\end{proof}

In the case of not necessarily Jacobian rational (quasi-)elliptic surfaces $\widetilde{Z}$, the  correspondence is more subtle, since contractions of two different $m$-sections might lead to non-isomorphic weak del Pezzo surfaces. Nevertheless, let us note the following corollary to Proposition \ref{prop: torsionpoint}, combined with Lemmata \ref{lemma Blanchard generalities} and \ref{lemma blowup stabilizer}: 

\begin{Corollary}[Non-Jacobian case] \label{cor: approach for classification non-jacobian}
Every rational (quasi-)elliptic surface $\widetilde{Z}$ of index $m > 1$ satisfying $h^0(\widetilde{Z},T_{\widetilde{Z}}) \neq 0$ is obtained by blowing up a weak del Pezzo surface $\widetilde{X}$ of degree $1$ with $h^0(\widetilde{X},T_{\widetilde{X}}) \neq 0$ in a point $\widetilde{P} \in \widetilde{X}$ such that the following conditions hold:
\begin{enumerate}
    \item $({\rm Stab}_{\Aut_{\widetilde{X}}^0}(\widetilde{P}))^0$ 
    is non-trivial.
    \item $\widetilde{P}$ is a point of exact order $m$ on the identity component $C^0$ of the unique curve $C \in |-K_{\widetilde{X}}|$ that contains $\widetilde{P}$.
\end{enumerate}
Moreover, we have $\Aut_{\widetilde{Z}}^0 \cong ({\rm Stab}_{\Aut_{\widetilde{X}}^0}(\widetilde{P}))^0.$
\end{Corollary}

Thus, in order to classify rational (quasi-)elliptic surfaces with global vector fields, we have to calculate the stabilizers of the action of $\Aut_{\widetilde{X}}^0$, where ${\widetilde{X}}$ is a weak del Pezzo surface of degree $1$ with global vector fields.

\subsection{Proof of Main Theorem (\hyperref[Main theorem Jacobian]{A})}
\label{subsection proof Jacobian case}
By Corollary \ref{cor: approach for classification jacobian}, Jacobian rational
(quasi-)elliptic surfaces with global vector fields are in bijection with weak del
Pezzo surfaces of degree $1$ with global vector fields, and \linebreak corresponding
surfaces have isomorphic connected automorphism schemes. Hence, degree $1$ part
of the classification in
\cite[Table 6]{WeakDelPezzoGlobalVectorFields} consists, for $p\neq2$,
precisely of the types $1A$--$1D$ and the additional
characteristic $3$ types $1E$--$1I$, hence the cases in Table \ref{Table main theorem Jacobian}.

It remains only to translate their configurations of $(-2)$-curves into Kodaira--N\'eron fiber types according to Subsection \ref{subsection reducible fibers}. Since blowing up the base point $\widetilde{Q}$ adds the affine component to each
anti-canonical configuration, i.e.,
$2D_4, E_6+A_2, E_7+A_1, E_8, D_7, E_7, A_8$
produce, respectively,
$
2{\rm I}_0^*,
{\rm IV}^*+{\rm IV},
{\rm III}^*+{\rm III},
{\rm II}^*,
{\rm I}_3^*,
{\rm III}^*,
{\rm I}_9$ as reducible fibers.

\hfill $\blacksquare$

\section{Strategy and proof of Main Theorem (\hyperref[Main theorem non-Jacobian]{B})}
\label{section strategy proof main theorem}

It remains to prove part (\hyperref[Main theorem non-Jacobian]{B}) of the Main
Theorem. Corollary \ref{cor: approach for classification non-jacobian}
reduces this assertion to the weak del Pezzo surfaces of degree $1$ with global
vector fields. By
\cite[Table 6]{WeakDelPezzoGlobalVectorFields}, the relevant types for
$p\neq2$ are $1A$--$1D$, together with the types $1E$--$1I$ that occur
only in characteristic $3$. We will follow the notation of types in \cite{WeakDelPezzoGlobalVectorFields} throughout.

\begin{Notation}
    \label{notation admissible point}
    Let $\widetilde{P} \in \widetilde{X}$ be as in Notation \ref{Notation all objects Ztilde Xtilde Ytilde, ... points}, i.e. a point of exact order $m>1$ on the identity component $C^0$ of the unique $C \in |-K_{\widetilde{X}}|$ for $\widetilde{X}$ a weak del Pezzo surface of degree $1$. 
    We call $\widetilde{P}$ (resp. its image $P$ on the anti-canonical model $X$) \emph{admissible} if in addition $({\rm Stab}_{\Aut_{\widetilde{X}}^0}(\widetilde{P}))^0 \neq \{ {\rm id} \}$ (resp. $({\rm Stab}_{\Aut_{\widetilde{X}}^0}(P))^0 \neq \{ {\rm id} \}$).
\end{Notation}

\begin{Strategy}
\label{strategy of proof non-Jaco}
We follow Notation \ref{Notation all objects Ztilde Xtilde Ytilde, ... points} and \ref{notation admissible point}.
For each $\widetilde X$ with non-trivial $\Aut_{\widetilde{X}}^0$ of type $1A$--$1I$, we
proceed along the following steps.
\begin{enumerate}
\item \label{strategy equation and action} Determine a Weierstra{\ss} equation for $X\subseteq\mathbb P(1,1,2,3)$ and the action of the subgroup
scheme $\Aut_{\widetilde X}^0\subseteq\Aut_X^0$ on it.
\item \label{strategy admissible points mit stabi} Determine all admissible $\widetilde{P} \in \widetilde{X}$ by computing all admissible $P \in X$. Keep the precise forms of the group schemes $({\rm Stab}_{\Aut_{\widetilde{X}}^0}(P))^0$.
\item \label{strategy moduli} Divide the admissible centers $\widetilde{P}$ into orbits under the full automorphism group $\Aut(\widetilde{X})$ in order to determine the moduli of surfaces ${\rm Bl}_{\widetilde{P}}(\widetilde{X})$. 
\item \label{strategy reducible fibers and identification} Determine the reducible fibers on
$\widetilde Z={\rm Bl}_{\widetilde P}(\widetilde X)$ using the
Kodaira--N\'eron classification, the rank bound in Lemma
\ref{lemma (-2)curves on Ztilde}, smoothness of the fixed locus of linearly reductive group scheme actions \cite[Proposition A.8.10(2)]{ConradGabberPrasad}, and, where useful, the contraction of a
different $(-1)$-curve. If fiber types and $\Aut_{\widetilde{Z}}^0$ coincide, check whether contracting a different $(-1)$-curve
identifies $\widetilde{Z}$ with a ${\rm Bl}_{\widetilde{P}'}(\widetilde{X}')$ for $\widetilde{P}' \in \widetilde{X}'$ admissible on another weak del Pezzo surface of degree $1$ of types $1A$--$1I$. 
\end{enumerate}
\end{Strategy}

 \begin{proof}[Proof of validity of Strategy~\ref{strategy of proof non-Jaco}]
Let $\widetilde Z$ be a non-Jacobian rational (quasi-)elliptic surface with
$h^0(\widetilde Z,T_{\widetilde Z})\neq0$, equivalently with
$\Aut_{\widetilde Z}^0\neq\{1\}$. By Corollary~\ref{cor: approach for classification non-jacobian},
$\widetilde Z={\rm Bl}_{\widetilde P}(\widetilde X)$ for a weak del Pezzo surface $\widetilde X$
of degree $1$ with $\Aut_{\widetilde X}^0\neq\{1\}$, where $\widetilde P\in C^0$ has exact
order $m>1$ and $\Aut_{\widetilde Z}^0\cong(\Stab_{\Aut_{\widetilde X}^0}(\widetilde P))^0$.
By \cite[Table 6]{WeakDelPezzoGlobalVectorFields}, the possibilities for $\widetilde X$ in
odd characteristic are exactly the types \hyperref[Tab1A]{$1A$}--\hyperref[Tab1D]{$1D$},
together with \hyperref[Tab1E]{$1E$}--\hyperref[Tab1I]{$1I$} in characteristic $3$. Hence
the strategy runs over every candidate.

Let $\pi\colon\widetilde X\longrightarrow X$ be the anti-canonical contraction and
$P=\pi(\widetilde P)$. Since $\widetilde P$ has exact order $m>1$ it differs from the neutral
element $\widetilde Q$, and as $\widetilde P\in C^0$ it lies neither on a $(-2)$-curve
contracted by $\pi$ nor at a singular point of $C$. Thus $P$ is neither the image $Q$ of the
anti-canonical base point nor a singular point of $X$, so $\pi$ is an isomorphism in a
neighborhood of $\widetilde P$. Being $\Aut_{\widetilde X}^0$-equivariant, $\pi$ identifies
the stabilizers,
\[
\Aut_{\widetilde Z}^0
\cong
\left(\Stab_{\Aut_{\widetilde X}^0}(\widetilde P)\right)^0
\cong
\left(\Stab_{\Aut_{\widetilde X}^0}(P)\right)^0,
\]
where in the last term $\Aut_{\widetilde X}^0$ is viewed as the liftable subgroup scheme of
$\Aut_X^0$. So the explicit equations and liftable actions of Step
(\ref{strategy equation and action}) let us compute all admissible centers and their
connected stabilizers directly on $X$; conversely, the Halphen construction recalled after
Proposition~\ref{prop: torsionpoint} attaches to each such torsion point its non-Jacobian
rational (quasi-)elliptic surface, and the displayed isomorphism decides whether it carries
global vector fields.

It remains to account for repetitions. For fixed $\widetilde X$, points in the same
$\Aut(\widetilde X)$-orbit give isomorphic blow-ups, and an isomorphism of two blow-ups
fixing the distinguished exceptional curve descends to the pointed weak del Pezzo surfaces,
so it is detected in Step~\ref{strategy admissible points mit stabi} and the subsequent orbit
calculation. If instead the exceptional curve is sent to a different $(-1)$-curve, contracting
the latter yields another weak del Pezzo surface of degree $1$, again with non-trivial
connected automorphism scheme by Blanchard's Lemma and hence again of one of the types
$1A$--$1I$. These are precisely the alternative presentations checked in the final step of the
strategy. Hence the strategy produces all non-Jacobian rational (quasi-)elliptic surfaces with
global vector fields and accounts for all repetitions among them.
\end{proof}

\noindent
\textbf{Structure of the proof:}
The proof follows the above strategy and employs Notations \ref{Notation all objects Ztilde Xtilde Ytilde, ... points} and \ref{notation admissible point}.
\begin{itemize}
\item
    Step (\ref{strategy equation and action}) will be done in Subsection \ref{subsection equations and liftable actions} (see Proposition \ref{prop: char3 equations and liftable actions}).
    \item Each of the nine types will be studied separately in Subsections \ref{subsection 1A}, \ref{subsection 1B}, \ref{subsection 1C}, \ref{subsection 1D}, \ref{subsection 1E char3 revised}, \ref{subsection 1F char3 revised}, \ref{subsection 1G char3 revised}, \ref{subsection 1H char3 revised}, and \ref{subsection 1I char3 revised}.
\item
    Step (\ref{strategy admissible points mit stabi}) resp. Step (\ref{strategy moduli}) is carried out in the first Proposition and following Corollary in each of those subsections. 
    \item Step (\ref{strategy reducible fibers and identification}) is then discussed and summarized in a last Corollary in each of those subsections.
    \item In all configuration figures for the types $1A$--$1I$, thick curves denote
$(-2)$-curves and thin curves denote $(-1)$-curves. The anti-canonical base
point $\widetilde{Q}$, $Q$, the zero-section, and their images are gray. The admissible point
$\widetilde P$, its exceptional curve, and their images are red. 
 Unless stated otherwise, a five-term diagram of figures is arranged as follows:
$$
\xymatrix{
  \widetilde{Z} \ar[r]
  & \widetilde{X} \ar[d]
  & \widetilde{Y} \ar[l] \ar[d] 
  \\
  & X
  & Y \ar[l]
} 
$$
On $\widetilde X$, the figures display all
negative curves; on $\widetilde Y$, they display all fiber components and the
relevant visible sections. On $\widetilde Z$, they are complete for the
$(-2)$-curves, equivalently for the reducible fibers, but not necessarily for
all $(-1)$-curves. Intersection multiplicities $1$ and $2$ are read from the
drawing, while larger intersection multiplicities are written next to the
point.
\end{itemize}

\subsection{Equations and liftable actions on anti-canonical models}
\label{subsection equations and liftable actions}
\label{subsection char3 equations and actions}

For cases \hyperref[Tab1A]{$1A$}--\hyperref[Tab1D]{$1D$} one has
$\Aut_{\widetilde X}^0\cong\mathbb G_m $, which is smooth; hence, since smooth group scheme actions lift to the minimal resolution, the faithful
$\mathbb G_m$-actions on respective $X$ recorded in the below Table \ref{Table four families} automatically lift to $\widetilde X$
and coincide with the respective $\Aut_{\widetilde X}^0$-action by comparison with
\cite[Table 6]{WeakDelPezzoGlobalVectorFields} and Blanchard's Lemma. The
equations are the characteristic-free simplifications of those in
\cite{ExtremalChar0,ExtremalCharpII,QuasiEllipticChar3,QuasiEllipticChar2, WhichRDPsOccurOnDelPezzoSurfaces}.
Those equations, actions and configurations of negative curves are summarized in the following Table \ref{Table four families}.

\begin{table}[H] \renewcommand{\arraystretch}{1.35}
\begin{adjustbox}{center}
 $
 \begin{array}{|c|c|c|c|c|}
 \hline
 \text{Case}
& \begin{tabular}{c}
$(-2)$-curves \\
on $\widetilde{X}$
\end{tabular}
& \begin{tabular}{c}
Negative curves on $\widetilde{X}$
\end{tabular}
& \begin{array}{c}
\text{Action of } \Aut_{\widetilde{X}}^0 \text{ on} \\
\text{the Weierstra{{\ss}} equation of $X$}
\end{array}
& \Char (k)
\\
\hline \hline
{1A} \label{Tab1A}
& 2D_4 
&  \begin{array}{c}
\addstackgap[7pt]{{\includegraphics[width=0.29\textwidth]{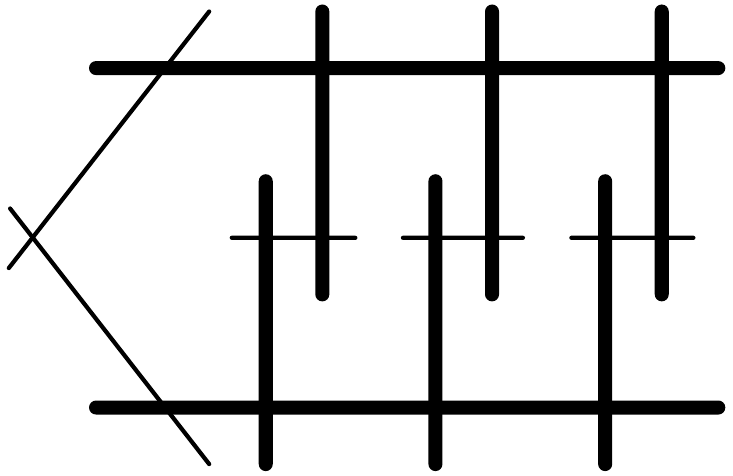}}} 
\end{array} 
& \begin{array}{c}
\mathbb{G}_m: 
{[\lambda s:  \lambda^{-1} t: x:y]} \\
\text{acting on}\\
y^2 = x^3 + a st x^2 + s^2t^2x \\
\text{ where } a \in k \text{ and } a^2 \neq 4 
\end{array}
& \text{any}
\\
\hline
{1B} \label{Tab1B}
& E_6 + A_2 
&  \begin{array}{c}
\addstackgap[7pt]{{\includegraphics[width=0.22\textwidth]{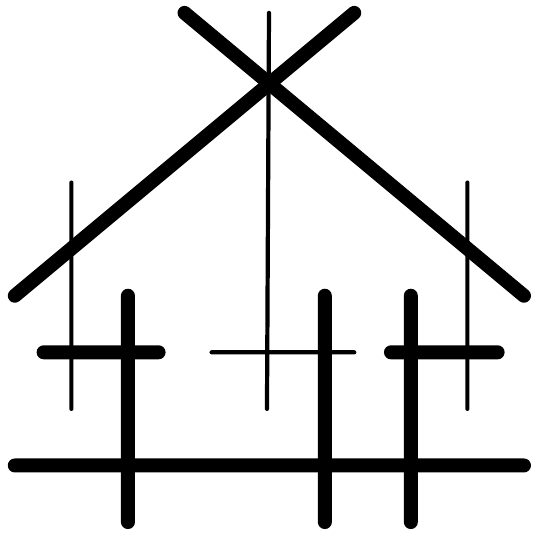}}}
\end{array} 
& \begin{array}{c}
\mathbb{G}_m: 
{[\lambda^2 s:  \lambda^{-1} t: x:y]} \\
\text{acting on}\\
y^2 +st^2y = x^3  
\end{array}
& \text{any}
\\
\hline
{1C} \label{Tab1C}
& E_7 + A_1
& \begin{array}{c}
\addstackgap[7pt]{{\includegraphics[width=0.22\textwidth]{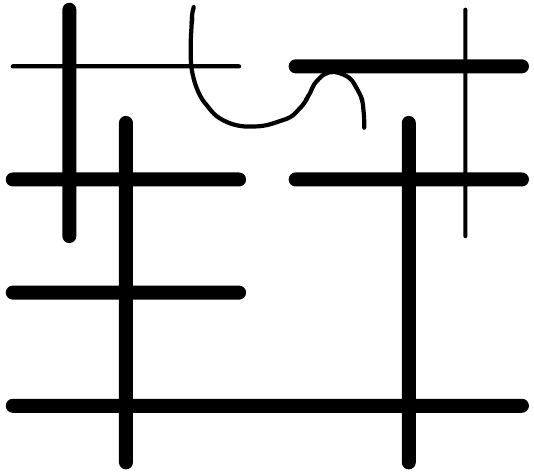}}}
\end{array} 
& \begin{array}{c}
\mathbb{G}_m: 
{[\lambda^3 s:  \lambda^{-1} t: x:y]} \\
\text{acting on}\\
y^2 = x^3 + st^3x 
\end{array}
& \neq 2
\\
\hline
{1D} \label{Tab1D}
& E_8
& \begin{array}{c}
\addstackgap[7pt]{{\includegraphics[width=0.36\textwidth]{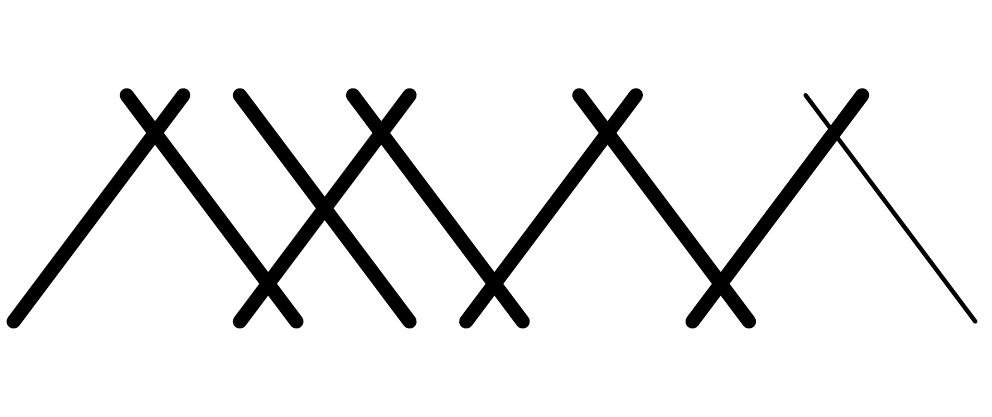}}}
\end{array} 
& \begin{array}{c}
\mathbb{G}_m: 
{[\lambda^5 s:  \lambda^{-1} t: x:y]} \\
\text{acting on} \\
y^2= x^3 + st^5 
\end{array}
& \neq 2,3
\\
\hline
\end{array}
$
\caption{Equations for $X$ and $\Aut_{\widetilde{X}}^0$-actions for cases $1A$--$1D$} \label{Table four families}
\end{adjustbox}
\end{table}

In characteristic $3$, the connected automorphism scheme $\Aut_X^0$ can be
non-reduced and the closed immersion
$\Aut_{\widetilde X}^0\hookrightarrow\Aut_X^0$ induced by the minimal
resolution may be strict, as studied in
\cite{RDPDelPezzoGlobalVectorFields}. Hence determining $\Aut_X^0$ (as done in \cite{RDPDelPezzoGlobalVectorFields}) is not
enough: one must isolate the subgroup scheme whose action lifts to
$\widetilde X$. The next proposition records the equations and liftable actions
for the additional cases \hyperref[Tab1E]{$1E$}--\hyperref[Tab1I]{$1I$}.

\begin{Proposition} \label{prop: char3 equations and liftable actions}
Let $\widetilde X$ be a weak del Pezzo surface of degree $1$ with global vector of type \hyperref[Tab1E]{$1E$}, \hyperref[Tab1F]{$1F$}, \hyperref[Tab1G]{$1G$}, \hyperref[Tab1H]{$1H$}, or \hyperref[Tab1I]{$1I$}. Then the configuration of negative curves on $\widetilde X$, the Weierstra{\ss} equation for its anti-canonical model $X$, and the action of $\Aut_{\widetilde X}^0$ on $X$ are as given in Table \ref{Table char3 equations and liftable actions}.
\end{Proposition}

\begin{table}[H] \renewcommand{\arraystretch}{1.35}
\begin{adjustbox}{center}
$
\begin{array}{|c|c|c|c|}
\hline
\text{Case}
& \begin{tabular}{c}
$(-2)$-curves \\
on $\widetilde{X}$
\end{tabular}
& \begin{tabular}{c}
Negative curves on $\widetilde{X}$
\end{tabular}
& \begin{array}{c}
\text{Action of } \Aut_{\widetilde{X}}^0 \text{ on} \\
\text{the Weierstra{{\ss}} equation of $X$}
\end{array}
\\
\hline \hline
{1E} \label{Tab1E}
& D_7
& \begin{array}{c}
\addstackgap[7pt]{{\includegraphics[width=0.31\textwidth]{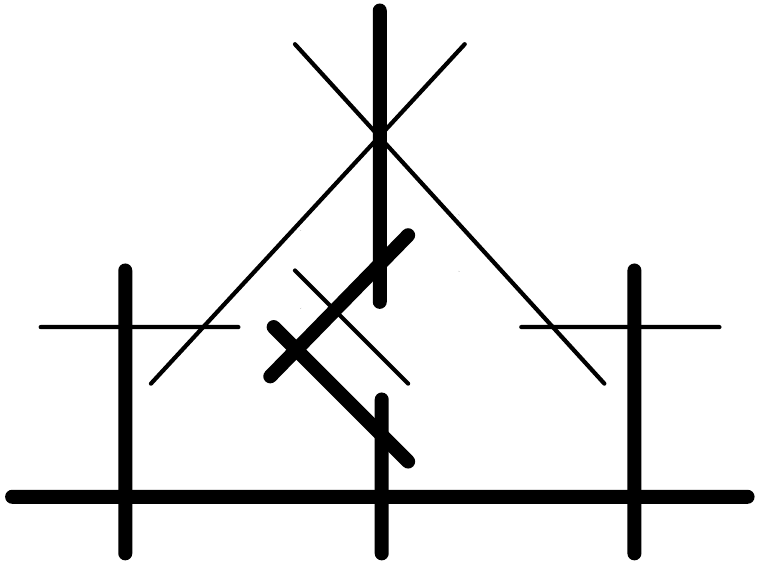}}}
\end{array}
& \begin{array}{c}
\mu_3:
 [s:\lambda t:\lambda x:y]\\
\text{acting on}\\
y^2=x^3+stx^2+t^6
\end{array}
\\
\hline
{1F} \label{Tab1F}
& E_7
& \begin{array}{c}
\addstackgap[7pt]{{\includegraphics[width=0.30\textwidth]{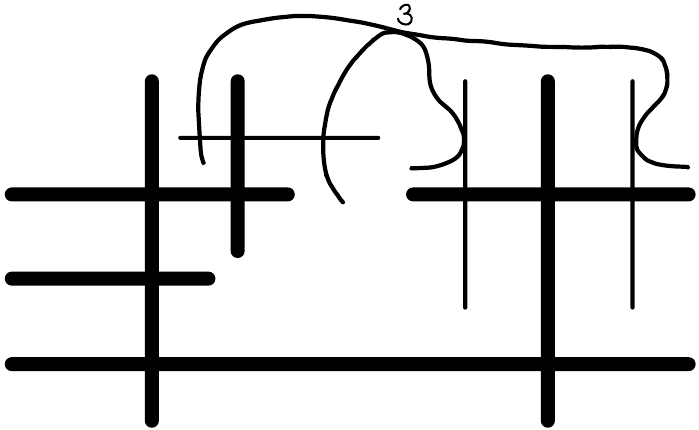}}}
\end{array}
& \begin{array}{c}
\mu_3:
 [s:\lambda t:x:y]\\
\text{acting on}\\
y^2=x^3+st^3x+t^6
\end{array}
\\
\hline
{1G} \label{Tab1G}
& A_8
& \begin{array}{c}
\addstackgap[7pt]{{\includegraphics[width=0.23\textwidth]{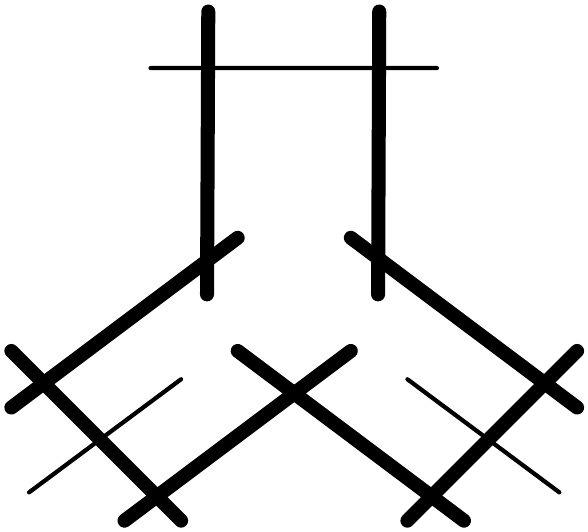}}}
\end{array}
& \begin{array}{c}
\mu_3:
 [s:\lambda t:x:y]\\
\text{acting on}\\
y^2=x^3+s^2x^2+st^3x+t^6
\end{array}
\\
\hline
{1H} \label{Tab1H}
& E_8
& \begin{array}{c}
\addstackgap[7pt]{{\includegraphics[width=0.36\textwidth]{1D1H1I1S1T=aaaaccaa.pdf}}}
\end{array}
& \begin{array}{c}
\mathbb G_a:
[s:t-(a^3+a)s:x+a^3s^2:y]\\
\text{acting on}\\
y^2=x^3+s^4x+s^3t^3
\end{array}
\\
\hline
{1I} \label{Tab1I}
& E_8
& \begin{array}{c}
\addstackgap[7pt]{{\includegraphics[width=0.36\textwidth]{1D1H1I1S1T=aaaaccaa.pdf}}}
\end{array}
& \begin{array}{c}
\mathbb G_a:
 [s:t-a^3s:x+as^2:y]\\
\mathbb G_m:
[\lambda s:\lambda^{-5}t:x:y]\\
\text{acting on}\\
y^2=x^3+s^5t
\end{array}
\\
\hline
\end{array}
$
\caption{Equations for $X$ and $\Aut_{\widetilde{X}}^0$-actions for cases $1E$--$1I$ ($\Char(k)=p=3$)} \label{Table char3 equations and liftable actions}
\end{adjustbox}
\end{table}

\begin{proof}
Recall that these cases only exist in characteristic $3$. The first three columns of Table \ref{Table char3 equations and liftable actions} follow from the classification of weak del Pezzo surfaces with global vector fields in \cite[Table 6]{WeakDelPezzoGlobalVectorFields}. Let us now focus on identifying the correct Weierstra{\ss} equations for the anti-canonical models $X\subset \bbP(1,1,2,3)$ and the induced $\Aut_{\widetilde X}^0$-actions on $X$.

By Blanchard's Lemma, applied to the minimal resolution $\widetilde X\to X$, there is a closed immersion $\Aut_{\widetilde X}^0\hookrightarrow \Aut_X^0$. This immersion may be strict, as studied in \cite{RDPDelPezzoGlobalVectorFields}. Thus, when working on $X$, we have to single out the subgroup scheme of $\Aut_X^0$ that actually comes from the minimal resolution $\widetilde X$. We also note that, by the classification in \cite[Table 6]{WeakDelPezzoGlobalVectorFields}, each of the weak del Pezzo surfaces of types \hyperref[Tab1E]{$1E$}--\hyperref[Tab1I]{$1I$} is unique up to isomorphism, and so is its anti-canonical model. Therefore, once we find a Weierstra{\ss} equation with the prescribed singularity type, together with a connected group scheme action that lifts to the minimal resolution and agrees with the group scheme in \cite[Table 6]{WeakDelPezzoGlobalVectorFields}, this identifies the corresponding case.

In the cases \hyperref[Tab1F]{$1F$}, \hyperref[Tab1G]{$1G$}, \hyperref[Tab1H]{$1H$}, and \hyperref[Tab1I]{$1I$}, candidates for the equations of $X$ are contained in \cite[Table 10]{RDPDelPezzoGlobalVectorFields}, together with descriptions of $\Aut_X^0$.
 In cases \hyperref[Tab1H]{$1H$} resp. \hyperref[Tab1I]{$1I$}, Table \ref{Table char3 equations and liftable actions} shows actions by $\mathbb G_a$ resp. $\mathbb G_a\rtimes \mathbb G_m$, which are in particular smooth group schemes and hence these actions lift to the minimal resolution $\widetilde X$ and a comparison with \cite[Table 6]{WeakDelPezzoGlobalVectorFields} shows that these are the desired $\Aut_{\widetilde X}^0$-actions on $X$. \footnote{Note that in case \hyperref[Tab1H]{$1H$}, a typo in the $\mathbb{G}_a$-action in the last line of Table 10 in \cite{RDPDelPezzoGlobalVectorFields} was corrected.}
Let us now identify the liftable subgroup schemes in cases \hyperref[Tab1F]{$1F$} and \hyperref[Tab1G]{$1G$}: 

In case \hyperref[Tab1F]{$1F$}, \cite[Table 10]{RDPDelPezzoGlobalVectorFields} gives $\Aut_X^0=\alpha_3\rtimes \mu_3$, acting by $[s:t:x:y]\mapsto [s:\epsilon s+\lambda t:x:y]$ with $\epsilon^3=0$ and $\lambda^3=1$. The $E_7^0$-singularity is at $P=[1:0:0:0]$, and we see that the $\alpha_3$-action does not fix $P$. Therefore every liftable subgroup scheme is contained in $\Stab_{\Aut_X^0}(P)^0=\mu_3$. Since \cite[Table 6]{WeakDelPezzoGlobalVectorFields} gives $\Aut_{\widetilde X}^0\cong \mu_3$ for type \hyperref[Tab1F]{$1F$}, the closed immersion $\Aut_{\widetilde X}^0\hookrightarrow \Stab_{\Aut_X^0}(P)^0=\mu_3$ is an isomorphism. This gives the action displayed in Table \ref{Table char3 equations and liftable actions}.

The same argument applies in case \hyperref[Tab1G]{$1G$}. Here, $\Aut_X^0=\alpha_9\rtimes \mu_3$ acts by $[s:t:x:y]\mapsto [s:\epsilon s+\lambda t:x+\epsilon^3s^2:y]$ where $\epsilon^9=0$, $\lambda^3=1$, and the $A_8$-singularity at $P=[1:0:0:0]$. We see that the stabilizer subgroup scheme of $P$ is $\mu_3$, and hence the closed immersion $\Aut_{\widetilde X}^0\hookrightarrow \Stab_{\Aut_X^0}(P)^0=\mu_3$ is an isomorphism by comparison with \cite[Table 6]{WeakDelPezzoGlobalVectorFields}. Thus the liftable action is precisely the $\mu_3$-action shown in Table \ref{Table char3 equations and liftable actions}.

It remains to treat case \hyperref[Tab1E]{$1E$}: This case does not appear in \cite[Table 10]{RDPDelPezzoGlobalVectorFields}, because here no additional vector fields occur on the anti-canonical model which fail to lift to the minimal resolution. The equation $y^2=x^3+stx^2+t^6$ is obtained by applying Tate's algorithm \cite{TateAlgo} (see also \cite[Subsection 4.1]{WhichRDPsOccurOnDelPezzoSurfaces}) to a Weierstra{\ss} equation whose associated Jacobian surface has a fiber of type $I_3^*$; this corresponds to a $D_7$-singularity on the anti-canonical model $X$, by the usual correspondence between Kodaira fiber types and rational double points recalled in Table \ref{Table KodairaNeron2nd and RDPs mit tilde}. For the above candidate equation, the discriminant is $-s^3t^9$, the $I_3^*$-fiber lies over $t=0$ and, also by Tate's algorithm we see that the fiber over $s=0$ is irreducible of type $II$. Hence the $X$ given by the above equation has only one singularity at $[1:0:0:0]$, which is of type $D_7$. The displayed $\mu_3$-action $[s:t:x:y]\mapsto [s:\lambda t:\lambda x:y]$, where $\lambda^3=1$, preserves the equation and fixes the $D_7$-singularity. Since $D_7$ is equivariant in characteristic $3$, this action lifts to $\widetilde X$ (compare \cite[Table 10]{RDPDelPezzoGlobalVectorFields}). Finally, by the uniqueness of the weak del Pezzo surface of type \hyperref[Tab1E]{$1E$}, this is the required action of $\Aut_{\widetilde X}^0$ on $X$, and finishes the proof of all entries of Table \ref{Table char3 equations and liftable actions}.
\end{proof}

\begin{Remark}\label{remark characteristic two deferred}
The geometric reduction and the stabilizer strategy work verbatim in
characteristic $2$. What is missing there (for the remaining characteristic $2$ types $1J$--$1T$) is a uniform list of convenient
Weierstra{\ss} equations for the anti-canonical models together with the
subgroup schemes of $\Aut_X^0$ that lift to the minimal resolutions (in
\cite{RDPDelPezzoGlobalVectorFields} these were determined only in odd
characteristic). The characteristic $2$ case will be treated in an upcoming article.
\end{Remark}

The following nine subsections will conclude the proof of Main Theorem (\hyperref[Main theorem non-Jacobian]{B}), following the structure as described above and in Strategy \ref{strategy of proof non-Jaco}.

\pagebreak

\subsection{Case \hyperref[Tab1A]{$1A$}} \label{subsection 1A}

These surfaces occur for arbitrary $\Char(k)=p \geq 0$.

\begin{Proposition} \label{prop: 1A}
Let $\widetilde{X}$ be of type \hyperref[Tab1A]{$1A$}. 
\begin{enumerate}
    \item[(0)] If $p \neq 2$, then there are no admissible $\widetilde{P} \in \widetilde{X}$.
    \item\label{prop case 1A p=2 II} If $p = 2$, then $\widetilde{P}$ is admissible if and only if $C$ is of type ${\rm II}$ 
    {and $\widetilde{P}$ lies on a $(-1)$-curve}.
    Moreover, then $({\rm Stab}_{\Aut_{\widetilde{X}}^0}(\widetilde{P}))^0 \cong \mu_2$ and $m=2$.
\end{enumerate}
\end{Proposition}

\begin{proof}
$X$ is given by
$y^2= x^3 + astx^2 + s^2t^2x$ with $a \in k$ and $a^2 \neq 4$
 and $\Aut_{\widetilde{X}}^0 \cong \mathbb{G}_m$ acts as $[s:t:x:y] \mapsto [\lambda s:\lambda^{-1}t:x:y]$ (see Table \ref{Table four families}). We note that the two $D_4$-singularities are at $[1:0:0:0]$ and $[0:1:0:0]$. For finding admissible $P \in X$ (according to Strategy \ref{strategy admissible points mit stabi}), we distinguish the following cases:
\begin{enumerate}[leftmargin=0.8cm]

\item[(a)] If $s=0$, we can assume $t=1$. For the action $[0:1:x:y] \mapsto [0:1: \lambda^2 x : \lambda^3y]$ to fix $P$, we must have either $x=y=0$, in which case $P$ would be a $D_4$-singularity, or $x,y \neq 0$ and $\lambda =1$, in which case $\Stab_{\mathbb{G}_m}(P)$ is trivial.

\item[(b)] Thus, we can assume $s=1$. Exploiting the symmetry between $s$ and $t$, we can assume $t \neq 0$ by (a).

\begin{enumerate}
\item[(1)] \label{1A cases s=1, t neq 0, char 2}
For the action $[1:t:x:y] \mapsto [1: \lambda^{-2}t: \lambda^{-2}x: \lambda^{-3}y]$ to fix $P$, we immediately see that $\lambda^2=1$ must hold. 
If furthermore $y$ was non-zero, this would imply $\lambda= \lambda^3 =1$. Thus, we can assume $y=0$. 
Then, $(\Stab_{\mathbb{G}_m}(P))^0$ is non-trivial if and only if $p=2$ and 
$$
P=[1:t:x:0] \hspace{2mm} \text{with}  \hspace{2mm} t \neq 0 \hspace{2mm}\text{and}\hspace{2mm} x^3+atx^2+t^2x=0.
$$
We note that 
$ x^3+atx^2+t^2x = x (x+bt)(x+(a+b)t)$ for $b \in k$ a solution of $z^2+az+1=0$, and thus $$P \in \{ [1:t:0:0], [1:t:bt:0], [1:t:(a+b)t]\}.$$
For such points $P$, we have $(\Stab_{\mathbb{G}_m}(P))^0 \cong \mu_2$. 
Moreover, from the location of the singular points of $X$, we see that ${\widetilde{P}}$ lies in an irreducible fiber $C$ of $\widetilde{X} \to X \dashrightarrow \mathbb{P}^1$.
Since $p=2$, our equation for $X$ is the Weierstra{\ss} equation of a quasi-elliptic fibration \cite[Theorem 5.2.(d)]{QuasiEllipticChar2}, hence $C$ is of type ${\rm II}$, $C^0 \cong \mathbb{G}_a$ and $m=2$ by Corollary \ref{cor: approach for classification non-jacobian}.
{Finally, note that the equations $x = 0,x+bt = 0,$ and $x+(a+b)t = 0$ are exactly the equations of the $(-1)$-curves on $\widetilde{X}$ that are not contained in members of $|-K_{\widetilde{X}}|$.} 
\end{enumerate}
\end{enumerate}
\vspace{-5mm} \end{proof}

\begin{Corollary} \label{cor: 1A non-Jac}
Let $\widetilde{Z}$ be arising from an $\widetilde{X}$ of type \hyperref[Tab1A]{$1A$} and assume that $h^0(\widetilde{Z},T_{\widetilde{Z}}) \neq 0$. Then, 
$p=2$, such $\widetilde{Z}$ form a $1$-dimensional family, each of them has one multiple fiber $2 {\rm II}$, and $\Aut_{\widetilde{Z}}^0 \cong \mu_2$. 
\end{Corollary}

\begin{proof}
Everything except the number of moduli follows by combining Corollary \ref{cor: approach for classification non-jacobian} with Proposition \ref{prop: 1A}.

     To see that these surfaces form a $1$-dimensional family, note that weak del Pezzo surfaces of type \hyperref[Tab1A]{$1A$} form a $1$-dimensional family, so it suffices to show that for every fixed $\widetilde{X}$ of type \hyperref[Tab1A]{$1A$}, the choice of $\widetilde{P}$ is unique up to automorphisms of the surface. For this, firstly, we observe that all the curves $C \in |-K_{\widetilde{X}}|$ of type ${\rm II}$ are conjugate under $\Aut(\widetilde{X})$, and, secondly, that in every such fiber $C$, the three points $\widetilde{P}$, whose blow-up yields $\widetilde{Z}$ are permuted by an $S_3$-action on $X$. Both follow from our description of the $\mathbb{G}_m$-action on $X$ in Table \ref{Table four families} and the proof of Proposition \ref{prop: 1A}(\ref{prop case 1A p=2 II}): First, $\mathbb{G}_m$ sends a fiber $\{[1:t:x:y] \hspace{1mm}| \hspace{1mm}y^2=x^3 + atx^2 + t^2x \}$ over $[1:t], t \neq 0$, to the fiber over $[1: \lambda^{-2}t]$, hence all such fibers are conjugate under $\mathbb{G}_m$. Second, for fixed $t \neq 0$, the involutions $x \mapsto x+bst$ (resp. $x \mapsto x+(a+b)st$) of $X$ interchange $[1:t:0:0]$ and $[1:t:bt:0]$ (resp. $[1:t: (a+b)t:0]$).
\end{proof}

\begin{Discussion} \label{Discussion: 1A geometry and curve configs}
Note that, in the explicit description of the possibly blown up points $P \in X$ in the proof of Proposition \ref{prop: 1A}(\ref{prop case 1A p=2 II}) and their identification via automorphisms of $\widetilde{X}$ in Corollary \ref{cor: 1A non-Jac}, we see the structure of the Mordell--Weil group of the Jacobian rational quasi-elliptic fibration $\widetilde{Y} \to \mathbb{P}^1$ associated to $\widetilde{X}$: By \cite{OguisoShioda} the Mordell--Weil group is ${\rm MW}(\widetilde{Y} \to \mathbb{P}^1) \cong (\mathbb{Z}/2 \mathbb{Z})^2$, and, since $Y \to X$ is the contraction of the zero-section, the three sections different from the zero-section are visible in the equation of $X$; namely as $X \cap \{x=0\} $, $X \cap \{x =bst\}$ and $X \cap \{x= (a+b)st\}$. The involutions $x \mapsto x+bst$ and $x \mapsto x+(a+b)st$ generate ${\rm Aut}((\mathbb{Z}/2 \mathbb{Z})^2) =S_3$ and permute these sections resp. the three possibilities for $P$ on each fiber over $[1:t], t\neq 0$. 
The strict transforms of these sections in $\widetilde{X}$ are the three $(-1)$-curves intersecting only $(-2)$-curves.
Thus, $\widetilde{Z}$ contains a configuration of nine $(-2)$-curves of Kodaira--N\'eron type ${\rm I}_4^*$. By Lemma \ref{lemma (-2)curves on Ztilde}, $\widetilde{Z}$ cannot contain any further $(-2)$-curves. This situation is summarized in Figure \ref{figure 1A configs jaco, non-jaco} and Corollary \ref{cor: 1A (-2)configs}.

\begin{table}[H]
\begin{adjustbox}{center}
$
\begin{array}{ccccc}
 \begin{array}{c} \addstackgap[2pt]{\includegraphics[width=0.22\textwidth]{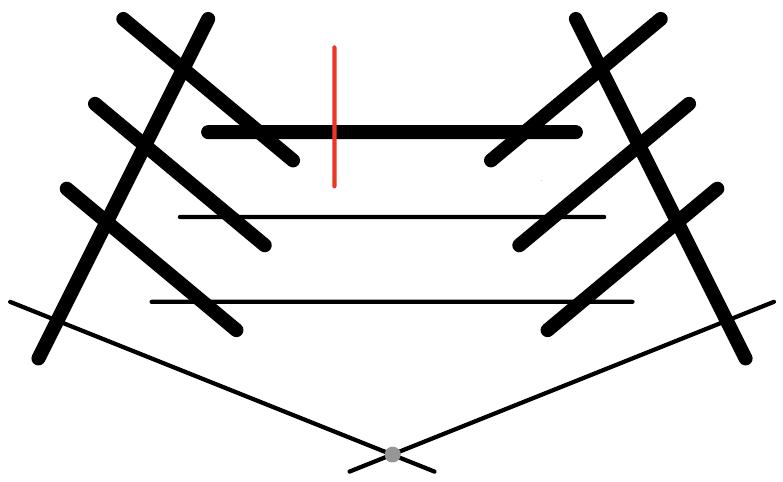} }\end{array}
 & \rightarrow
  & \begin{array}{c}\addstackgap[2pt]{\includegraphics[width=0.22\textwidth]{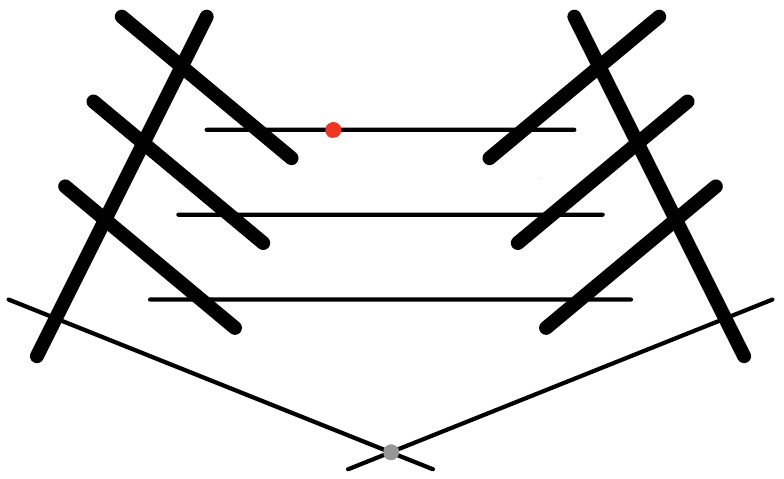} }\end{array}
  & \leftarrow
  & \begin{array}{c}\addstackgap[2pt]{\includegraphics[width=0.22\textwidth]{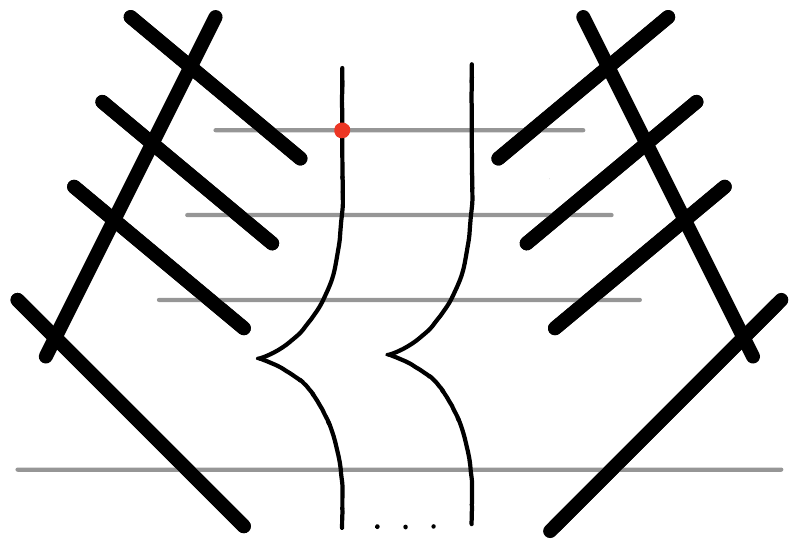} }\end{array}
  \\
  &
  & \downarrow
  &
  & \downarrow
  \\
  &
  & \begin{array}{c}\addstackgap[2pt]{\includegraphics[width=0.22\textwidth]{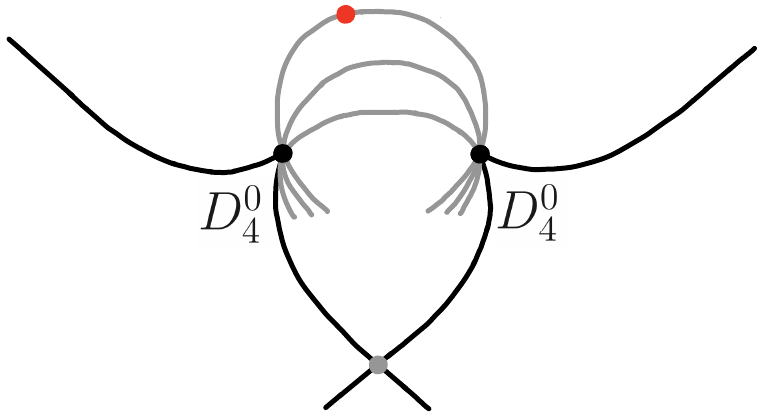}} \end{array}
  & \leftarrow
  & \begin{array}{c}\addstackgap[2pt]{ \includegraphics[width=0.22\textwidth]{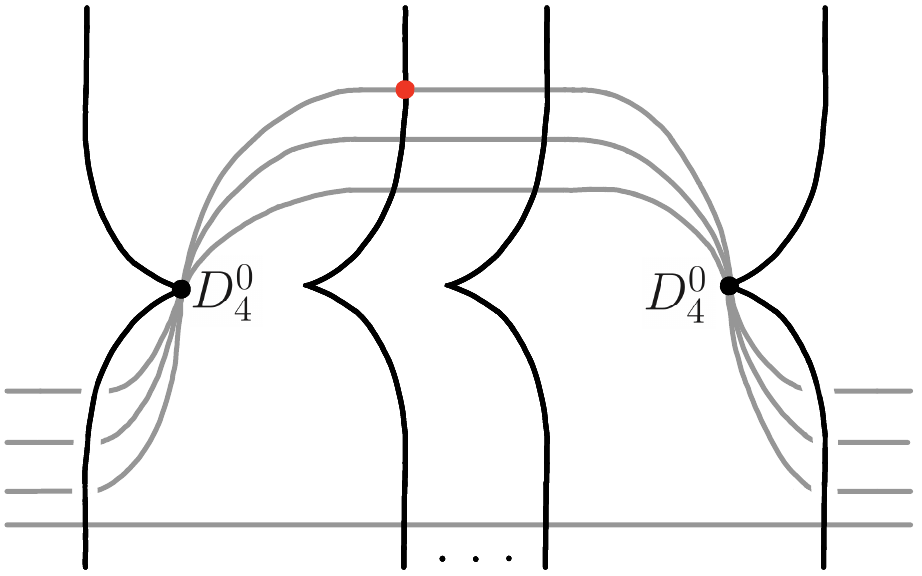}}  \end{array}
\end{array}
$
\captionof{figure}{Non-Jacobian and Jacobian fibrations with global vector fields originating from $\widetilde{X}$ of type \hyperref[Tab1A]{$1A$} ($p=2$)} \label{figure 1A configs jaco, non-jaco}
\end{adjustbox}
\end{table}

\end{Discussion}

\begin{Corollary} \label{cor: 1A (-2)configs}
    Each of the $\widetilde{Z}$ of Corollary \ref{cor: 1A non-Jac} contains nine $(-2)$-curves with dual graph of type $\widetilde{D}_8$ forming configuration ${\rm I}_4^*$.
\end{Corollary}

\subsection{Case \hyperref[Tab1B]{$1B$}} \label{subsection 1B}

Again, these surfaces occur for arbitrary $\Char(k)=p \geq 0$.

\begin{Proposition} \label{prop: 1B}
Let $\widetilde{X}$ be of type \hyperref[Tab1B]{$1B$}.
\begin{enumerate}
    \item[(0)] If $p \neq 2,3$, then there are no admissible $\widetilde{P} \in \widetilde{X}$.
    \item\label{prop case 1B p=2 IV*} If $p = 2$, then $\widetilde{P}$ is admissible if and only if $C$ is of type ${\rm IV}^*$. 
    Moreover, then $({\rm Stab}_{\Aut_{\widetilde{X}}^0}(\widetilde{P}))^0 \cong \mu_2$ and $m=2$.
    \item\label{prop case 1B p=3 II} If $p = 3$, then $\widetilde{P}$ is admissible if and only if $C$ is of type ${\rm II}$ 
    {and $\widetilde{P}$ lies on a $(-1)$-curve}. 
    Moreover, then $({\rm Stab}_{\Aut_{\widetilde{X}}^0}(\widetilde{P}))^0 \cong \mu_3$ and $m=3$.
\end{enumerate}
\end{Proposition}

\begin{proof}
$X$ is given by $y^2 + st^2y = x^3$
 with $\Aut_{\widetilde{X}}^0 \cong \mathbb{G}_m$ acting as $[s:t:x:y] \mapsto [\lambda^2s:\lambda^{-1}t:x:y]$ (see Table \ref{Table four families}). We note that the $E_6$-singularity is at $[1:0:0:0]$, whereas the $A_2$-singularity is at $[0:1:0:0]$. To find admissible $P \in X$ (according to Strategy \ref{strategy admissible points mit stabi}), we distinguish the following cases:
\begin{enumerate}[leftmargin=0.8cm]

\item[(a)] If $s=0$, we can assume $t=1$. For the action $[0:1:x:y] \mapsto [0:1: \lambda^2 x : \lambda^3y]$ to fix $P$, we must either have $x=y=0$, in which case $P$ would be the $A_2$-singularity, or $x,y \neq 0$ and $\lambda =1$, in which case $\Stab_{\mathbb{G}_m}(P)$ is trivial.

\item[(b)] \label{1B cases s=1} Thus, we can assume $s=1$.

\begin{enumerate}

\item[(1)] \label{1B cases s=1 ,t=0, char 2} 
If $t=0$, then for the action $[1:0:x:y] \mapsto [1: 0: \lambda^{-4}x: \lambda^{-6}y]$ to fix $P$, we must either have $x=y=0$, in which case $P$ would be the $E_6$-singularity, or $x,y \neq 0$ and $\lambda^2=1$. Thus, $(\Stab_{\mathbb{G}_m}(P))^0$ is non-trivial if and only if $p=2$ and 
$$
P=[1:0:x:y] \hspace{2mm} \text{with}  \hspace{2mm} (x,y) \neq (0,0) \hspace{2mm}\text{and}\hspace{2mm} y^2=x^3.
$$
In this case, $(\Stab_{\mathbb{G}_m}(P))^0 \cong \mu_2$. 
Moreover, since $P$ and the $E_6$-singularity lie on the same fiber of the projection $\mathbb{P}(1,1,2,3) \supseteq X \dashrightarrow \mathbb{P}^1$ onto $s$ and $t$, $\widetilde{P}$ lies on the identity component of a curve $C \in |-K_{\widetilde{X}}|$ of type ${\rm IV}^*$. Since $P$ lies on the cuspidal curve $X \cap \{t =0\}$, we have $C^0 \cong \mathbb{G}_a$ as group schemes and thus $m =2$ by Corollary \ref{cor: approach for classification non-jacobian}.

\item[(2)] \label{1B cases s=1, t neq 0, char 3}
If $t \neq 0$, then for the action $[1:t:x:y] \mapsto [1: \lambda^{-3}t: \lambda^{-4}x: \lambda^{-6}y]$ to fix $P$, we immediately see that $\lambda^3=1$ must hold. 
If furthermore $x$ was non-zero, this would imply $\lambda= \lambda^4 =1$. Thus, we can assume $x=0$. \\
Then, $(\Stab_{\mathbb{G}_m}(P))^0$ is non-trivial if and only if $p=3$ and 
$$
P=[1:t:0:y] \hspace{2mm} \text{with}  \hspace{2mm} t \neq 0 \hspace{2mm}\text{and}\hspace{2mm} y \in \{0, -t^2\}.
$$
For such points, $(\Stab_{\mathbb{G}_m}(P))^0 \cong \mu_3$. 
Moreover, from the location of the singular points of $X$, we see that ${\widetilde{P}}$ lies in an irreducible fiber $C$ of $\widetilde{X} \to X \dashrightarrow \mathbb{P}^1$. Since $p=3$, our equation for $X$ is the Weierstra{\ss} equation of a quasi-elliptic fibration \cite[Theorem 3.3(2)]{QuasiEllipticChar3}, hence $C$ is of type ${\rm II}$, $C^0 \cong \mathbb{G}_a$ and $m=3$ by Corollary \ref{cor: approach for classification non-jacobian}. 
{Finally, note that the equations $y = 0$ and $y = -t^2$ are exactly the equations of the $(-1)$-curves on $\widetilde{X}$ that are not contained in members of $|-K_{\widetilde{X}}|$.} 
\end{enumerate}
\end{enumerate}
\vspace{-5mm}\end{proof}

\begin{Corollary} \label{cor: 1B non-Jac}
Let $\widetilde{Z}$ be arising from an $\widetilde{X}$ of type \hyperref[Tab1B]{$1B$} and assume that $h^0(\widetilde{Z},T_{\widetilde{Z}}) \neq 0$. Then, 
\begin{enumerate}
\item\label{cor 1B p=2} \underline{either} $p=2$, $\widetilde{Z}$ is unique up to isomorphism, has one multiple fiber $2 {\rm IV}^*$,
and $\Aut_{\widetilde{Z}}^0 \cong \mu_2$, 
\item\label{cor 1B p=3} \underline{or} $p=3$, $\widetilde{Z}$ is unique up to isomorphism, has one multiple fiber $3 {\rm II}$,
and $\Aut_{\widetilde{Z}}^0 \cong \mu_3$.
\end{enumerate} 
\end{Corollary}

\begin{proof}
Everything except the uniqueness follows by combining Corollary \ref{cor: approach for classification non-jacobian} with Proposition \ref{prop: 1B}. 
\begin{enumerate}
    \item If $p=2$, for the uniqueness of $\widetilde{Z}$ in (\ref{cor 1B p=2}), it suffices to observe that all points $\widetilde{P} \in C^0$, where $C \in |-K_{\widetilde{X}}|$ is \emph{the} curve of type ${\rm IV}^*$ are conjugate under $\Aut(\widetilde{X})$. This follows from our description of the $\mathbb{G}_m$-action on $X$ in Table \ref{Table four families} and the proof of Proposition \ref{prop: 1B}(\ref{prop case 1B p=2 IV*}): $\mathbb{G}_m$ sends a point of the form $[1:0:x:y]$, where $x,y \neq 0$ and $y^2=x^3$, to $[1:0:\lambda^{-4}x: \lambda^{-6}y]$, so all such points are in the same $\mathbb{G}_m$-orbit.
    \item If $p=3$, for the uniqueness of $\widetilde{Z}$ in (\ref{cor 1B p=3}), firstly, we observe that all the curves $C \in |-K_{\widetilde{X}}|$ of type ${\rm II}$ are conjugate under $\Aut(\widetilde{X})$, and, secondly, that in every such fiber $C$, the two points $\widetilde{P}$, whose blow-up yields $\widetilde{Z}$ are interchanged simultaneously by an automorphism of $X$. Both follow from our description of the $\mathbb{G}_m$-action on $X$ in Table \ref{Table four families} and the proof of Proposition \ref{prop: 1B}(\ref{prop case 1B p=3 II}): First, $\mathbb{G}_m$ sends a fiber $\{[1:t:x:y]  \hspace{1mm}| \hspace{1mm} y^2+ t^2y= x^3 \}$ over $[1:t], t \neq 0$, to the fiber over $[1: \lambda^{-3}t]$, hence all such fibers are conjugate under $\mathbb{G}_m$. Second, for fixed $t \neq 0$, the two points $[1:t:0:0]$ and $[1:t:0:-t^2]$ are interchanged by the involution of $X$ given by $y \mapsto -y -st^2$.
\end{enumerate}
\end{proof}

\begin{Discussion} \label{Discussion: 1B geometry and curve configs}
In both of the above cases, we again see the Mordell--Weil group of the Jacobian rational (quasi-)elliptic fibration $\widetilde{Y} \to \mathbb{P}^1$ associated to $\widetilde{X}$: By \cite{OguisoShioda} ${\rm MW}(\widetilde{Y} \to \mathbb{P}^1) \cong \mathbb{Z}/3 \mathbb{Z}$ and the two sections that are visible in the equation for $X$ are given by $X \cap \{y=0\} $ and $X \cap \{y= -st^2\}$. These sections are interchanged by the automorphism $y \mapsto -y-st^2$. By \cite{ExtremalCharpII} and \cite{QuasiEllipticChar3}, $\widetilde{Y}$ is elliptic with singular fibers ${\rm IV}^*$ and ${\rm IV}$ if $p=2$, and quasi-elliptic with reducible fibers ${\rm IV}^*$ and ${\rm IV}$ if $p=3$.

    To determine the number and configuration of $(-2)$-curves on $\widetilde{Z}$ we treat cases (\ref{cor 1B p=2}) and (\ref{cor 1B p=3}) of Corollary \ref{cor: 1B non-Jac} separately.
    
    \begin{enumerate}
        \item By the proof of Proposition \ref{prop: 1B}(\ref{prop case 1B p=2 IV*}), we know that $\widetilde{P} \in \widetilde{X}$ lies on a $(-1)$-curve intersecting the $E_6$- but not on a $(-1)$-curve intersecting the $A_2$-configuration of $(-2)$-curves. Hence, $\widetilde{Z}$ contains configuration ${\rm IV}^*$.

\begin{table}[H]
\begin{adjustbox}{center}
$
\begin{array}{ccc}
 \begin{array}{c} 
 {\includegraphics[width=0.22\textwidth]{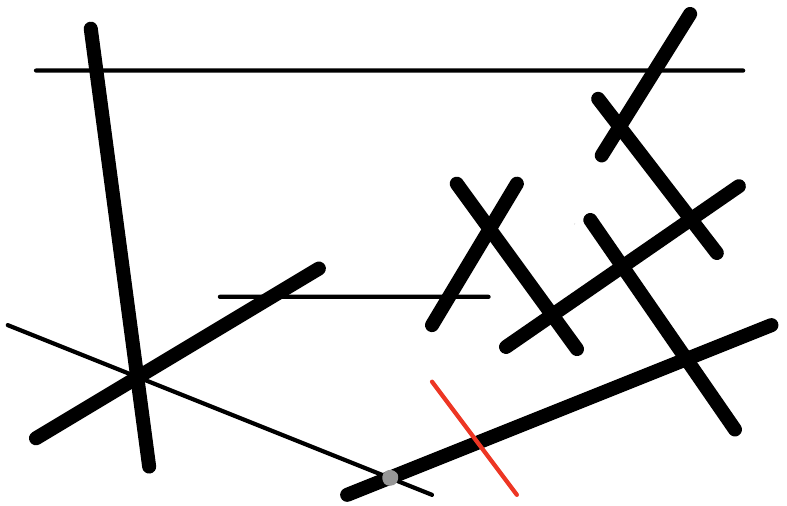} }\end{array}
 & \rightarrow
  & \begin{array}{c}
  {\includegraphics[width=0.22\textwidth]{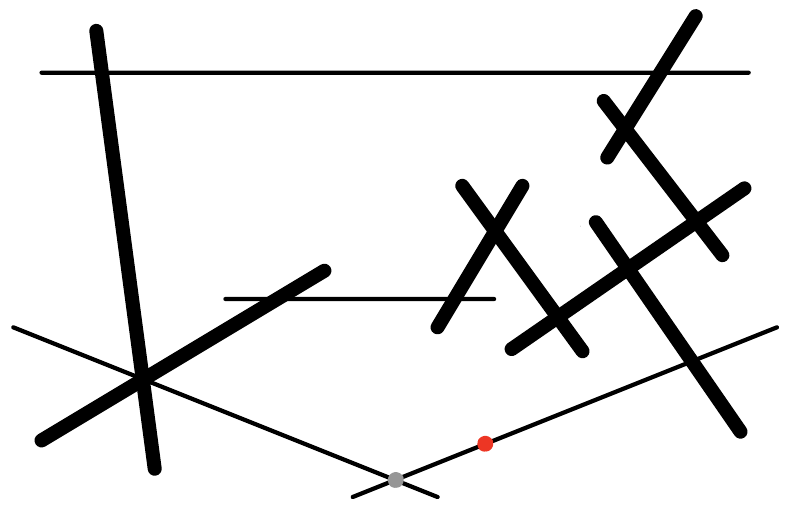} }\end{array}
\end{array}
$
\captionof{figure}{$\widetilde{Z} \to \widetilde{X}$ with incomplete $(-2)$-curve configuration on $\widetilde{Z}$} \label{figure 1B, p=2 Ztilde to Xtilde, incomplete}
\end{adjustbox}
\end{table}

        From the Kodaira--N\'eron classification of fiber types and Lemma \ref{lemma (-2)curves on Ztilde} we see that the other two obvious $(-2)$-curves on $\widetilde{Z}$ (in the left of the picture for $\widetilde{Z}$ in Figure \ref{figure 1B, p=2 Ztilde to Xtilde, incomplete}) have to constitute a dual graph $\widetilde{A}_2$ together with another $(-2)$-curve, that we were not yet able to see as a negative curve on $\widetilde{X}$. We will now determine the precise configuration -- either ${\rm IV}$ or ${\rm I_3}$ -- of the two ``known'' $(-2)$-curves with the ``new'' $(-2)$-curve: Taking into account that, by Proposition \ref{prop: 1B}(\ref{prop case 1B p=2 IV*}), every $(-1)$-curve on $\widetilde{Z}$ is a $2$-section, we obtain that the $(-1)$-curve in Figure \ref{figure 1B, p=2 Ztilde to Xtilde, incomplete} that intersects both the ``known'' $(-2)$-curves in one point cannot intersect the ``new'' $(-2)$-curve. Hence, configuration ${\rm IV}$ is not possible.

        We are therefore left with the four possibilities for the intersection behavior of ${\rm I_3}$ as in Figure \ref{Figure 1B p=2, 4 Möglichkeiten für configs}, where the ``known'' $(-2)$-curves are still drawn in black and the ``new'' $(-2)$-curve is drawn in purple. Note that we also assigned other colors to some of the remaining negative curves in order to be able to better refer to them in the argument. We remark that, in order to not overload these drawings, we did not yet include the intersection behavior of the red exceptional curve on $\widetilde{Z}$ with other curves.

\begin{table}[H]
\begin{adjustbox}{center}
$
 \begin{array}{c} 
 \includegraphics[width=0.90\textwidth]{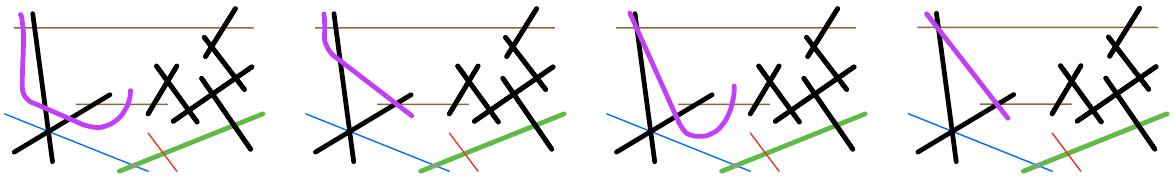} \end{array}
$
\captionof{figure}{Four possibilities for configurations of $(-2)$-curves on $\widetilde{Z}$} \label{Figure 1B p=2, 4 Möglichkeiten für configs}
\end{adjustbox}
\end{table}

When contracting the blue $(-1)$-curve in Figure \ref{Figure 1B p=2, 4 Möglichkeiten für configs}, we obtain a realization of $\widetilde{Z}$ as blow-up of another weak del Pezzo surface with global vector fields containing an $E_6$-configuration of $(-2)$-curves and at least five $(-1)$-curves. So, by the classification in \cite{WeakDelPezzoGlobalVectorFields}, this weak del Pezzo surface is either of type $1K$ or $1J$. Since the blue $(-1)$-curve does not intersect the purple $(-2)$-curve, its image under the contraction is still a $(-2)$-curve that does not intersect the \linebreak $E_6$-configuration of $(-2)$-curves on the contraction. So, $\widetilde{Z}$ is a blow-up of the weak del Pezzo surface $\widetilde{X}_{1K}$ of type $1K$. We can identify some of the curves in Figure \ref{Figure 1B p=2, 4 Möglichkeiten für configs} with curves on $\widetilde{X}_{1K}$ according to the color they are given below and learn about their intersection behavior.

\begin{table}[H]
\begin{adjustbox}{center}
$
\begin{array}{ccc}
 \begin{array}{c} 
 {\includegraphics[width=0.22\textwidth]{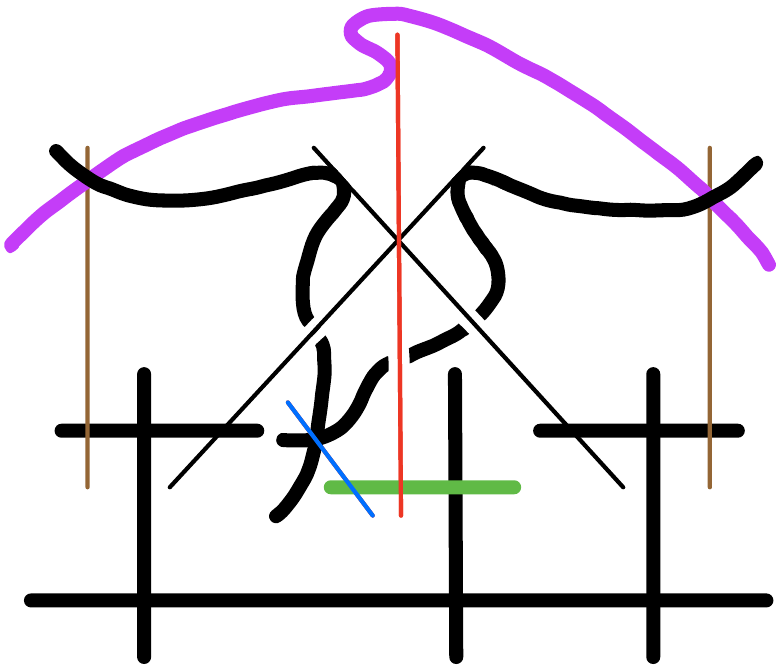} }\end{array}
 & \rightarrow
  & \begin{array}{c}
  {\includegraphics[width=0.22\textwidth]{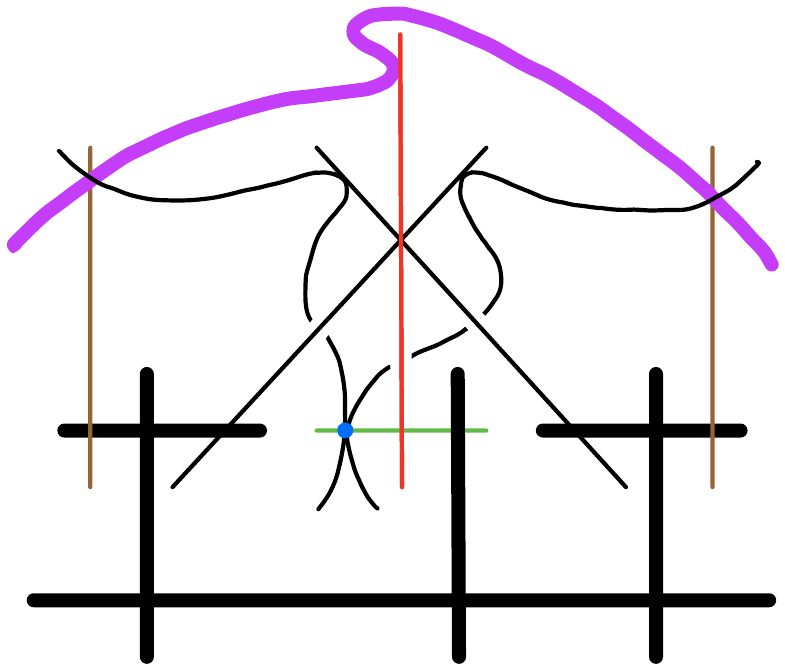} }\end{array}
\end{array}
$
\captionof{figure}{$\widetilde{Z}$ as a blow-up of a weak del Pezzo surface of type $1K$ ($p=2$)} \label{figure 1B, p=2 VON 1K}
\end{adjustbox}
\end{table}

From a comparison with Figure \ref{figure 1B, p=2 VON 1K} we see that the fourth configuration in Figure \ref{Figure 1B p=2, 4 Möglichkeiten für configs} is correct. Moreover, there are (at least) three more $(-1)$-curves on $\widetilde{Z}$ than visible in Figure \ref{figure 1B, p=2 Ztilde to Xtilde, incomplete}, and the red exceptional curve intersects the purple $(-2)$-curve in one point with multiplicity $2$.
The results of this discussion are summarized in the following Figure \ref{figure 1B, p=2 configs jaco, non-jaco} and Corollary \ref{cor: 1B (-2)configs}.

\begin{table}[H]
\begin{adjustbox}{center}
$
\begin{array}{ccccc}
 \begin{array}{c} \addstackgap[2pt]{\includegraphics[width=0.22\textwidth]{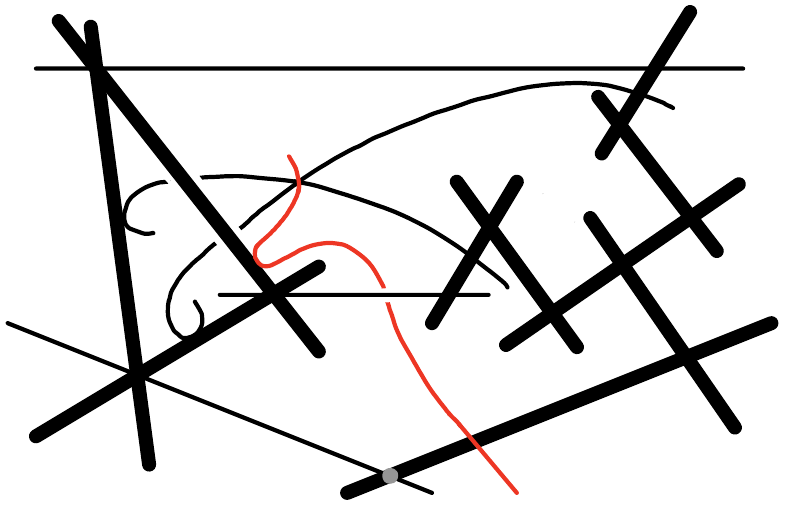} }\end{array}
 & \rightarrow
  & \begin{array}{c}\addstackgap[2pt]{\includegraphics[width=0.22\textwidth]{1B-2-Xtilde.pdf} }\end{array}
  & \leftarrow
  & \begin{array}{c}\addstackgap[2pt]{\includegraphics[width=0.22\textwidth]{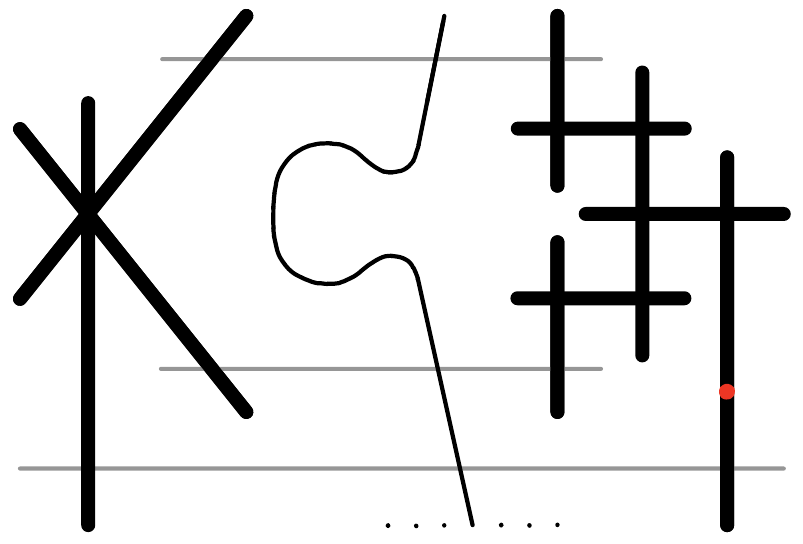} }\end{array}
  \\
  &
  & \downarrow
  &
  & \downarrow
  \\
  &
  & \begin{array}{c}\addstackgap[2pt]{\includegraphics[width=0.22\textwidth]{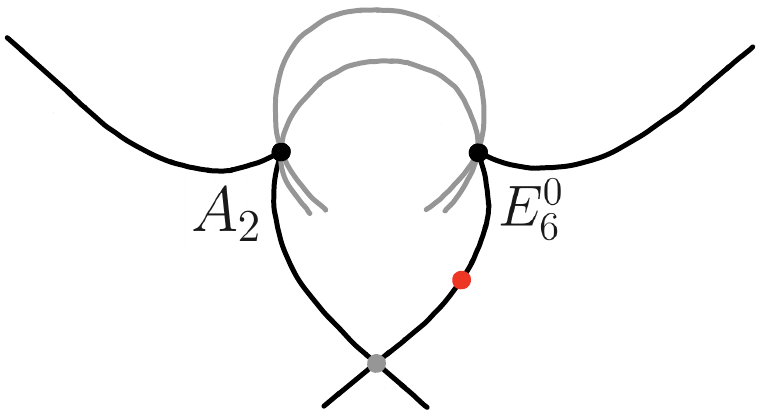}} \end{array}
  & \leftarrow
  & \begin{array}{c}\addstackgap[2pt]{ \includegraphics[width=0.22\textwidth]{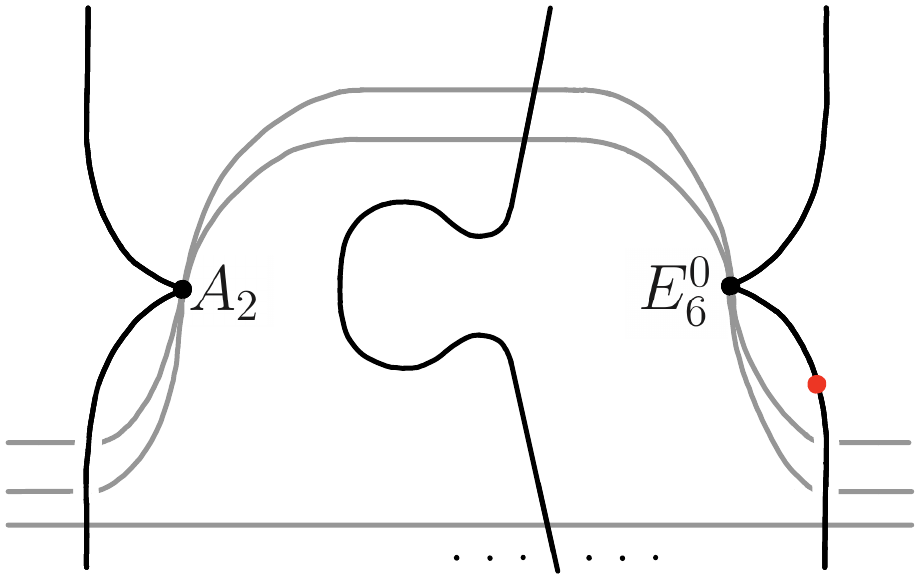}}  \end{array}
\end{array}
$
\captionof{figure}{Non-Jacobian and Jacobian fibrations with global vector fields originating from $\widetilde{X}$ of type \hyperref[Tab1B]{$1B$} ($p=2$)} \label{figure 1B, p=2 configs jaco, non-jaco}
\end{adjustbox}
\end{table}

    \item
   By the proof of Proposition \ref{prop: 1B}(\ref{prop case 1B p=3 II}), $\widetilde{P}$ lies on a $(-1)$-curve connecting the $E_6$- and the $A_2$-configuration of $(-2)$-curves on $\widetilde{X}$. Thus, $\widetilde{Z}$ contains $(-2)$-curves forming a configuration of type ${\rm II}^*$ and, by Lemma \ref{lemma (-2)curves on Ztilde} no further ones. As in the previous cases, the situation is illustrated and summarized below in Figure \ref{figure 1B, p=3 configs jaco, non-jaco} and Corollary \ref{cor: 1B (-2)configs}.

    \begin{table}[H]
\begin{adjustbox}{center}
$
\begin{array}{ccccc}
 \begin{array}{c} \addstackgap[2pt]{\includegraphics[width=0.22\textwidth]{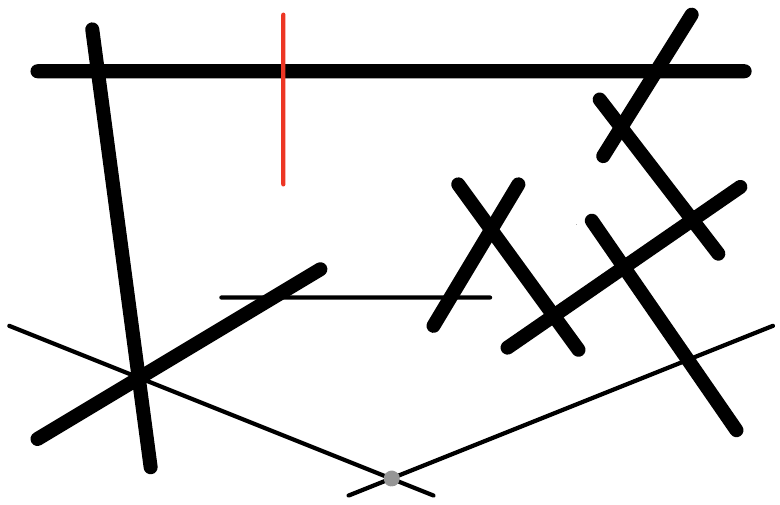} }\end{array}
 & \rightarrow
  & \begin{array}{c}\addstackgap[2pt]{\includegraphics[width=0.22\textwidth]{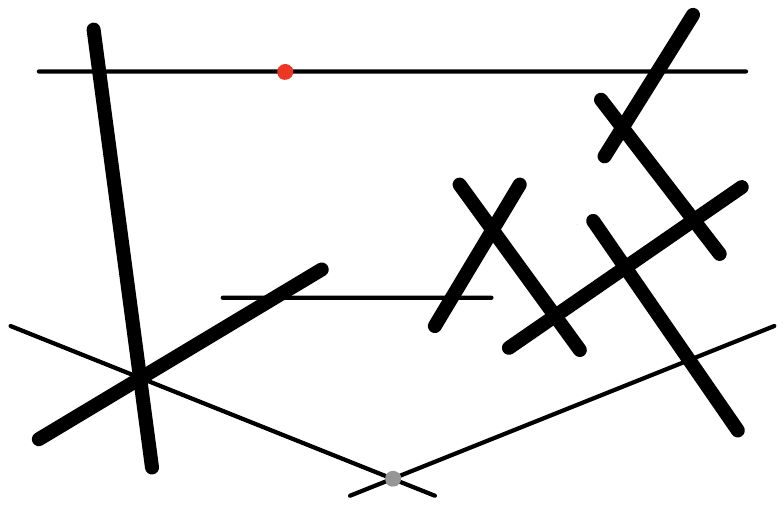} }\end{array}
  & \leftarrow
  & \begin{array}{c}\addstackgap[2pt]{\includegraphics[width=0.22\textwidth]{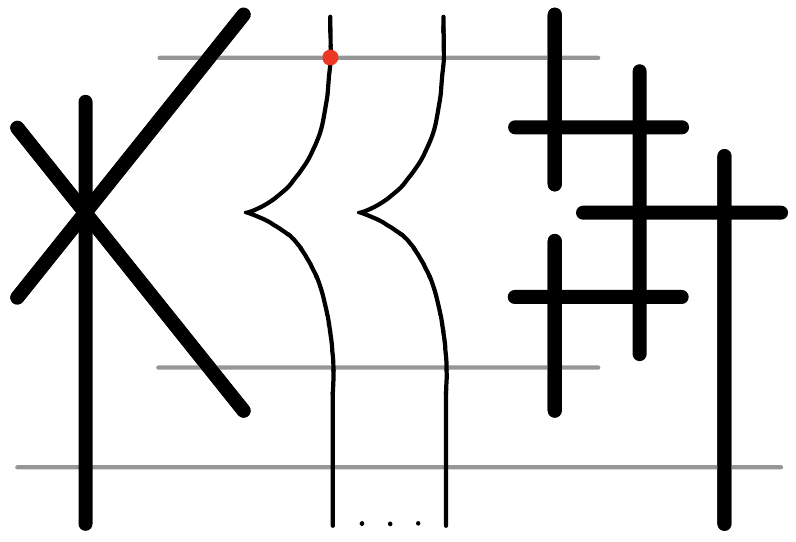} }\end{array}
  \\
  &
  & \downarrow
  &
  & \downarrow
  \\
  &
  & \begin{array}{c}\addstackgap[2pt]{\includegraphics[width=0.22\textwidth]{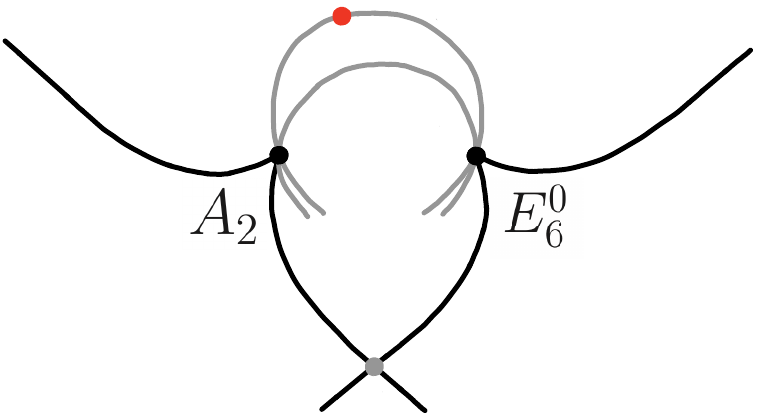}} \end{array}
  & \leftarrow
  & \begin{array}{c}\addstackgap[2pt]{ \includegraphics[width=0.22\textwidth]{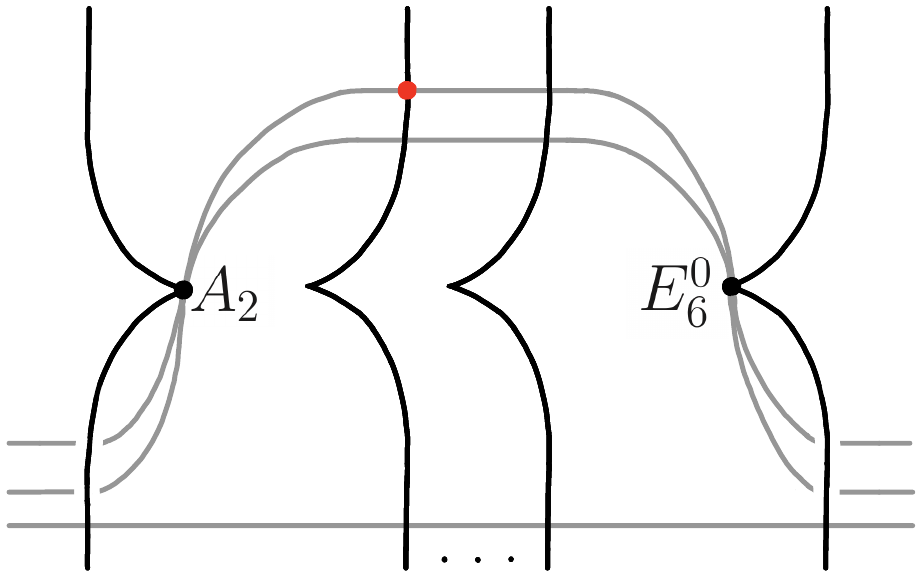}}  \end{array}
\end{array}
$
\captionof{figure}{Non-Jacobian and Jacobian fibrations with global vector fields originating from $\widetilde{X}$ of type \hyperref[Tab1B]{$1B$} ($p=3$)} \label{figure 1B, p=3 configs jaco, non-jaco}
\end{adjustbox}
\end{table}
        
    \end{enumerate}
\end{Discussion}

\begin{Corollary} \label{cor: 1B (-2)configs}
    Let $\widetilde{Z}$ be as in Corollary \ref{cor: 1B non-Jac}. Then, the following hold.
    \begin{enumerate}
    \item If $p=2$, $\widetilde{Z}$ contains ten $(-2)$-curves with dual graph of type $\widetilde{E}_6 + \widetilde{A}_2$ forming configurations ${\rm IV}^*$ and ${\rm I}_3$. Moreover, ${\rm IV}^*$ is the unique multiple fiber and of multiplicity $m=2$.
    \item If $p=3$, $\widetilde{Z}$ contains nine $(-2)$-curves with dual graph of type $\widetilde{E}_8$ forming configuration ${\rm II}^*$.
    \end{enumerate}
\end{Corollary}

\subsection{Case \hyperref[Tab1C]{$1C$}} \label{subsection 1C}

This family exists only if $\Char(k)=p \neq 2$.

\begin{Proposition} \label{prop: 1C}
Let $\widetilde{X}$ be of type \hyperref[Tab1C]{$1C$}. 
\begin{enumerate}
    \item[(0)] If $p \neq 3$, then there are no admissible $\widetilde{P} \in \widetilde{X}$.
    \item[(1)] If $p = 3$, then $\widetilde{P}$ is admissible if and only if $C$ is of type ${\rm III}^*$. Moreover, then $({\rm Stab}_{\Aut_{\widetilde{X}}^0}(\widetilde{P}))^0 \cong \mu_3$ and $m=3$.
\end{enumerate}
\end{Proposition}

\begin{proof}
$X$ is given by $y^2 = x^3 + st^3x$
 with $\Aut_{\widetilde{X}}^0 \cong \mathbb{G}_m$ acting as $[s:t:x:y] \mapsto [\lambda^3s:\lambda^{-1}t:x:y]$ (see Table \ref{Table four families}). We note that the $E_7$-singularity is at $[1:0:0:0]$, whereas the $A_1$-singularity is at $[0:1:0:0]$. 
 To find admissible points $P$ (according to Strategy \ref{strategy of proof non-Jaco}), we distinguish the following cases:
\begin{enumerate}[leftmargin=0.8cm]
\item[(a)] If $s=0$, we can assume $t=1$. For the action $[0:1:x:y] \mapsto [0:1:\lambda^2 x : \lambda^3 y]$ to fix $P$, we must either have $x=y=0$, in which case $P$ would be the $A_1$-singularity, or $x,y \neq 0$ and $\lambda=1$, in which case ${\rm Stab}_{\mathbb{G}_m}(P)$ is trivial.
\item[(b)] Thus, we can assume $s=1$. If $t \neq 0$, then for the action $[1:t:x:y] \mapsto [1:\lambda^{-4}t: \lambda^{-6}x: \lambda^{-9}y]$ to fix $P$, we see that $\lambda^4=1$ must hold. Since $p \neq 2$, this implies $(\Stab_{\mathbb{G}_m}(P))^0=  \{{\rm id}\}$.  
\begin{enumerate}
\item So, we can assume $t=0$ and see that for the above action to fix $P$ we get either $x=y=0$, in which case $P$ would be the $E_7$-singularity, or $x, y \neq 0$ and $\lambda^3=1$. Thus $(\Stab_{\mathbb{G}_m}(P))^0$ is non-trivial if and only if $p=3$ and 
$$P =[1:0:x:y] \hspace{2mm} \text{ with } \hspace{2mm} (x,y) \neq (0,0) \hspace{2mm} \text{and} \hspace{2mm} y^2= x^3. 
$$
In this case, $(\Stab_{\mathbb{G}_m}(P))^0 \cong \mu_3$. Moreover, since $P$ and the $E_7$-singularity lie on the same fiber of the projection $\mathbb{P}(1,1,2,3) \supseteq X \dashrightarrow \mathbb{P}^1$ onto $s$ and $t$, $\widetilde{P}$ lies on the identity component of a curve $C \in |-K_{\widetilde{X}}|$ of type ${\rm III}^*$. 
Since $P$ lies on the cuspidal curve $X \cap \{t=0\}$, we have $C^0 \cong \mathbb{G}_a$ and thus $m=3$ by Corollary \ref{cor: approach for classification non-jacobian}.  
\end{enumerate}
\end{enumerate}
\vspace{-5mm}\end{proof}

\begin{Corollary} \label{cor: 1C non-Jac}
Let $\widetilde{Z}$ be arising from an $\widetilde{X}$ of type \hyperref[Tab1C]{$1C$} and assume that $h^0(\widetilde{Z},T_{\widetilde{Z}}) \neq 0$. Then, $p = 3$, $\widetilde{Z}$ is unique up to isomorphism, has one multiple fiber $3 {\rm III}^*$,
and $\Aut_{\widetilde{Z}}^0 \cong \mu_3$. 
\end{Corollary}

\begin{proof}
Everything except the uniqueness follows by combining Corollary \ref{cor: approach for classification non-jacobian} with Proposition \ref{prop: 1C}. To show that $\widetilde{Z}$ is unique up to isomorphism, it suffices to observe that all points $\widetilde{P} \in C^0$ where $C \in |-K_{\widetilde{X}}|$ is the curve of type ${\rm III}^*$ are conjugate under $\Aut(\widetilde{X})$. This follows from our description of the $\mathbb{G}_m$-action on the Weierstra{\ss} model in Table \ref{Table four families}: The points on $C^0$ which do not lie on the zero section are of the form $[1:0:x:y]$ with $x,y \neq 0, y^2 = x^3$, and $\mathbb{G}_m$ sends such a point to $[1:0:\lambda^{-6}x: \lambda^{-9}y]$, so all such points are in the same $\mathbb{G}_m$-orbit.
\end{proof}

\begin{Discussion} \label{Discussion: 1C geometry and curve configs}
    
In the configuration of curves on $\widetilde{X}$, the base point has to be the intersection of the two intersecting $(-1)$-curves. Blowing it up yields $\widetilde{Y}$, which is elliptic with singular fibers ${\rm III}^*$ and ${\rm III}$ by \cite{ExtremalCharpII} and has Mordell--Weil group $\mathbb{Z}/2\mathbb{Z}$ by \cite{OguisoShioda}. Thus, the $(-1)$-curve on $\widetilde{X}$ which does not intersect another $(-1)$-curve corresponds to the non-zero section on $\widetilde{Y}$. 

We have seen in the proof of Proposition \ref{prop: 1C} that $\widetilde{P} \in \widetilde{X}$ lies on the $(-1)$-curve that contains the base point and intersects the $E_7$-configuration of $(-2)$-curves. Thus, $\widetilde{Z}$ contains configuration ${\rm III}^*$. 

\begin{table}[H]
\begin{adjustbox}{center}
$
\begin{array}{ccc}
 \begin{array}{c} 
 {\includegraphics[width=0.22\textwidth]{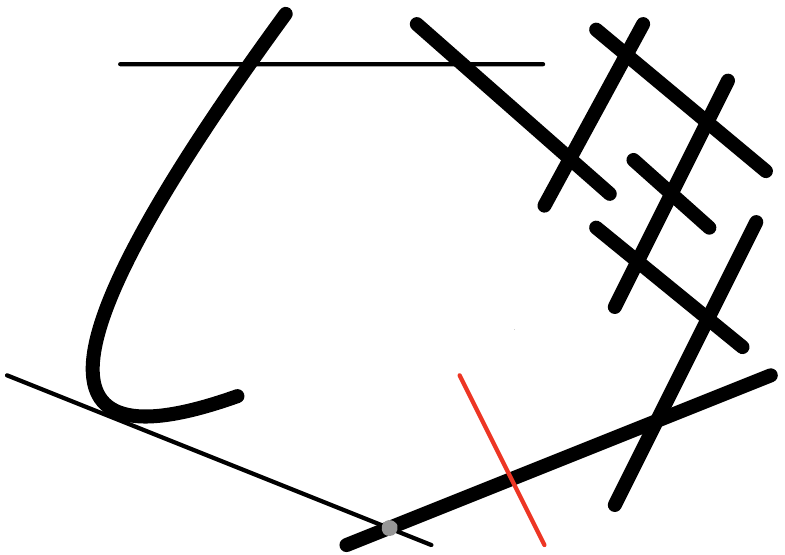} }\end{array}
 & \rightarrow
  & \begin{array}{c}
  {\includegraphics[width=0.22\textwidth]{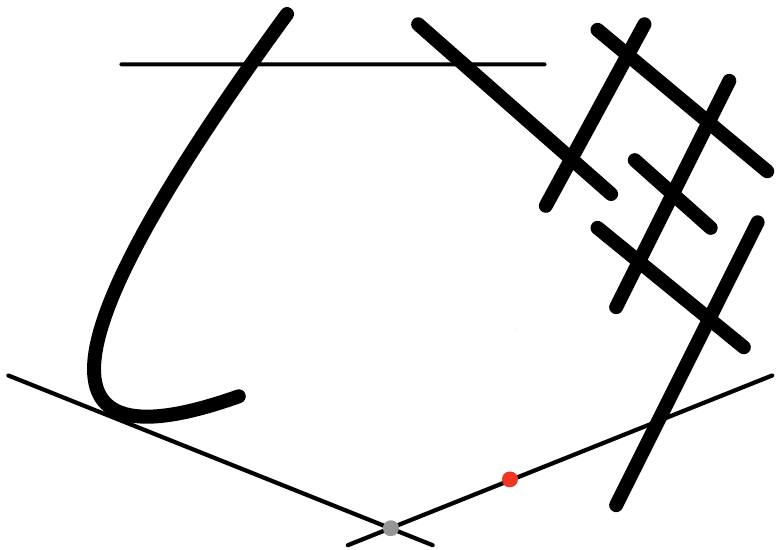} }\end{array}
\end{array}
$
\captionof{figure}{$\widetilde{Z} \to \widetilde{X}$ with incomplete $(-2)$-curve configuration on $\widetilde{Z}$} \label{figure 1C, Ztilde to Xtilde, incomplete}
\end{adjustbox}
\end{table}

Aiming for the entire configuration of $(-2)$-curves on $\widetilde{Z}$, we recall from the classification of reducible fibers of $\widetilde{Z}$ and Lemma \ref{lemma (-2)curves on Ztilde} that the single $(-2)$-curve on $\widetilde{Z}$ has to constitute a dual graph $\widetilde{A}_1$ together with another $(-2)$-curve, that we were not yet able to see as a negative curve on $\widetilde{X}$. We will now determine the precise configuration -- either ${\rm I_2}$ or ${\rm III}$ -- of the ``known'' $(-2)$-curve with the ``new'' one on $\widetilde{Z}$: Taking into account that every $(-1)$-curve on $\widetilde{Z}$ is a $3$-section, we have the following ten possibilities for their intersection behavior, where the first two rows in Figure \ref{figure 1C, 10 Möglichkeiten für (-2)-Kurve} show the ${\rm I_2}$-cases and the third row shows the two ${\rm III}$-cases. Here, the ``known'' resp. ``new'' $(-2)$-curves are drawn in black resp. purple. Note that we also assigned other colors to some of the remaining negative curves to be able to better refer to them in the argument. We remark that, in order to not overload these drawings, we did not yet include the intersection behavior of the red exceptional curve on $\widetilde{Z}$ with other curves.

\begin{table}[H]
\begin{adjustbox}{center}
$
 \begin{array}{c} 
 \includegraphics[width=0.88\textwidth]{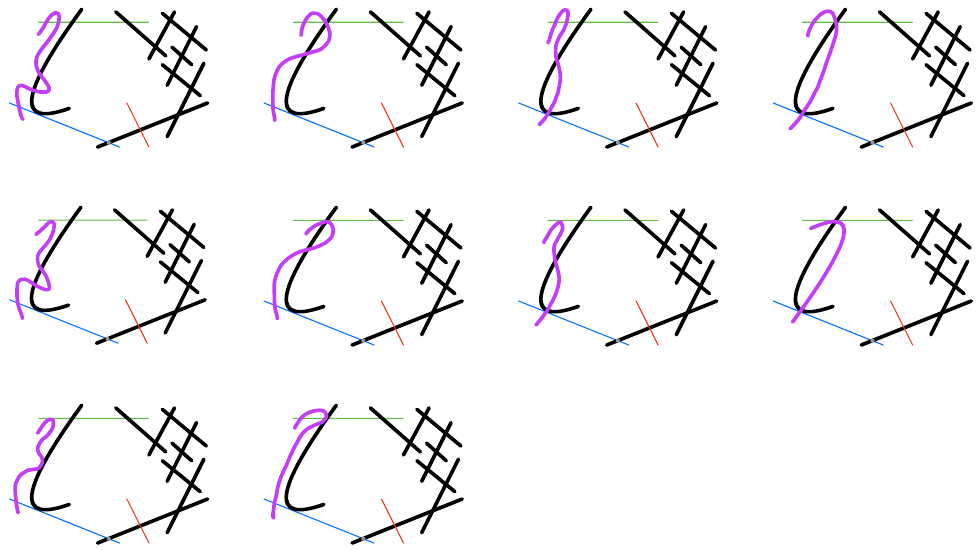} \end{array}
$
\captionof{figure}{Ten possibilities for configurations of $(-2)$-curves on $\widetilde{Z}$} \label{figure 1C, 10 Möglichkeiten für (-2)-Kurve}
\end{adjustbox}
\end{table}

When contracting the blue $(-1)$-curve in Figure \ref{figure 1C, 10 Möglichkeiten für (-2)-Kurve}, this yields a realization of $\widetilde{Z}$ as a blow-up of another weak del Pezzo surface with global vector fields and an $E_7$-configuration of $(-2)$-curves and two $(-1)$-curves with intersection number $2$. By the classification of \cite{WeakDelPezzoGlobalVectorFields}, this must be $\widetilde{X}_{1F}$ of type \hyperref[Tab1F]{$1F$}. 
The blown-up point $\widetilde{P}_{1F} \in \widetilde{X}_{1F}$ has to be one of the intersection points of the horizontal $(-1)$-curve with the curved ones in Case \hyperref[Tab1F]{$1F$} in Table \ref{Table char3 equations and liftable actions}. By symmetry of the configuration on $\widetilde{X}_{1F}$, we can choose $\widetilde{P}_{1F}$ to be the blue point in Figure \ref{figure 1C, Ztilde und Xtilde von 1F}.
 Although, when contracting the blue $(-1)$-curve, the ``known'' $(-2)$-curve becomes a curve of non-negative self-intersection and hence is no longer visible in the curve configuration below, we can identify some images of the curves in Figure \ref{figure 1C, 10 Möglichkeiten für (-2)-Kurve} with curves on $\widetilde{X}_{1F}$ according to the color they are given below and learn about their intersection behavior.

\begin{table}[H]
\begin{adjustbox}{center}
$
\begin{array}{ccc}
 \begin{array}{c} 
 {\includegraphics[width=0.22\textwidth]{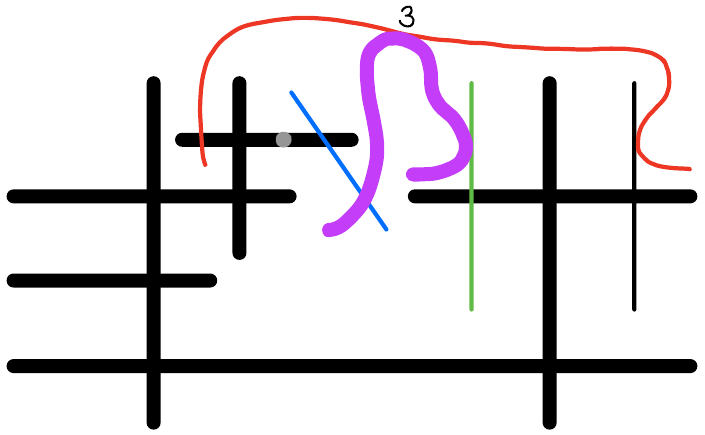} }\end{array}
 & \rightarrow
  & \begin{array}{c}
  {\includegraphics[width=0.22\textwidth]{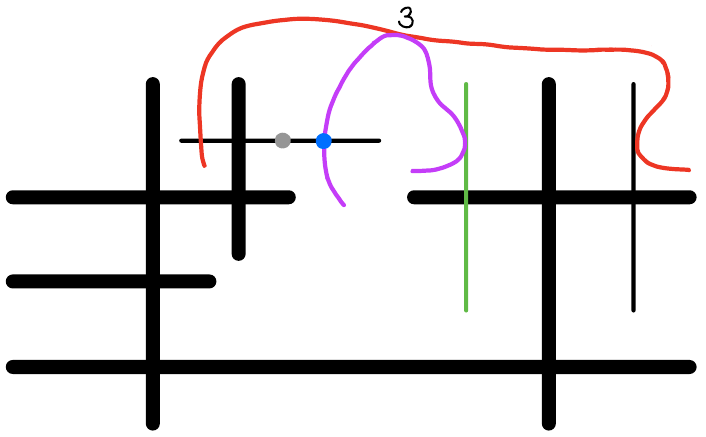} }\end{array}
\end{array}
$
\captionof{figure}{$\widetilde{Z}$ as a blow-up of a weak del Pezzo surface $\widetilde{X}_{1F}$ of type \hyperref[Tab1F]{$1F$}} \label{figure 1C, Ztilde und Xtilde von 1F}
\end{adjustbox}
\end{table}

 Thus, in Figure \ref{figure 1C, 10 Möglichkeiten für (-2)-Kurve} the purple $(-2)$-curve and the green $(-1)$-curve intersect in one point with multiplicity two, which rules out the first four possibilities in Figure \ref{figure 1C, 10 Möglichkeiten für (-2)-Kurve}. 
 {Moreover, we see that the red $(-1)$-curve meets the purple $(-2)$-curve in one point with multiplicity $3$.} To find the true configuration among the remaining six possibilities we need stronger techniques using the $\mu_3$-action on $\widetilde{Z}$:

 Since $\mu_3$ is linearly reductive, its fixed locus on $\widetilde{Z}$ is smooth by \cite[Proposition A.8.10(2)]{ConradGabberPrasad}. 
Moreover, $\mu_3$ preserves every negative curve and hence transverse intersections of negative curves are fixed points. This excludes the first configuration in the second row of Figure \ref{figure 1C, 10 Möglichkeiten für (-2)-Kurve}. Indeed, since there are at least $3$ fixed points on the purple curve and on the black {``known''} $(-2)$-curve, and since $\mu_3$ has at most $2$ fixed points on $\mathbb{P}^1$ (see for example \cite[Lemma 2.34(i)]{Martin}), both these curves have to be fixed pointwise. This contradicts the smoothness of the fixed locus $\widetilde{Z}^{\mu_3}$.

To exclude the sixth, seventh and tenth configuration of Figure \ref{figure 1C, 10 Möglichkeiten für (-2)-Kurve}, we refine the above argument and carry it out for the sixth configuration (the other two use the analogous argument for differently colored curves): The purple curve contains $3$ fixed points, hence is fixed pointwise. Let $Q$ be the point on the green $(-1)$-curve $C$, where the purple $(-2)$-curve touches $C$. Since $\mu_3$ fixes their intersection $C_1= \Spec k[x]/(x^2)$, the non-reduced $C_1$ is contained in the fixed locus $C^{\mu_3}$, which is smooth by \cite[Proposition A.8.10(2)]{ConradGabberPrasad} applied to $C$. Thus, $\mu_3$ has to act trivially on $C$. The purple curve and $C$ being contained in $\widetilde{Z}^{\mu_3}$ yields a contradiction to smoothness of $\widetilde{Z}^{\mu_3}$ (again applying \cite[Proposition A.8.10(2)]{ConradGabberPrasad} to $\widetilde{Z}$).

To exclude the ninth configuration in Figure \ref{figure 1C, 10 Möglichkeiten für (-2)-Kurve}, we cannot immediately tell that the green, blue, purple or black curve is fixed pointwise since there are not enough transverse intersections. To overcome this, let us have a closer look at a point where two of these curves meet: $\mu_3$ acts on the first order neighborhood $C_1\coloneqq k[x]/(x^2)$ of such a point. In the proof of Proposition 5.8 in \cite{RDPDelPezzoGlobalVectorFields}, we saw $\Aut_{C_1} \cong \mathbb{G}_m$ acting as $x \mapsto a_1 x$ if $p \neq 2$. Thus, the closed point of $C_1$ is fixed by $\mathbb{G}_m$ and thus also by $\mu_3$ since $C_1^{\mu_3} \supseteq C_1^{\mathbb{G}_m}$. Therefore, the green, blue, purple and black curve are fixed pointwise, which contradicts smoothness of $\widetilde{Z}^{\mu_3}$.

So, we showed that $\widetilde{Z}$ contains configurations ${\rm III}^*$ and ${\rm I}_2$ as depicted in the eighth configuration of Figure \ref{figure 1C, 10 Möglichkeiten für (-2)-Kurve}. Note that from the comparison with Figure \ref{figure 1C, Ztilde und Xtilde von 1F} we see that there is an additional $(-1)$-curve on $\widetilde{Z}$ intersecting the red $(-1)$-curve in a point with multiplicity $2$ and the ``known'' $(-2)$-curve in a point of multiplicity $3$ (the latter follows from the smooth fixed loci argument that we used above). We summarize the results of the previous discussion in Figure \ref{figure 1C, configs jaco, non-jaco} and Corollary \ref{cor: 1C (-2)configs} below.

 \begin{table}[H]
\begin{adjustbox}{center}
$
\begin{array}{ccccc}
 \begin{array}{c} \addstackgap[2pt]{\includegraphics[width=0.22\textwidth]{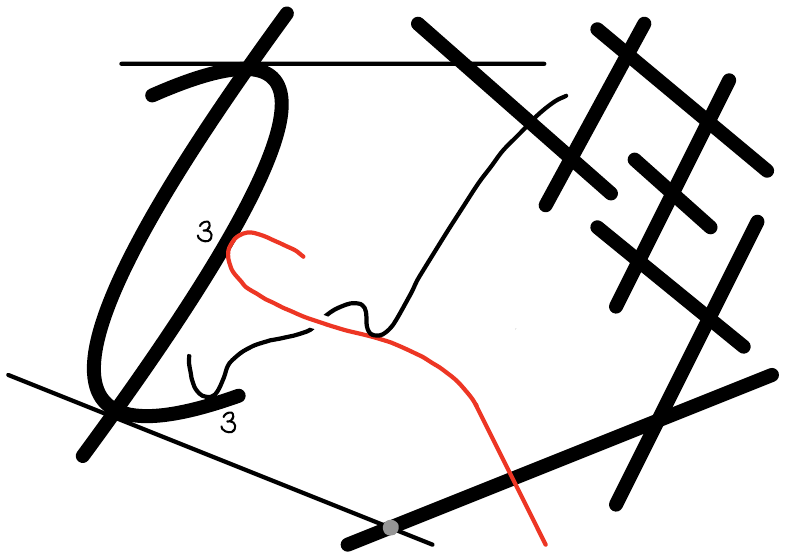} }\end{array}
 & \rightarrow
  & \begin{array}{c}\addstackgap[2pt]{\includegraphics[width=0.22\textwidth]{1C-Xtilde.pdf} }\end{array}
  & \leftarrow
  & \begin{array}{c}\addstackgap[2pt]{\includegraphics[width=0.22\textwidth]{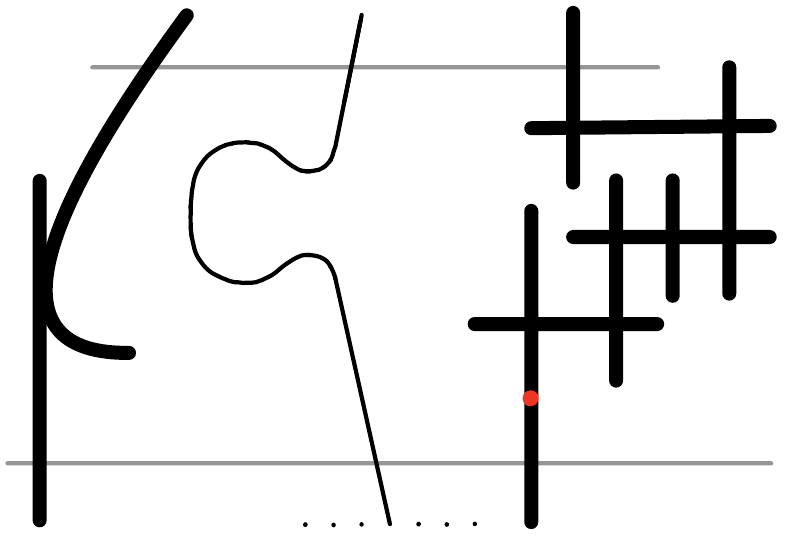} }\end{array}
  \\
  &
  & \downarrow
  &
  & \downarrow
  \\
  &
  & \begin{array}{c}\addstackgap[2pt]{\includegraphics[width=0.22\textwidth]{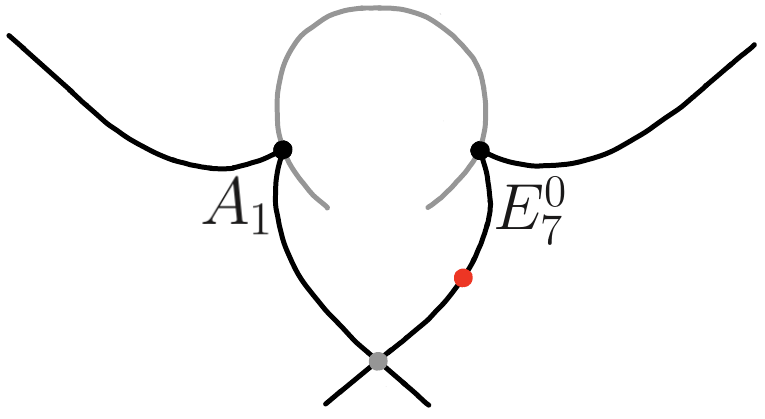}} \end{array}
  & \leftarrow
  & \begin{array}{c}\addstackgap[2pt]{ \includegraphics[width=0.22\textwidth]{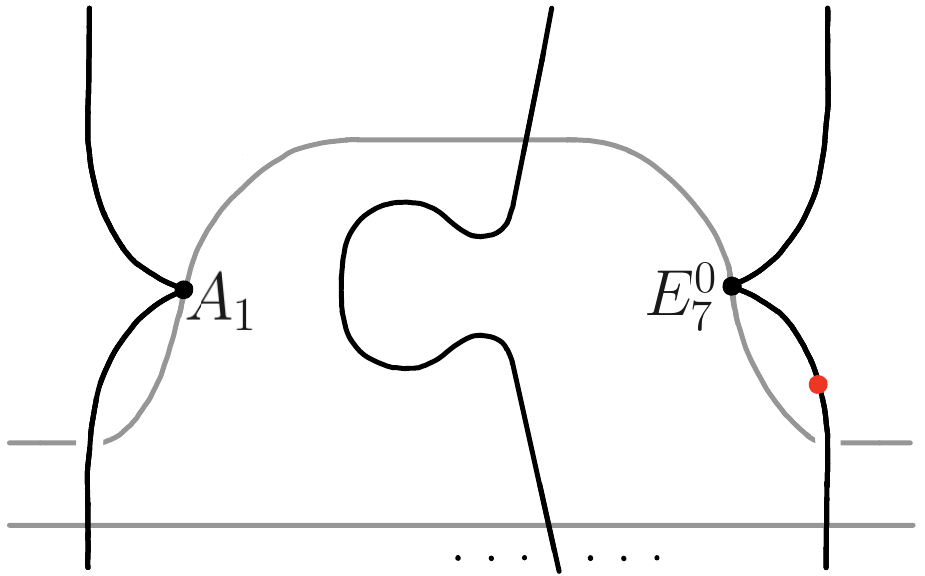}}  \end{array}
\end{array}
$
\captionof{figure}{Non-Jacobian and Jacobian fibrations with global vector fields originating from $\widetilde{X}$ of type \hyperref[Tab1C]{$1C$} ($p=3$)} \label{figure 1C, configs jaco, non-jaco}
\end{adjustbox}
\end{table}
\end{Discussion}

\begin{Corollary} \label{cor: 1C (-2)configs}
The $\widetilde{Z}$ of Corollary \ref{cor: 1C non-Jac} contains nine $(-2)$-curves with dual graph of type $\widetilde{E}_7 + \widetilde{A}_1$ forming configurations ${\rm III}^*$ and ${\rm I_2}$. Moreover, ${\rm III}^*$ is the unique multiple fiber and of multiplicity $m=3$.
\end{Corollary}

\subsection{Case \hyperref[Tab1D]{$1D$}} \label{subsection 1D}

This family exists only if $\Char(k)=p \neq 2,3$.

\begin{Proposition} \label{prop: 1D}
Let $\widetilde{X}$ be of type \hyperref[Tab1D]{$1D$}.
\begin{enumerate}
    \item[(0)] If $p \neq 5$, then there are no admissible $\widetilde{P} \in \widetilde{X}$.
    \item\label{prop 1D case p=5 II*} If $p = 5$, then $\widetilde{P}$ is admissible if and only if $C$ is of type ${\rm II}^*$. Moreover, then $({\rm Stab}_{\Aut_{\widetilde{X}}^0}(\widetilde{P}))^0 \cong \mu_5$ and $m=5$.
\end{enumerate}
\end{Proposition}

\begin{proof}
$X$ is given by $y^2 = x^3 + st^5$ with $\Aut_{\widetilde{X}}^0 \cong \mathbb{G}_m$ acting as $[s:t:x:y] \mapsto [\lambda^5s:\lambda^{-1}t:x:y]$ (see Table \ref{Table four families}). We note that the $E_8$-singularity is at $[1:0:0:0]$.
 To find admisible $P \in X$ (according to Strategy \ref{strategy of proof non-Jaco}), we distinguish the following cases:
\begin{enumerate}[leftmargin=0.8cm]
\item[(a)] If $s=0$, we can assume $t=1$. For the action $[0:1:x:y] \mapsto [0:1:\lambda^2 x : \lambda^3 y]$ to fix $P$, we have two possibilities: either $x, y \neq 0$ and $\lambda =1$, or $x=y=0$. But the point $[0:1:0:0]$ corresponds to the singular point of a $C \in |-K_{\widetilde{X}}|$ of type ${\rm II}$, hence does not lie in the smooth $C^0$.
\item[(b)] Thus, we can assume $s=1$. If $t \neq 0$, then for the action $[1:t:x:y] \mapsto$ $[1:\lambda^{-6}t: \lambda^{-10}x: \lambda^{-15}y]$ to fix $P$, we see that $\lambda^6=1$ must hold. Since $p \neq 2,3$, this implies $(\Stab_{\mathbb{G}_m}(P))^0=  \{{\rm id}\}$.  
\begin{enumerate}
\item[(1)] So, we can assume $t=0$ and see that for the above action to fix $P$ we get either $x=y=0$, in which case $P$ would be the $E_8$-singularity, or $x, y \neq 0$ and $\lambda^5=1$. Thus $(\Stab_{\mathbb{G}_m}(P))^0$ is non-trivial if and only if $p=5$ and 
$$P =[1:0:x:y] \hspace{2mm} \text{ with } \hspace{2mm} (x,y) \neq (0,0) \hspace{2mm} \text{and} \hspace{2mm} y^2= x^3. 
$$
In this case, $(\Stab_{\mathbb{G}_m}(P))^0 \cong \mu_5$. Moreover, since $P$ and the $E_8$-singularity lie on the same fiber of the projection $\mathbb{P}(1,1,2,3) \supseteq X \dashrightarrow \mathbb{P}^1$ onto $s$ and $t$, $\widetilde{P}$ lies on the identity component of a curve $C \in |-K_{\widetilde{X}}|$ of type ${\rm II}^*$. 
Since $P$ lies on the cuspidal curve $X \cap \{t=0\}$, we have $C^0 \cong \mathbb{G}_a$ and thus $m=5$ by Corollary \ref{cor: approach for classification non-jacobian}.  
\end{enumerate}
\end{enumerate}
\vspace{-5mm}\end{proof}

\begin{Corollary} \label{cor: 1D non-Jac}
Let $\widetilde{Z}$ be arising from an $\widetilde{X}$ of type \hyperref[Tab1D]{$1D$} and assume that $h^0(\widetilde{Z},T_{\widetilde{Z}}) \neq 0$. Then, $p = 5$, $\widetilde{Z}$ is unique up to isomorphism, has one multiple fiber $5 {\rm II}^*$, and $\Aut_{\widetilde{Z}}^0 \cong \mu_5$. 
\end{Corollary}

\begin{proof}
Everything except the uniqueness follows by combining Corollary \ref{cor: approach for classification non-jacobian} with Proposition \ref{prop: 1D}. To show that $\widetilde{Z}$ is unique up to isomorphism, it suffices to observe that all points $\widetilde{P} \in C^0$ where $C \in |-K_{\widetilde{X}}|$ is the curve of type ${\rm II}^*$ are conjugate under $\Aut(\widetilde{X})$. This follows from our description of the $\mathbb{G}_m$-action on the Weierstra{\ss} model in Table \ref{Table four families}: The points on $C^0$ which do not lie on the zero section are of the form $[1:0:x:y]$ with $x,y \neq 0, y^2 = x^3$, and $\mathbb{G}_m$ sends such a point to $[1:0:\lambda^{-10}x: \lambda^{-15}y]$, so all such points are in the same $\mathbb{G}_m$-orbit.
\end{proof}

\begin{Discussion} \label{Discussion: 1D geometry and curve configs}
By \cite[Theorem 4.1.]{ExtremalCharpII} and \cite{ExtremalChar0}, $\widetilde{Y}$ contains singular fibers ${\rm II}^*$ and ${\rm II}$, the \linebreak Mordell--Weil group of $\widetilde{Y}$ is trivial by \cite{OguisoShioda}, and thus there are no other $(-1)$-curves on $\widetilde{Y}$ besides the zero-section. By the computation in the proof of Proposition \ref{prop: 1D}(\ref{prop 1D case p=5 II*}) $\widetilde{P}$ lies on this $(-1)$-curve as well. So, $\widetilde{Z}$ contains configuration ${\rm II}^*$ and by Lemma \ref{lemma (-2)curves on Ztilde} these are all $(-2)$-curves on $\widetilde{Z}$. The situation is summarized in the following Figure \ref{figure 1D, configs jaco, non-jaco} and Corollary \ref{cor: 1D (-2)configs}.

\begin{table}[H]
\begin{adjustbox}{center}
$
\begin{array}{ccccc}
 \begin{array}{c} \addstackgap[2pt]{\includegraphics[width=0.22\textwidth]{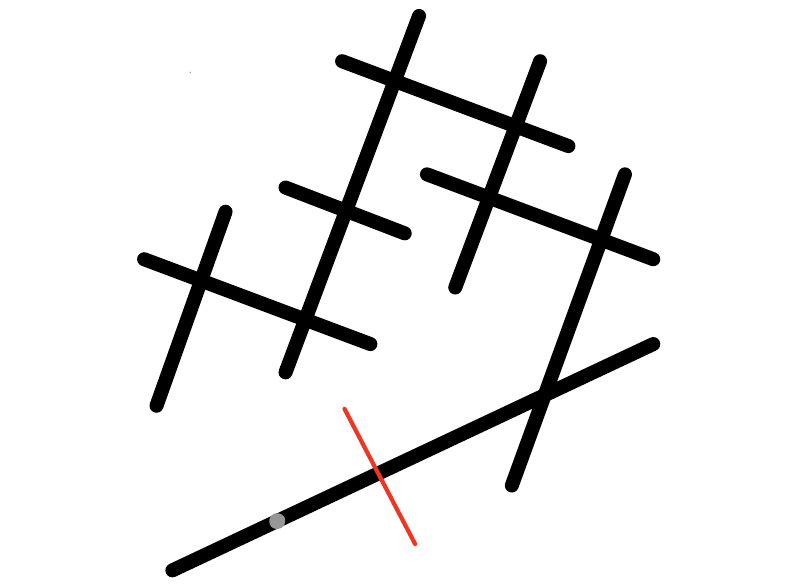} }\end{array}
 & \rightarrow
  & \begin{array}{c}\addstackgap[2pt]{\includegraphics[width=0.22\textwidth]{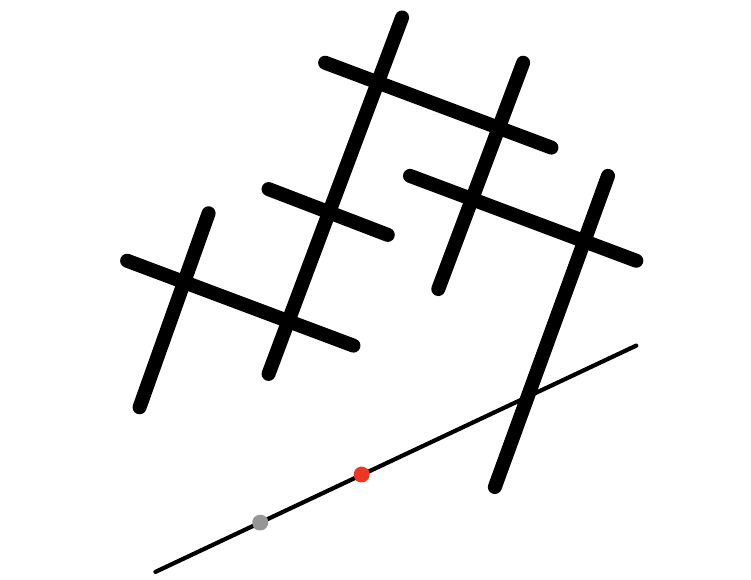} }\end{array}
  & \leftarrow
  & \begin{array}{c}\addstackgap[2pt]{\includegraphics[width=0.22\textwidth]{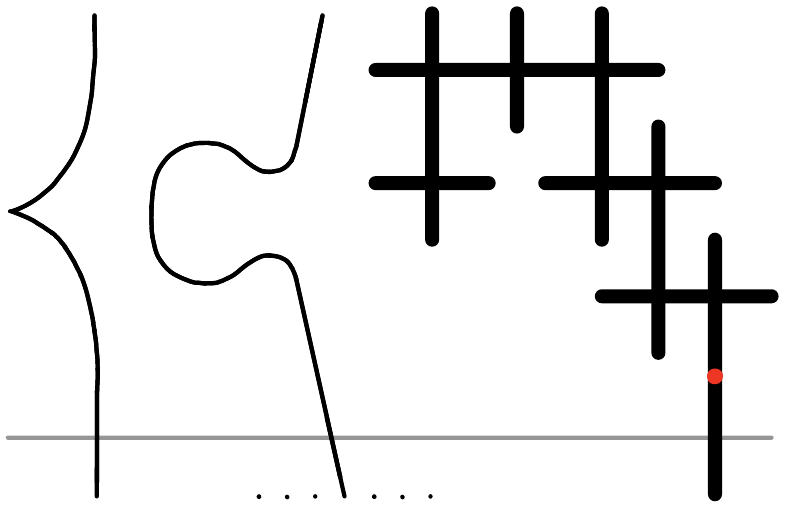} }\end{array}
  \\
  &
  & \downarrow
  &
  & \downarrow
  \\
  &
  & \begin{array}{c}\addstackgap[2pt]{\includegraphics[width=0.22\textwidth]{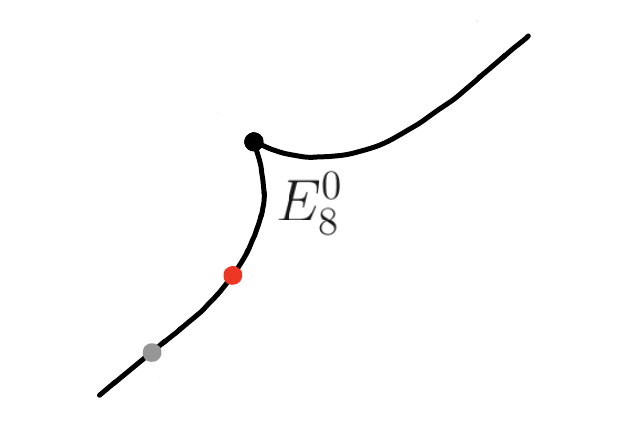}} \end{array}
  & \leftarrow
  & \begin{array}{c}\addstackgap[2pt]{ \includegraphics[width=0.22\textwidth]{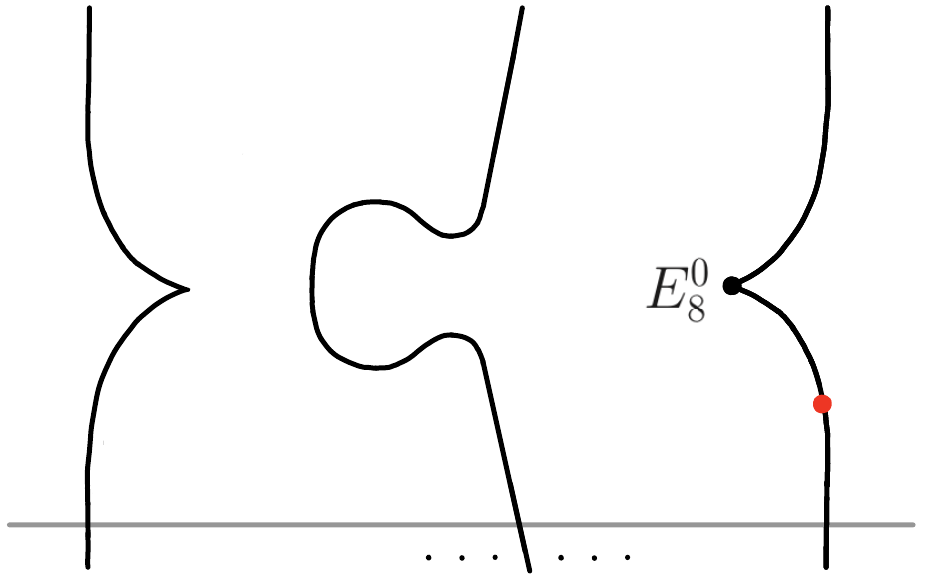}}  \end{array}
\end{array}
$
\captionof{figure}{Non-Jacobian and Jacobian fibrations with global vector fields originating from $\widetilde{X}$ of type \hyperref[Tab1D]{$1D$} ($p=5$)} \label{figure 1D, configs jaco, non-jaco}
\end{adjustbox}
\end{table}

\end{Discussion}

\begin{Corollary} \label{cor: 1D (-2)configs}
The $\widetilde{Z}$ of Corollary \ref{cor: 1D non-Jac} contains nine $(-2)$-curves with dual graph of type $\widetilde{E}_8$ forming configuration ${\rm II}^*$. Moreover, ${\rm II}^*$ is the unique multiple fiber and of multiplicity $m=5$.
\end{Corollary}

\subsection{Case \hyperref[Tab1E]{$1E$}} \label{subsection 1E char3 revised}
This case exists only if $\Char(k)=p=3$.

\begin{Proposition} \label{prop: 1E char3 fixed points revised}
Let $\widetilde X$ of type \hyperref[Tab1E]{$1E$}. Then, $\widetilde P$ is admissible if and only if $C$ is of type ${\rm II}$ and $\widetilde P$ is one of the two intersections of pairs of $(-1)$-curves in $\widetilde X$. Moreover, then $({\rm Stab}_{\Aut_{\widetilde X}^0}(\widetilde P))^0 \cong \mu_3$ and $m=3$.
\end{Proposition}

\begin{proof}
By Proposition \ref{prop: char3 equations and liftable actions}, $X$ is given by
$y^2=x^3+stx^2+t^6$, and the action of $\Aut_{\widetilde X}^0\cong \mu_3$ on $X$ is given by
$[s:t:x:y]\mapsto [s:\lambda t:\lambda x:y]$. To find admissible $P \in X$ (according to Strategy \ref{strategy of proof non-Jaco}), we distinguish the following cases:
\begin{enumerate}[leftmargin=0.8cm]
 \item[(a)]
If $s\neq 0$, we may set $s=1$. Then the action is $[1:t:x:y]\mapsto [1:\lambda t:\lambda x:y]$. Thus a point in this chart has non-trivial connected stabilizer if and only if $t=x=0$. But then the equation gives $y=0$, so we obtain the point $[1:0:0:0]$, which is the $D_7$-singularity. Hence this point cannot be the image of $\widetilde P$.

\item[(b)]
It remains to consider the chart $s=0$. Away from the base point, we may set $t=1$. After rescaling, the action becomes $[0:1:x:y]\mapsto [0:1:\lambda^{-1}x:y]$. Hence the connected stabilizer is non-trivial precisely if $x=0$. The equation of $X$ then gives $y^2=1$, so the only possible points are $P_+=[0:1:0:1]$ and $P_-=[0:1:0:-1]$. For both of these points, the connected stabilizer $({\rm Stab}_{\mu_3}(P_\pm))^0$ is all of $\mu_3$.

The anti-canonical curve containing these points is $C=X\cap\{s=0\}$. It is the cuspidal curve $y^2=x^3+1=(x+1)^3$ with cusp at $[0:1:-1:0]$, whereas the neutral element of $C^0$ is the point where the strict transform of $C$ meets the exceptional curve over the base point of $|-K_{\widetilde X}|$. Thus $P_+$ and $P_-$ lie in the smooth locus of $C$ and are distinct from the neutral element. Therefore $C^0\cong \mathbb G_a$, and since $\Char(k)=3$, both points have exact \mbox{order $3$.}
\end{enumerate}

It remains to identify the position of these points on the configuration of negative curves. The $(-1)$-curves on $\widetilde{X}$ are visible on the anti-canonical model. Indeed, the divisor $\{x=0\}$ on $X$ splits as the union of the two horizontal curves $S_\pm=\{x=0,\ y=\pm t^3\}$, and the divisor $\{x+st=0\}$ splits into the two horizontal curves $T_\pm=\{x=-st,\ y=\pm t^3\}$. Together with $B=\{t=0,\ y^2=x^3\}$, the non-contracted component of the fiber over $t=0$, their strict transforms are the five $(-1)$-curves on the weak del Pezzo surface of type \hyperref[Tab1E]{$1E$}. Moreover, $P_+=S_+\cap T_+$ and $P_-=S_-\cap T_-$. Hence the two possible points $\widetilde P$ are precisely the intersection points of the corresponding pairs of $(-1)$-curves on $\widetilde X$, which finishes the proof.
\end{proof}

 \begin{Corollary} \label{cor: 1E char3 nonJac revised}
Let $\widetilde Z$ be arising from an $\widetilde X$ of type \hyperref[Tab1E]{$1E$} and assume that $h^0(\widetilde Z,T_{\widetilde Z})\neq 0$. Then $\widetilde Z$ is unique up to isomorphism, has one multiple fiber $3{\rm II}$, and $\Aut^0_{\widetilde Z}\cong \mu_3$.
\end{Corollary}

\begin{proof}
Everything except the uniqueness follows by combining Corollary \ref{cor: approach for classification non-jacobian} with Proposition \ref{prop: 1E char3 fixed points revised}. To show that $\widetilde Z$ is unique up to isomorphism, it suffices to observe that the two possible points $\widetilde P$ are conjugate under $\Aut(\widetilde X)$. Indeed, in the proof of Proposition \ref{prop: 1E char3 fixed points revised} we saw that on the anti-canonical model $X$ they correspond to the two points $[0:1:0:1]$ and $[0:1:0:-1]$ on the curve $C=X\cap\{s=0\}$. The involution
$[s:t:x:y]\mapsto [s:t:x:-y]$
preserves the equation $y^2=x^3+stx^2+t^6$ and interchanges these two points. Since it preserves the unique singular point of $X$, it lifts to an automorphism of the minimal resolution $\widetilde X$. Hence the blow-ups of $\widetilde X$ in the two possible points $\widetilde P$ are isomorphic.
\end{proof}

\begin{Discussion} \label{discussion: 1E curves char3 revised}
The Jacobian surface $\widetilde Y$ associated with $\widetilde X$ has singular fibers ${\rm I}_3^*$ and ${\rm II}$. Indeed, the discriminant of $y^2=x^3+stx^2+t^6$ is $-s^3t^9$, hence $\widetilde{Y}$ is elliptic and there are two singular fibers. Tate's algorithm gives a fiber of type ${\rm I}_3^*$ over $t=0$, and the fiber over $s=0$ is an irreducible cuspidal fiber of type ${\rm II}$. Thus the reducible fiber contributes the root lattice $D_7$. By \cite{OguisoShioda}, the Mordell--Weil group of the Jacobian surface has rank $1$ and no torsion. In particular, the four sections visible on the anti-canonical model as $S_\pm$ and $T_\pm$ are only four among infinitely many sections of $\widetilde Y$.

For the non-Jacobian surface $\widetilde Z$ of Corollary \ref{cor: 1E char3 nonJac revised}, we blow up one of the two points $P_\pm=S_\pm\cap T_\pm$ and by the uniqueness statement in Corollary \ref{cor: 1E char3 nonJac revised}, we may choose either of them. Then the strict transforms of the two corresponding $(-1)$-curves become $(-2)$-curves on $\widetilde Z$. Together with the old $D_7$-configuration, they form a fiber of type ${\rm II}^*$. The strict transform of the cuspidal curve $X \cap \{s=0\}$ is the reduced multiple fiber, so the multiple fiber is $3{\rm II}$. The rank bound of Lemma \ref{lemma (-2)curves on Ztilde} is attained because of the ${\rm II}^*$-fiber, hence there are no further reducible fibers.

The situation is summarized in the following Figure \ref{figure 1E configs jaco non-jaco} and Corollary \ref{cor: 1E (-2)configs char3}.
Using the known multiplicities of the components in the ${\rm II}^*$-fiber, we observe that on $\widetilde Z$ the red exceptional curve is a $3$-section, hence intersects the singular fibers correctly. Similarly, the other thin black curves are $3$-sections.

\smallskip \noindent
 \emph{This $\widetilde{Z}$ is not ``new'':} 
 If in the picture for $\widetilde Z$ in Figure \ref{figure 1E configs jaco non-jaco}, we contract the thin black $3$-section that does not meet any other visible section, we get another weak del Pezzo surface of degree $1$. The configuration of $(-2)$-curves on this surface is of type $E_6+A_2$, and the surface still has global vector fields. By the classification of weak del Pezzo surfaces with global vector fields in \cite[Table 6]{WeakDelPezzoGlobalVectorFields}, this is the unique surface of type \hyperref[Tab1B]{$1B$}. Thus $\widetilde Z$ is also obtained from type \hyperref[Tab1B]{$1B$}. By Corollary \ref{cor: 1B non-Jac}, the latter construction gives a unique non-Jacobian surface in characteristic $3$ and we can identify those.

\begin{table}[H]
\begin{adjustbox}{center}
$
\begin{array}{ccccc}
 \begin{array}{c}
 \addstackgap[2pt]{\includegraphics[width=0.18\textwidth]{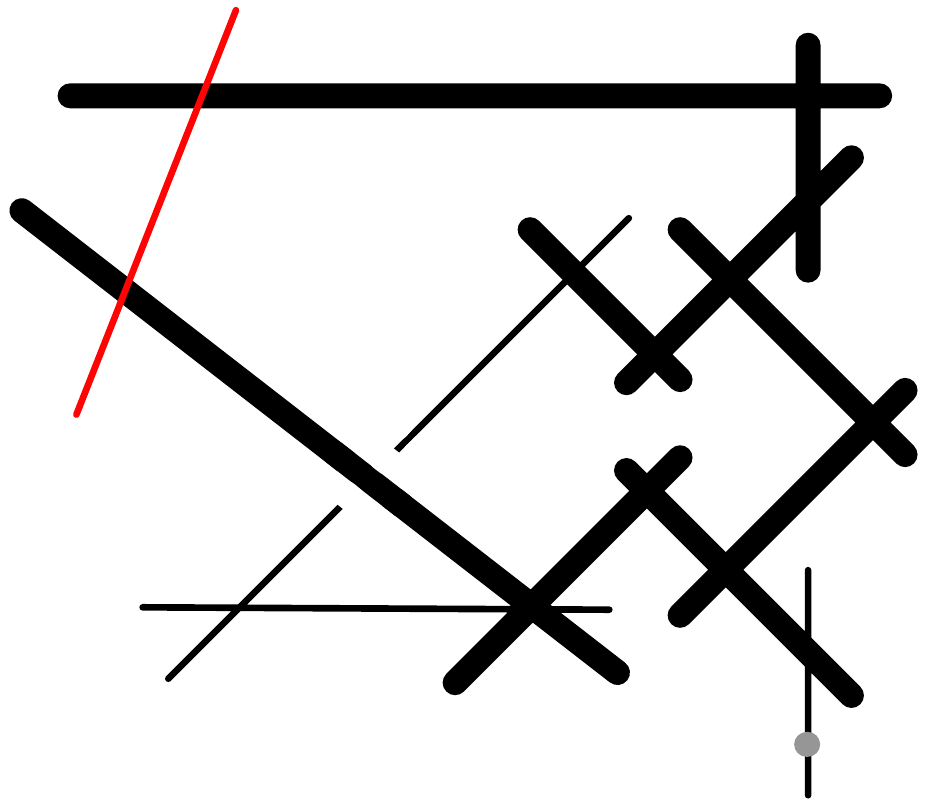}}
 \end{array}
 & \rightarrow
 & \begin{array}{c}
 \addstackgap[2pt]{\includegraphics[width=0.18\textwidth]{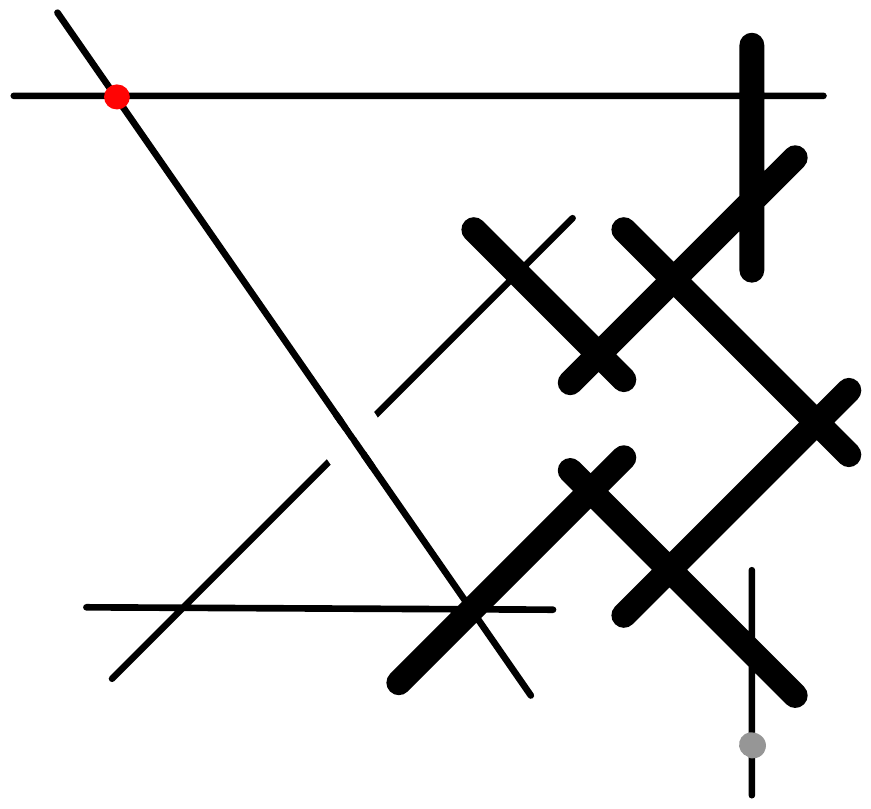}}
 \end{array}
 & \leftarrow
 & \begin{array}{c}
 \addstackgap[2pt]{\includegraphics[width=0.19\textwidth]{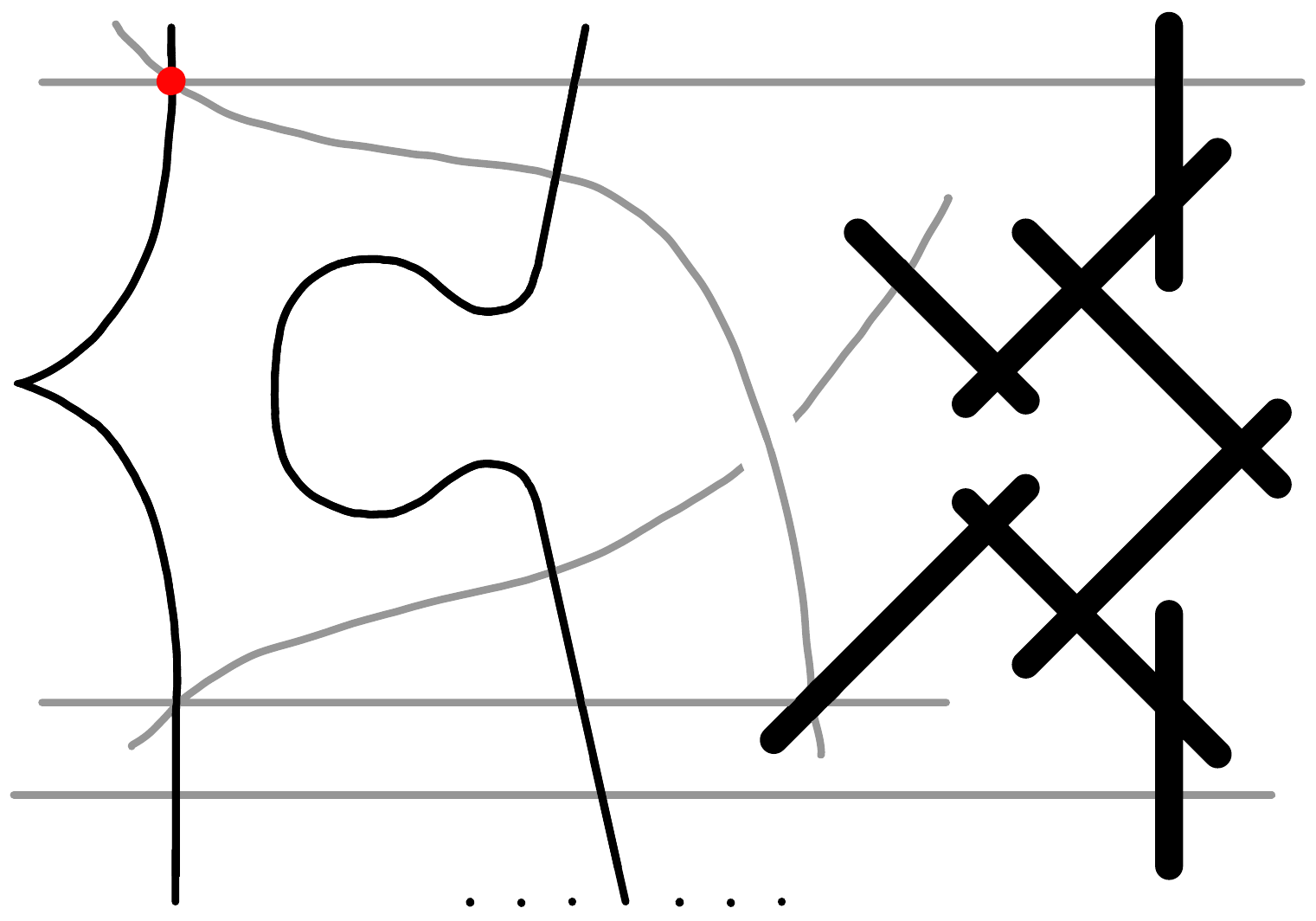}}
 \end{array}
 \\
 &
 & \downarrow
 &
 & \downarrow
 \\
 &
 & \begin{array}{c}
 \addstackgap[2pt]{\includegraphics[width=0.17\textwidth]{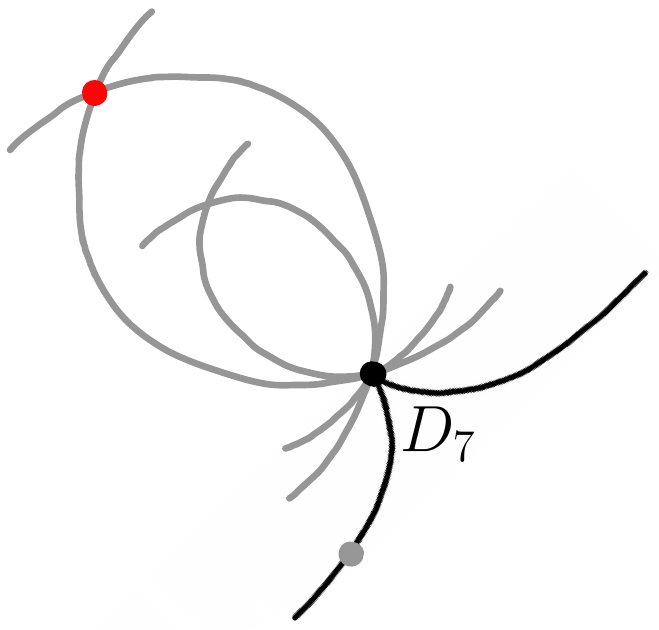}}
 \end{array}
 & \leftarrow
 & \begin{array}{c}
 \addstackgap[2pt]{\includegraphics[width=0.18\textwidth]{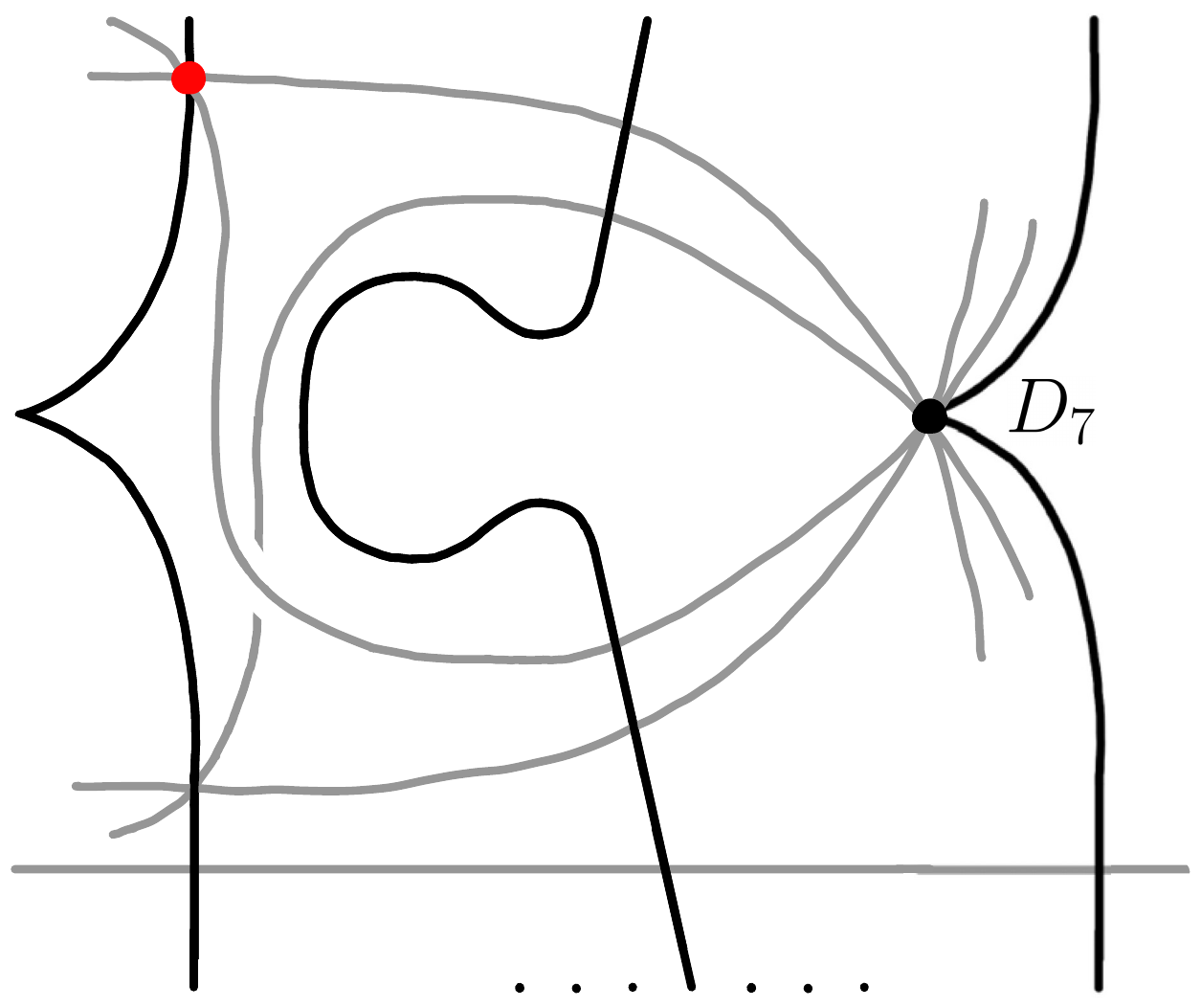}}
 \end{array}
\end{array}
$
\captionof{figure}{Non-Jacobian and Jacobian fibrations with global vector fields originating from $\widetilde X$ of type \hyperref[Tab1E]{$1E$}}
\label{figure 1E configs jaco non-jaco}
\end{adjustbox}
\end{table}

\end{Discussion}

\begin{Corollary} \label{cor: 1E (-2)configs char3}
The $\widetilde Z$ of Corollary \ref{cor: 1E char3 nonJac revised} contains nine $(-2)$-curves with dual graph of type $\widetilde E_8$ forming a reducible fiber of type ${\rm II}^*$. Moreover, $\widetilde Z$ is isomorphic to the unique surface in Corollaries \ref{cor: 1B non-Jac} and \ref{cor: 1B (-2)configs} in characteristic $3$, obtained from a weak del Pezzo surface of type \hyperref[Tab1B]{$1B$}.
\end{Corollary}

\subsection{Case \hyperref[Tab1F]{$1F$}} \label{subsection 1F char3 revised}
This case exists only if $\Char(k)=p=3$.

\begin{Proposition} \label{prop: 1F char3 fixed points revised}
Let $\widetilde X$ be of type \hyperref[Tab1F]{$1F$}. Then, $\widetilde{P}$ is admissible if and only if one of the following holds.
\begin{enumerate}
\item\label{prop 1F s0 revised} $C$ is of type ${\rm II}$ and $\widetilde P$ is one of the two intersection points of multiplicity $2$ of pairs of $(-1)$-curves on $\widetilde X$.
\item\label{prop 1F t0 special revised} $C$ is of type ${\rm III}^*$ and $\widetilde P$ lies on the unique $(-1)$-curve contained in this anti-canonical member and is one of the two intersection points with another $(-1)$-curve.
\item\label{prop 1F t0 general revised} $C$ is of type ${\rm III}^*$ and $\widetilde P$ lies on the unique $(-1)$-curve contained in this anti-canonical member and on no other $(-1)$-curve.
\end{enumerate}
Moreover, then $({\rm Stab}_{\Aut_{\widetilde X}^0}(\widetilde P))^0\cong \mu_3$ and $m=3$.
\end{Proposition}

\begin{proof}
By Proposition $X$ is given by
$y^2=x^3+st^3x+t^6$, and the action of $\Aut_{\widetilde X}^0\cong \mu_3$ on $X$ is given by
$[s:t:x:y]\mapsto [s:\lambda t:x:y]$. To find admissible $P \in X$ (according to Strategy \ref{strategy of proof non-Jaco}), we distinguish the following cases:

\begin{enumerate}[leftmargin=0.8cm]
\item[(a)]
If $s\neq 0$, we may set $s=1$. Then the action is $[1:t:x:y]\mapsto [1:\lambda t:x:y]$. Thus a point in this chart has non-trivial connected stabilizer if and only if $t=0$. The point $[1:0:0:0]$ is the $E_7^0$-singularity and hence cannot be the image of $\widetilde P$. The remaining points on this fiber are the points $[1:0:x:y]$ with $x,y\neq 0$ and $y^2=x^3$. They lie on the smooth locus of the non-contracted component of the anti-canonical member over $t=0$. On $\widetilde X$, this component is the unique $(-1)$-curve contained in this member, and after blowing up the base point of $|-K_{\widetilde X}|$ it becomes the identity component of the fiber of type ${\rm III}^*$. Its identity component is isomorphic to $\mathbb G_a$, so every non-zero point has exact order $3$. The connected stabilizer of each such point is all of $\mu_3$.

\item[(b)]
It remains to consider the chart $s=0$. Away from the base point, we may set $t=1$. After rescaling, the action becomes $[0:1:x:y]\mapsto [0:1:\lambda x:y]$. Hence the connected stabilizer is non-trivial precisely if $x=0$. The equation of $X$ then gives $y^2=1$, so the only possible points are $P_+=[0:1:0:1]$ and $P_-=[0:1:0:-1]$. For both of these points, the connected stabilizer $({\rm Stab}_{\mu_3}(P_\pm))^0$ is all of $\mu_3$.

The anti-canonical curve containing these points is $C=X\cap\{s=0\}$. It is the cuspidal curve $y^2=x^3+1=(x+1)^3$. Its cusp is $[0:1:-1:0]$, whereas the neutral element of $C^0$ is the point where the strict transform of $C$ meets the exceptional curve over the base point of $|-K_{\widetilde X}|$. Thus $P_+$ and $P_-$ lie in the smooth locus of $C$ and are distinct from the neutral element. Therefore $C^0\cong \mathbb G_a$, and since $\Char(k)=3$, both points have exact \mbox{order $3$.}
\end{enumerate}

It remains to identify the position of these points on the configuration of negative curves. The $(-1)$-curves on $\widetilde X$ are visible on the anti-canonical model. Indeed, the divisor $\{x=0\}$ on $X$ splits as the union of the two horizontal curves $S_\pm=\{x=0,\ y=\pm t^3\}$, and the divisor $\{x-s^2=0\}$ splits into the two horizontal curves
$R_+=\{x=s^2,\ y=t^3-s^3\}$ and $R_-=\{x=s^2,\ y=s^3-t^3\}$. Together with $B=\{t=0,\ y^2=x^3\}$, the non-contracted component of the fiber over $t=0$, their strict transforms are the five $(-1)$-curves on the weak del Pezzo surface of type \hyperref[Tab1F]{$1F$}.

Moreover, $P_+=S_+\cap R_+$ and $P_-=S_-\cap R_-$, these intersections have multiplicity $2$, and hence the two points on the curve of type ${\rm II}$ are exactly the two points described in (\ref{prop 1F s0 revised}).
On the other hand, the two special points on the $(-1)$-curve $B$ are
$Q_+=B\cap R_+=[1:0:1:-1]$ and $Q_-=B\cap R_-=[1:0:1:1]$. Thus $Q_+$ and $Q_-$ are exactly the points described in (\ref{prop 1F t0 special revised}). Finally, by the above list of the five $(-1)$-curves, every other point of $B^0$ lies on no other $(-1)$-curve, giving precisely the points described in (\ref{prop 1F t0 general revised}).
\end{proof}

\begin{Corollary} \label{cor: 1F char3 nonJac revised}
Let $\widetilde Z$ arsing from an $\widetilde X$ of type \hyperref[Tab1F]{$1F$} and assume that $h^0(\widetilde Z,T_{\widetilde Z})\neq 0$. Then $\Aut^0_{\widetilde Z}\cong \mu_3$, and the following hold.
\begin{enumerate}
\item[(1)] \label{cor 1F uniqueness(1) solo} If $\widetilde Z$ has one multiple fiber $3{\rm II}$, then $\widetilde Z$ is unique up to isomorphism.
\item[(2,3)] \label{cor 1F uniqueness(2,3) 1 dim} Otherwise, the corresponding surfaces form a $1$-dimensional family, and each of them has one multiple fiber $3{\rm III}^*$. More precisely:
\begin{enumerate}
\item[(2)]\label{cor 1F uniquenes(2,3)(2)} If, in the construction of $\widetilde Z$, the point $\widetilde P$ lies on two $(-1)$-curves on $\widetilde X$, then all surfaces obtained in this way are isomorphic.
\item[(3)]\label{cor 1F uniqueness(2,3)(3) 1dim} If, in the construction of $\widetilde Z$, the point $\widetilde P$ lies on no other $(-1)$-curve, then the surfaces obtained in this way form a $1$-dimensional family.
\end{enumerate}
\end{enumerate}
\end{Corollary}

\begin{proof}
 Everything except the moduli statements follows by combining Corollary \ref{cor: approach for classification non-jacobian} with Proposition \ref{prop: 1F char3 fixed points revised}. We treat the three cases of Proposition \ref{prop: 1F char3 fixed points revised} separately and use the explicit Weierstra{\ss} equation for $X$, as well as the images in $X$ of the possible points $\widetilde P$, as described in the proof of that proposition.

\begin{enumerate}
\item[(\ref{prop 1F s0 revised},\ref{prop 1F t0 special revised})] If the image of such a $\widetilde{P}$ in $X$ is $P_\pm=[0:1:0:\pm 1]$ resp. $Q_\pm=B\cap R_\pm=[1:0:1:\mp 1]$, the involution 
$[s:t:x:y]\mapsto [s:t:x:-y]$ preserves the equation $y^2=x^3+st^3x+t^6$, interchanges $P_+$ and $P_-$ resp. $Q_+$ and $Q_-$, and preserves the unique singular point of $X$. Hence it lifts to an automorphism of $\widetilde X$, and the two possible blow-ups in (\ref{prop 1F s0 revised}) resp. (\ref{prop 1F t0 special revised}) are isomorphic.
\item[(\ref{prop 1F t0 general revised})]
Otherwise, we are in case (\ref{prop 1F t0 general revised}), and the possible points $\widetilde{P}$ form an open subset $U = B^0 \setminus \{0, Q_+, Q_-\}$ of the $(-1)$-curve $\widetilde{B} \subseteq \widetilde{X}$, hence a $1$-dimensional family.  Since $\widetilde{X}$ has finite automorphism group by \cite{WeakDelPezzoGlobalVectorFields}, the resulting blow-ups form a $1$-dimensional family as well.
\end{enumerate}
\end{proof}

\begin{Discussion} \label{discussion: 1F curves char3 revised}
The Jacobian surface $\widetilde Y$ associated with $\widetilde X$ has singular fibers ${\rm III}^*$ and ${\rm II}$. Indeed, the discriminant of $y^2=x^3+st^3x+t^6$ is $-s^3t^9$, hence $\widetilde Y$ is elliptic and there are two singular fibers. Tate's algorithm gives a fiber of type ${\rm III}^*$ over $t=0$, and the fiber over $s=0$ is an irreducible cuspidal fiber of type ${\rm II}$. Thus the reducible fiber contributes the root lattice $E_7$. By \cite{OguisoShioda}, the Mordell--Weil group of the Jacobian surface has rank $1$ and no torsion. In particular, the four sections visible on the anti-canonical model as $S_\pm$ and $R_\pm$ are only four among infinitely many sections of $\widetilde Y$.

To determine the number and configuration of $(-2)$-curves on the non-Jacobian surface $\widetilde Z$, we treat the cases (\hyperref[cor 1F uniqueness(1) solo]{1}), (\hyperref[cor 1F uniquenes(2,3)(2)]{2}), and (\hyperref[cor 1F uniqueness(2,3)(3) 1dim]{3}) of Corollary \ref{cor: 1F char3 nonJac revised} separately.

\begin{enumerate}
\item[(1)]\label{discussion 1F multiple II} By the proof of Proposition \ref{prop: 1F char3 fixed points revised}(\ref{prop 1F s0 revised}), the point $\widetilde P$ is one of the two points on $\widetilde X$ where two $(-1)$-curves meet with intersection multiplicity $2$. By Corollary \ref{cor: 1F char3 nonJac revised}(\hyperref[cor 1F uniqueness(1) solo]{1}), we may choose either of them and take the $\widetilde{P}$ corresponding to $P_+=S_+\cap R_+=[0:1:0:1]$ on $X$.

After blowing up this point, the strict transforms of $S_+$ and $R_+$ become $(-2)$-curves. Since $S_+$ and $R_+$ meet with multiplicity $2$ at $\widetilde P$, their strict transforms still meet, now transversely, on $\widetilde Z$. Together with the old $E_7$-configuration, they form a fiber of type ${\rm II}^*$. Moreover, the two curves $S_+$ and $R_+$ meet the type ${\rm II}$ fiber over $s=0$, transversely on $\widetilde X$. Therefore, the blow-up separates their strict transforms from the strict transform of this cuspidal curve, which becomes the reduced multiple fiber. Hence the multiple fiber is $3{\rm II}$. The rank bound of Lemma \ref{lemma (-2)curves on Ztilde} is attained by the ${\rm II}^*$-fiber, so there are no further reducible fibers.

The situation is summarized in the following Figure \ref{figure 1F multiple II configs jaco non-jaco} and Corollary \ref{cor: 1F (-2)configs char3}(\ref{cor 1F configs multiple II}). Using the known multiplicities of the components of Kodaira--N{\'e}ron fibers, we observe that on $\widetilde Z$ all thin smooth curves in the picture are $3$-sections.
\end{enumerate}

\begin{enumerate}
    \item[ ] \emph{This $\widetilde Z$ is not ``new'':} If in the picture for $\widetilde Z$ in Figure \ref{figure 1F multiple II configs jaco non-jaco}, we contract the thin black curve that meets only one $(-2)$-curve transversely, we get another weak del Pezzo surface of degree $1$. The configuration of $(-2)$-curves on this surface is of type $E_6+A_2$, and the surface still has global vector fields. Hence, as in case \hyperref[subsection 1E char3 revised]{$1E$}, by the classification of weak del Pezzo surfaces with global vector fields in \cite[Table 6]{WeakDelPezzoGlobalVectorFields}, this is the unique surface of type \hyperref[Tab1B]{$1B$}. Since by Corollary \ref{cor: 1B non-Jac} this construction gives a unique non-Jacobian surface in characteristic $3$, we can identify this $\widetilde{Z}$ with the one from type \hyperref[Tab1B]{$1B$} in characteristic $3$ (as we did in case \hyperref[Tab1E]{$1E$} above). 
\end{enumerate}

\begin{table}[H]
\begin{adjustbox}{center}
$
\begin{array}{ccccc}
 \begin{array}{c} \addstackgap[2pt]{\includegraphics[width=0.22\textwidth]{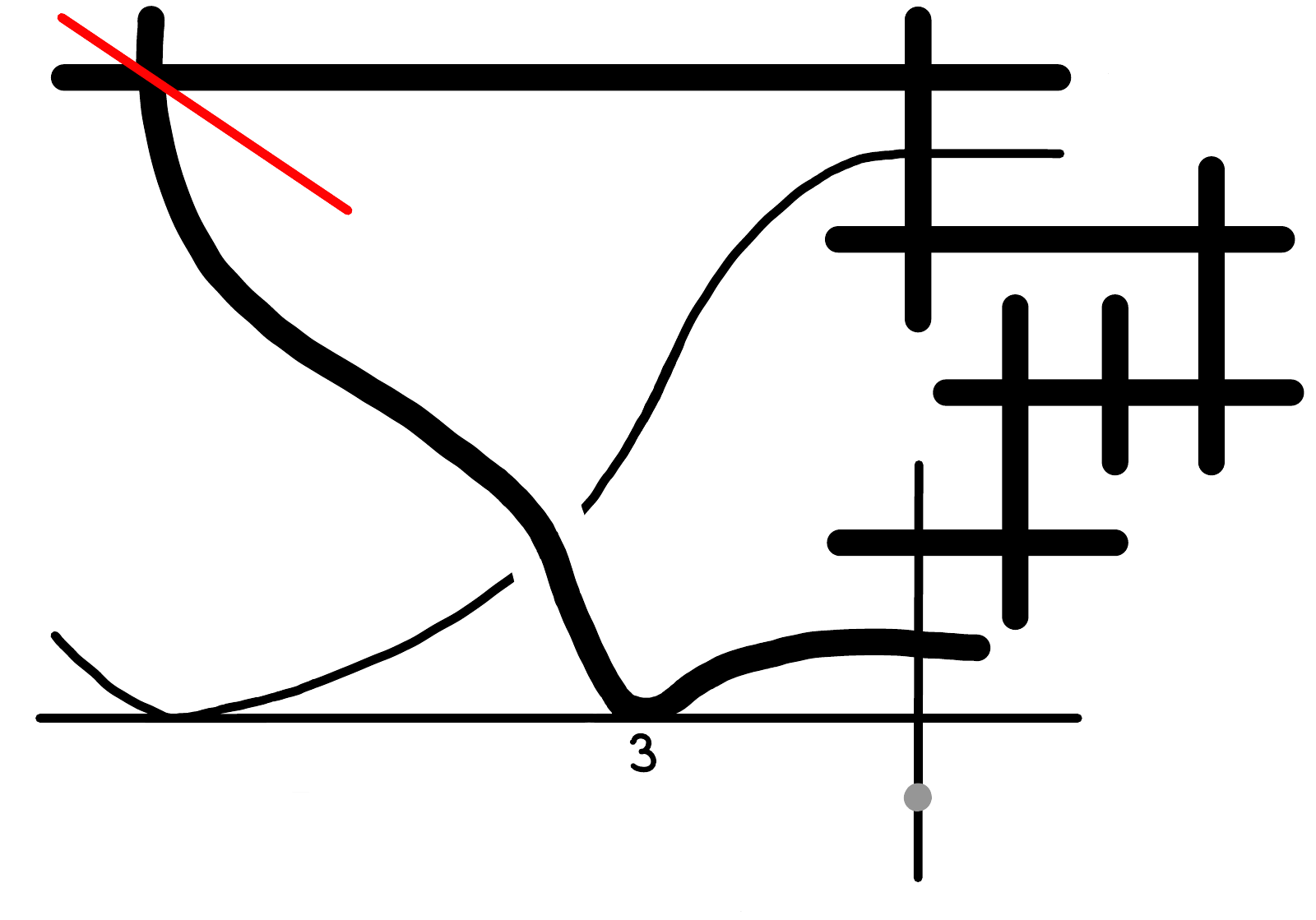} }\end{array}
 & \rightarrow
  & \begin{array}{c}\addstackgap[2pt]{\includegraphics[width=0.22\textwidth]{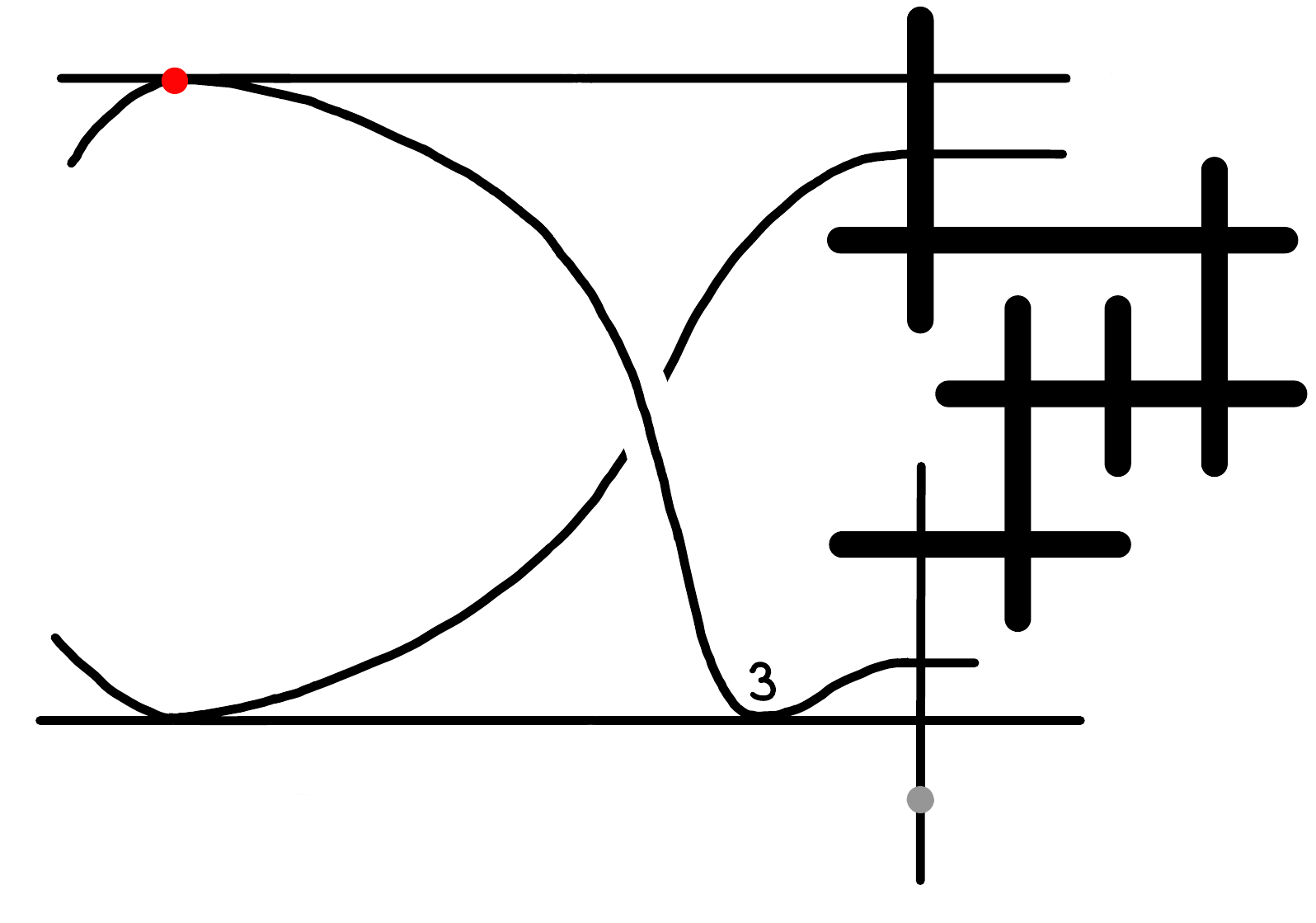} }\end{array}
  & \leftarrow
  & \begin{array}{c}\addstackgap[2pt]{\includegraphics[width=0.21\textwidth]{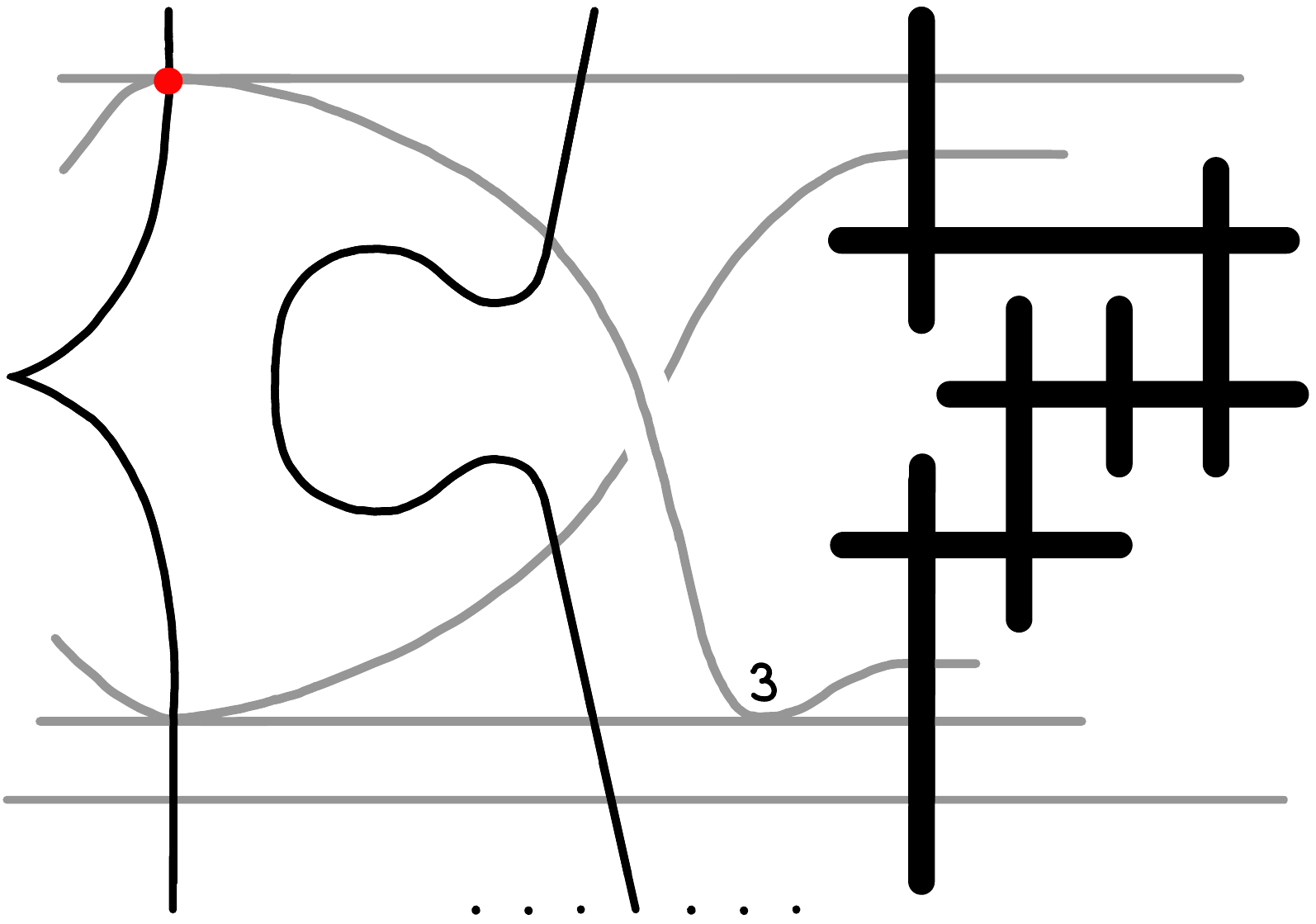} }\end{array}
  \\
  &
  & \downarrow
  &
  & \downarrow
  \\
  &
  & \begin{array}{c}\addstackgap[2pt]{\includegraphics[width=0.18\textwidth]{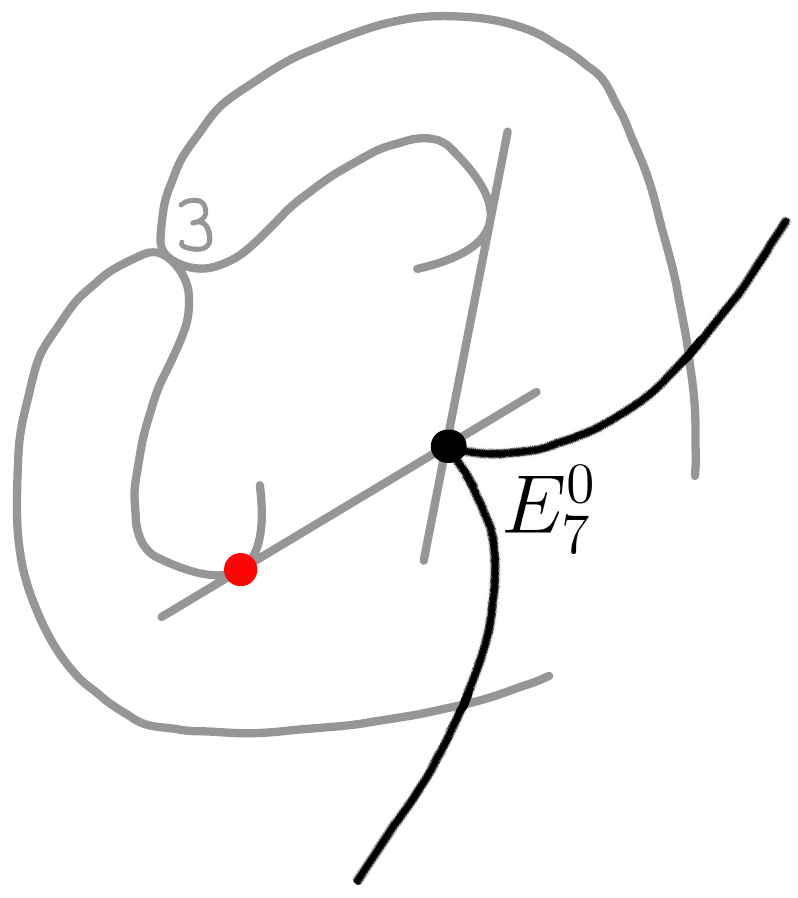}} \end{array}
  & \leftarrow
  & \begin{array}{c}\addstackgap[2pt]{ \includegraphics[width=0.20\textwidth]{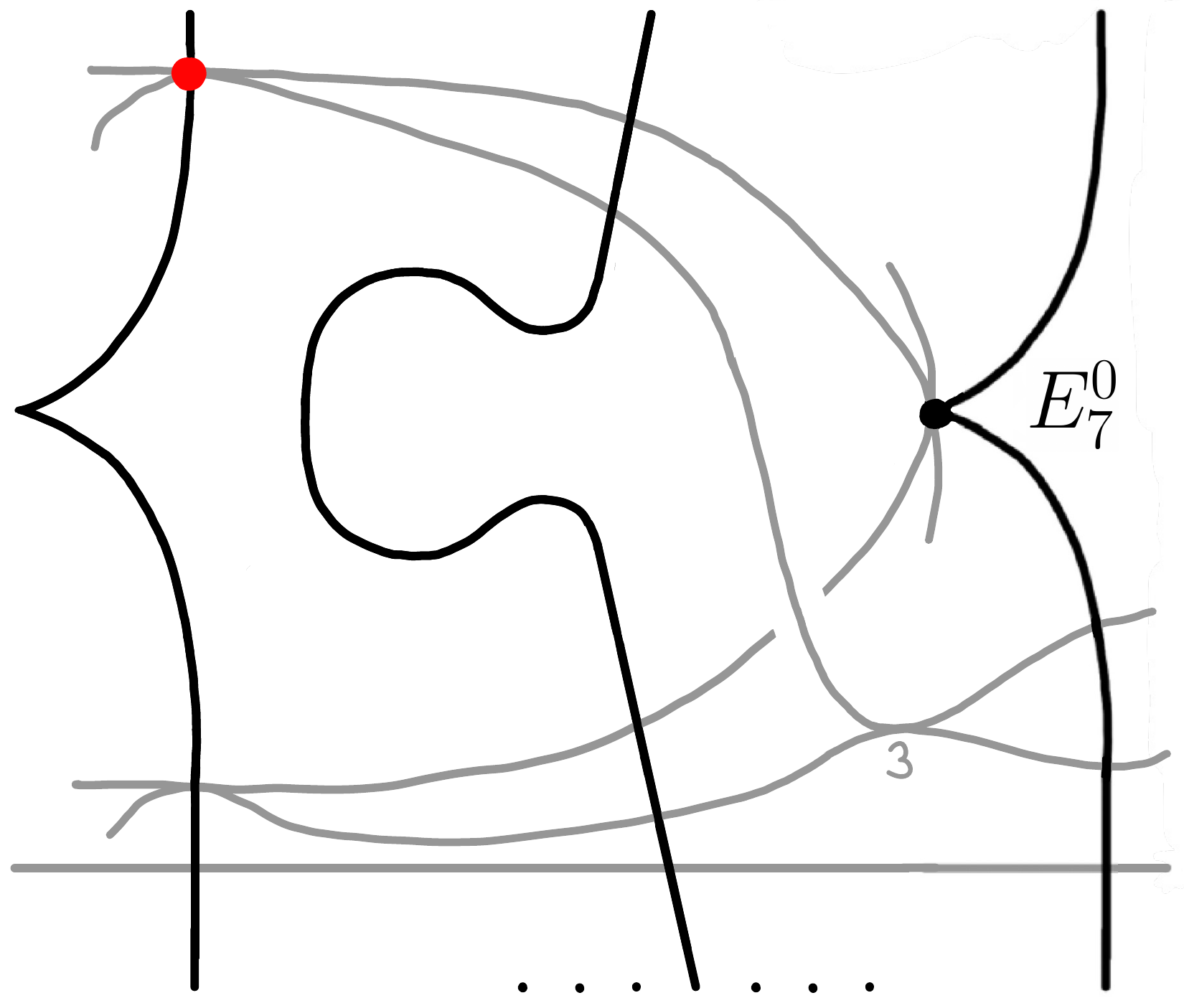}}  \end{array}
\end{array}
$
\captionof{figure}{Non-Jacobian and Jacobian fibrations with global vector fields originating from $\widetilde X$ of type \hyperref[Tab1F]{$1F$}, where the center lies on the type ${\rm II}$ curve}
\label{figure 1F multiple II configs jaco non-jaco}
\end{adjustbox}
\end{table}

\begin{enumerate}
\item[(2)]\label{discussion 1F special IIIstar} By the proof of Proposition \ref{prop: 1F char3 fixed points revised}(\ref{prop 1F t0 special revised}), the point $\widetilde P$ is one of the two points on $\widetilde X$ that lie on a transverse intersection of two $(-1)$-curves in the configuration picture of Table \ref{Table char3 equations and liftable actions}. By Corollary \ref{cor: 1F char3 nonJac revised}(\hyperref[cor 1F uniquenes(2,3)(2)]{2}), we may choose either of them and take the $\widetilde P$ corresponding to $Q_+=B\cap R_+=[1:0:1:-1]$ on $X$.

After blowing up this point, the $(-1)$-curve on $\widetilde X$ through $\widetilde P$ that does not meet a $(-2)$-curve becomes itself a $(-2)$-curve on $\widetilde Z$. It is separated by the exceptional divisor from the ${\rm III}^*$-configuration of $(-2)$-curves on $\widetilde Z$ to which the other $(-1)$-curve through $\widetilde P$ contributes. By the classification of fiber types of rational (quasi-)elliptic fibrations, this newly produced $(-2)$-curve on $\widetilde Z$ that does not meet the ${\rm III}^*$ multiple fiber has to occur in a configuration with other $(-2)$-components on $\widetilde Z$ that were not yet visible as negative curves on $\widetilde X$. Hence, by the rank bound of Lemma \ref{lemma (-2)curves on Ztilde}, the only possibilities for this other reducible fiber are ${\rm I}_2$ and ${\rm III}$. We will determine that it has to be of type ${\rm I}_2$, intersecting the already visible curves as indicated in the figure below, by realizing our $\widetilde Z$ as the blow-up of another weak del Pezzo surface of degree $1$.

The situation is summarized in the following Figure \ref{figure 1F special IIIstar configs jaco non-jaco} and Corollary \ref{cor: 1F (-2)configs char3}(\ref{cor 1F configs special IIIstar}). Using the known multiplicities of the components of Kodaira--N{\'e}ron fibers, we observe that on $\widetilde Z$ all thin smooth curves in the picture are $3$-sections.
\end{enumerate}

\begin{enumerate}
\item[] \emph{This $\widetilde Z$ is not ``new'':}
If in the above figure, we look at $\widetilde X$ and the position of the blown-up red point $\widetilde P$, we see that if we contract on the resulting $\widetilde Z$ the green $(-1)$-curve, we get another weak del Pezzo surface of degree $1$ that has a configuration of $(-2)$-curves of type $E_7+A_1$ and still has global vector fields by Blanchard's Lemma. Hence, by the classification of weak del Pezzo surfaces with global vector fields in \cite[Table 6]{WeakDelPezzoGlobalVectorFields}, this is the unique surface of type \hyperref[Tab1C]{$1C$}. Thus $\widetilde Z$ is also obtained from case \hyperref[Tab1C]{$1C$}. By Corollary \ref{cor: 1C non-Jac}, the latter construction gives a unique non-Jacobian surface in characteristic $3$, which by Corollary \ref{cor: 1C (-2)configs} and Discussion \ref{Discussion: 1C geometry and curve configs} has ${\rm I}_2$ as its other reducible fiber. Hence the above configuration of reducible fibers for $\widetilde Z$ in Figure \ref{figure 1F special IIIstar configs jaco non-jaco} is correct.
\end{enumerate}

\begin{table}[H]
\begin{adjustbox}{center}
$
\begin{array}{ccccc}
 \begin{array}{c} \addstackgap[2pt]{\includegraphics[width=0.22\textwidth]{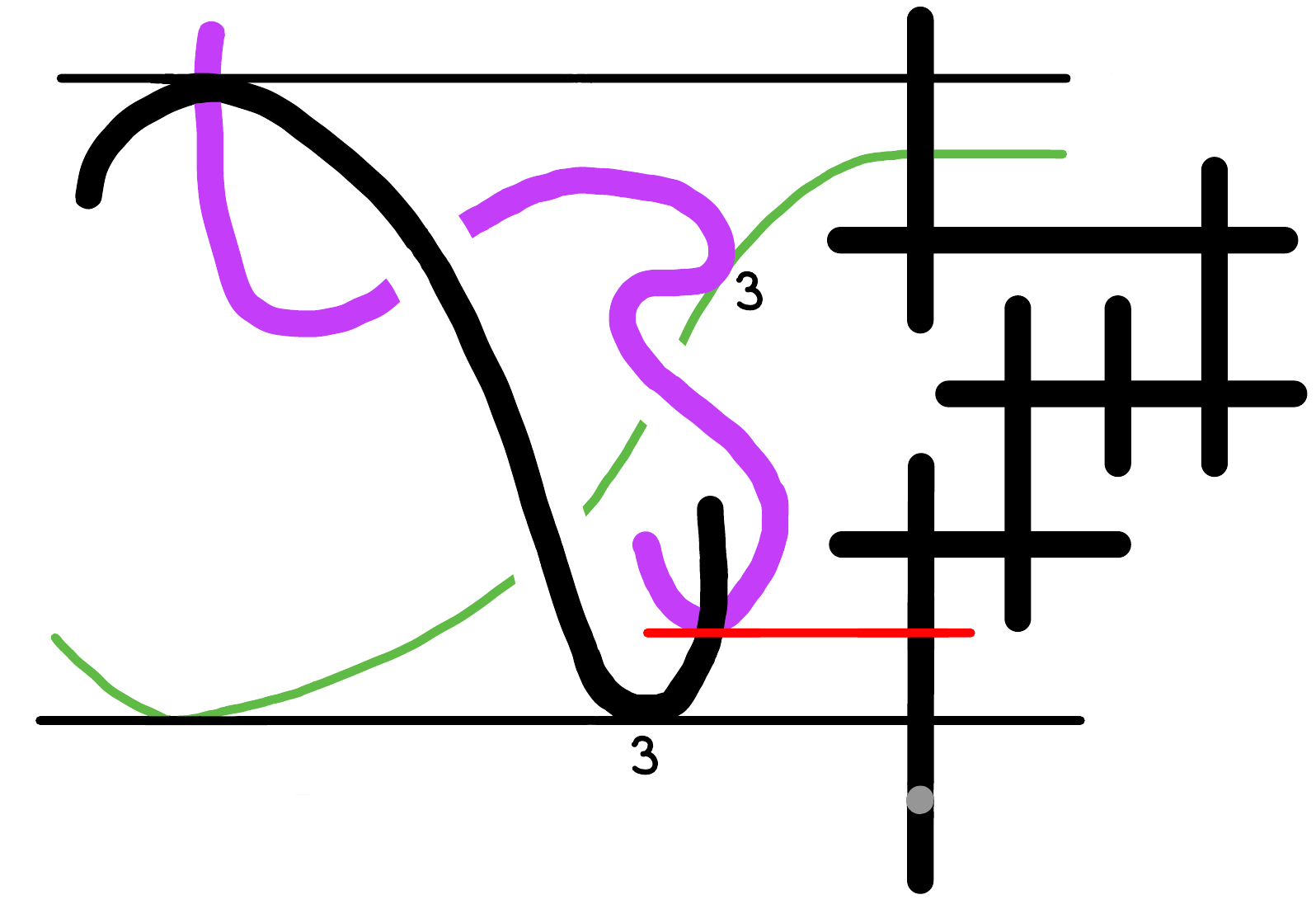} }\end{array}
 & \rightarrow
  & \begin{array}{c}\addstackgap[2pt]{\includegraphics[width=0.22\textwidth]{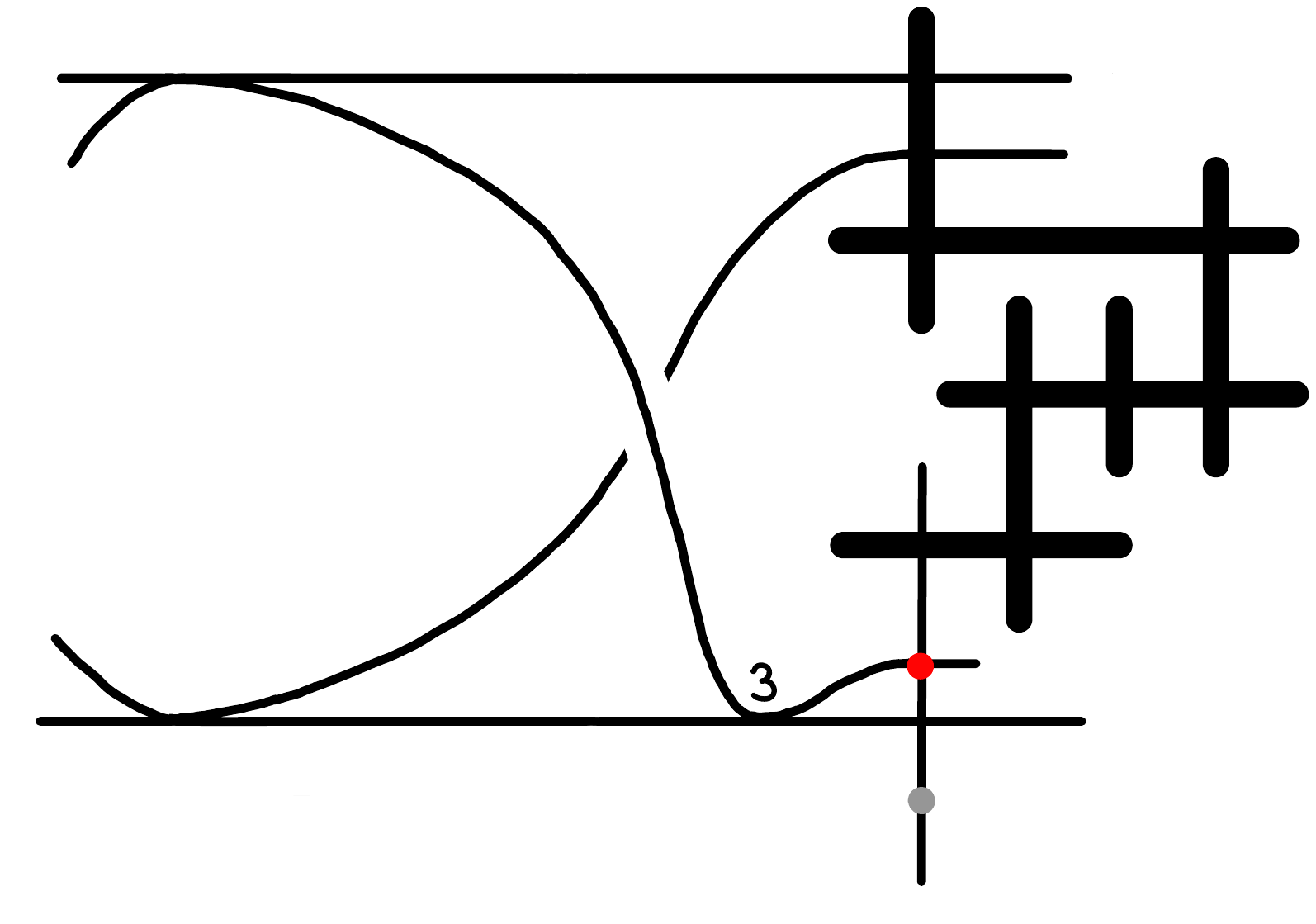} }\end{array}
  & \leftarrow
  & \begin{array}{c}\addstackgap[2pt]{\includegraphics[width=0.21\textwidth]{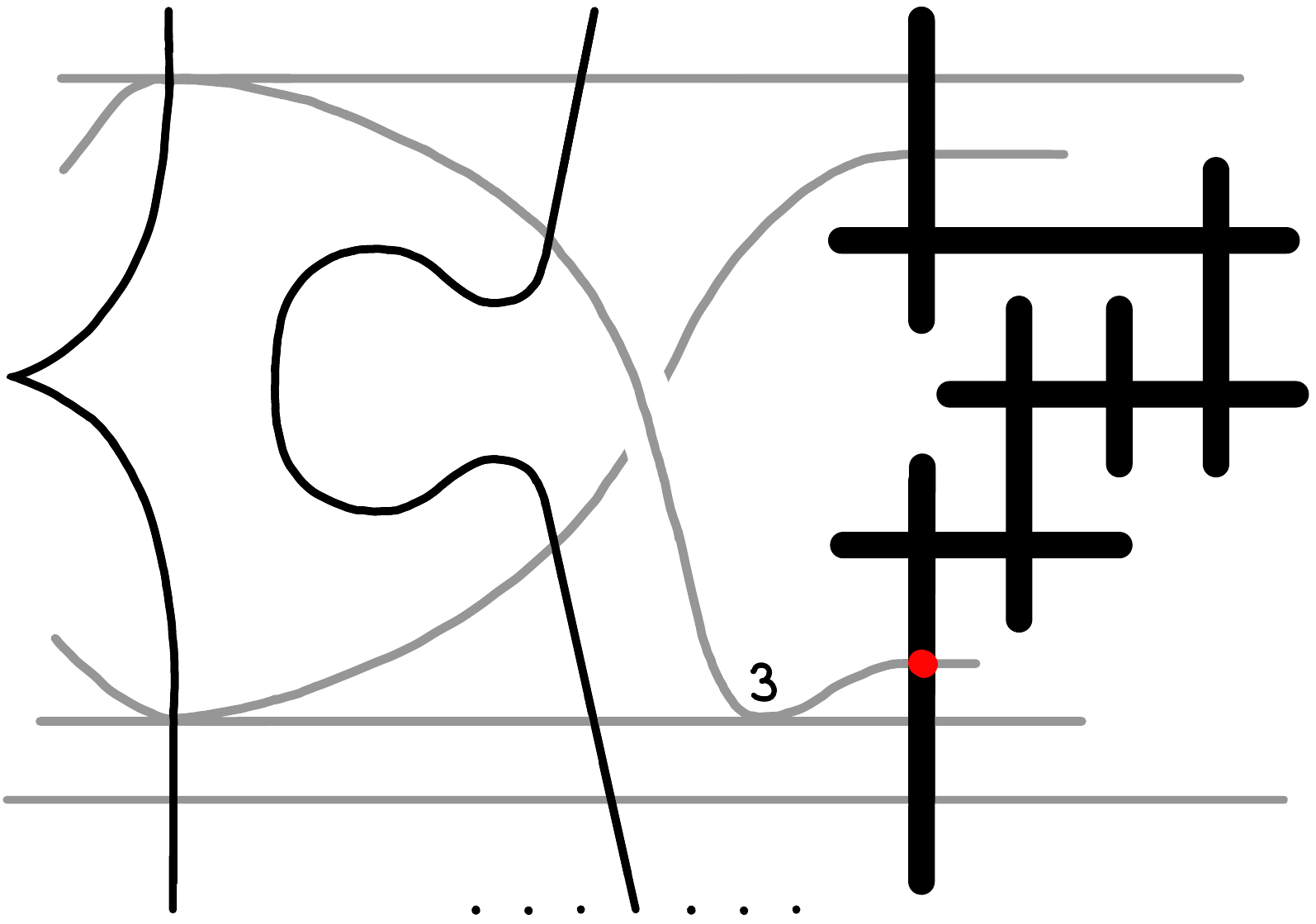} }\end{array}
  \\
  &
  & \downarrow
  &
  & \downarrow
  \\
  &
  & \begin{array}{c}\addstackgap[2pt]{\includegraphics[width=0.18\textwidth]{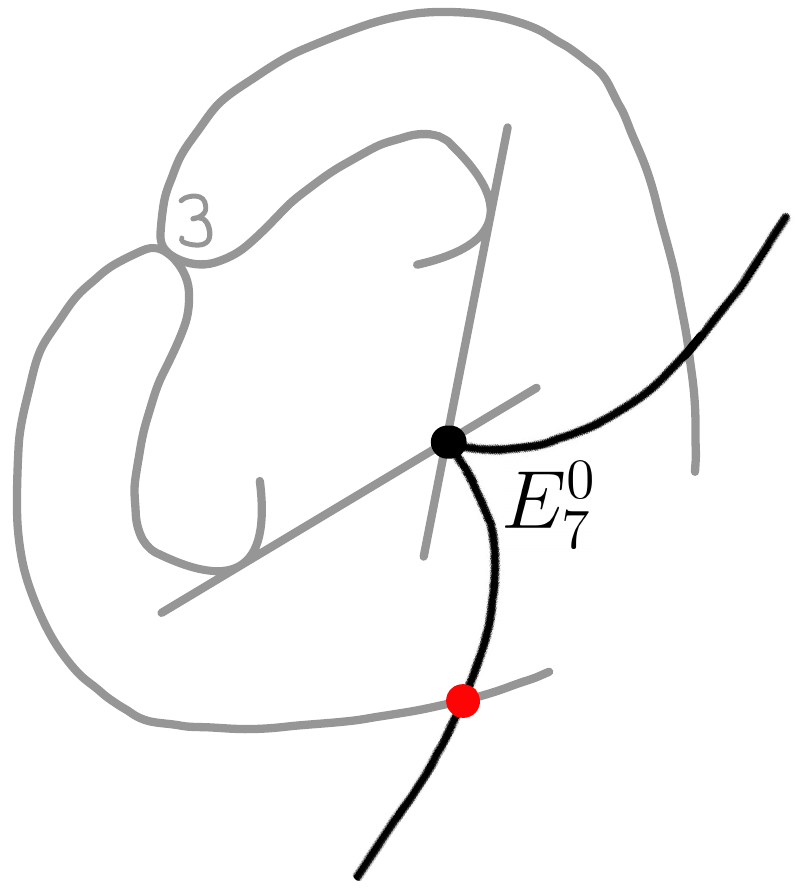}} \end{array}
  & \leftarrow
  & \begin{array}{c}\addstackgap[2pt]{ \includegraphics[width=0.20\textwidth]{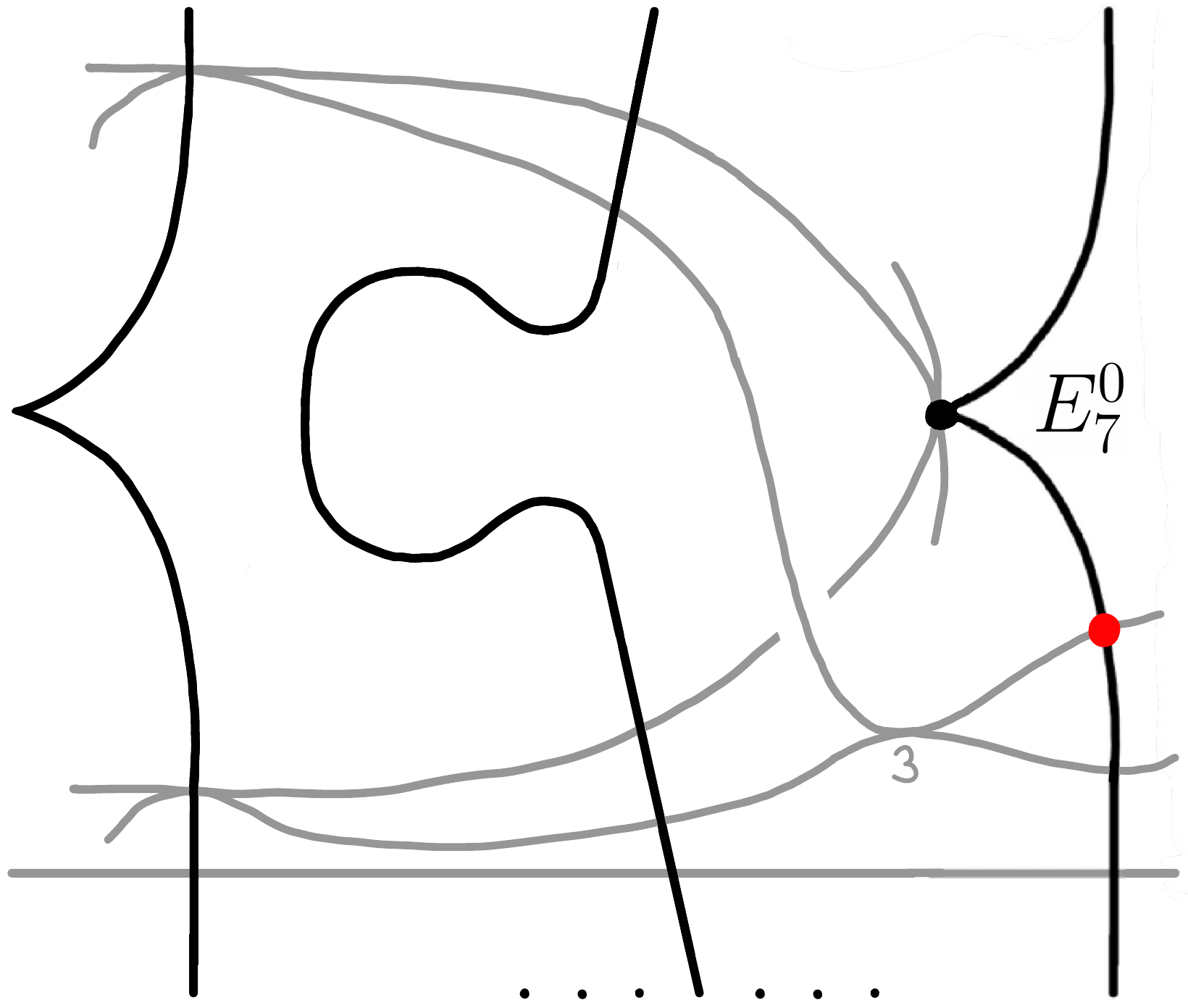}}  \end{array}
\end{array}
$
\captionof{figure}{Non-Jacobian and Jacobian fibrations with global vector fields originating from $\widetilde X$ of type \hyperref[Tab1F]{$1F$}, where the center is one of the two special points on the $(-1)$-component of the ${\rm III}^*$-member}
\label{figure 1F special IIIstar configs jaco non-jaco}
\end{adjustbox}
\end{table}

\begin{enumerate}
\item[(3)]\label{discussion 1F general IIIstar} By the proof of Proposition \ref{prop: 1F char3 fixed points revised}(\ref{prop 1F t0 general revised}), the point $\widetilde P$ can be chosen among the $1$-dimensional family of points in $B^0\setminus\{0,Q_+,Q_-\}$ on the unique $(-1)$-component $B$ of the ${\rm III}^*$-member. Blowing up such a point, the $(-1)$-curve $B$ contributes a $(-2)$-curve on $\widetilde Z$ to the resulting ${\rm III}^*$-configuration, which is the reduced multiple fiber of index $3$. Since the rank bound of Lemma \ref{lemma (-2)curves on Ztilde} is not reached yet, we now want to exclude the possibility that $\widetilde Z$ contains an additional fiber of type ${\rm I}_2$ or ${\rm III}$, whose components were not yet visible as negative curves on $\widetilde X$.

Assume, for a contradiction, that there exists such an additional reducible fiber, and denote its configuration by $\Gamma\in\{ {\rm I}_2,{\rm III}\}$. We study the possible intersection behavior of the red exceptional curve $E$ on $\widetilde Z$ with the two $(-2)$-components of $\Gamma$. Since $E$ is a $3$-section, it meets $\Gamma$ with total multiplicity $3$. Hence there is a component $F$ of $\Gamma$ with $E.F=0$ or $E.F=1$.
\begin{itemize}
\item If $E.F=0$, then the image of $F$ in $\widetilde X$ is a $(-2)$-curve. But we know all $(-2)$-curves on $\widetilde X$, and all of them are already contained in the old $E_7$-configuration. This is a contradiction.
\item If $E.F=1$, then the image of $F$ in $\widetilde X$ is a $(-1)$-curve through $\widetilde P$. Since $\widetilde P$ was chosen to lie on no other $(-1)$-curve besides $B$, this is again a contradiction.
\end{itemize}
Thus there is no additional reducible fiber. Note that this argument heavily relies on the fact that we are working in characteristic $3$.
The situation is summarized in the following Figure \ref{figure 1F general IIIstar configs jaco non-jaco} and Corollary \ref{cor: 1F (-2)configs char3}(\ref{cor 1F configs general IIIstar}). 
\end{enumerate}

\begin{table}[H]
\begin{adjustbox}{center}
$
\begin{array}{ccccc}
 \begin{array}{c} \addstackgap[2pt]{\includegraphics[width=0.22\textwidth]{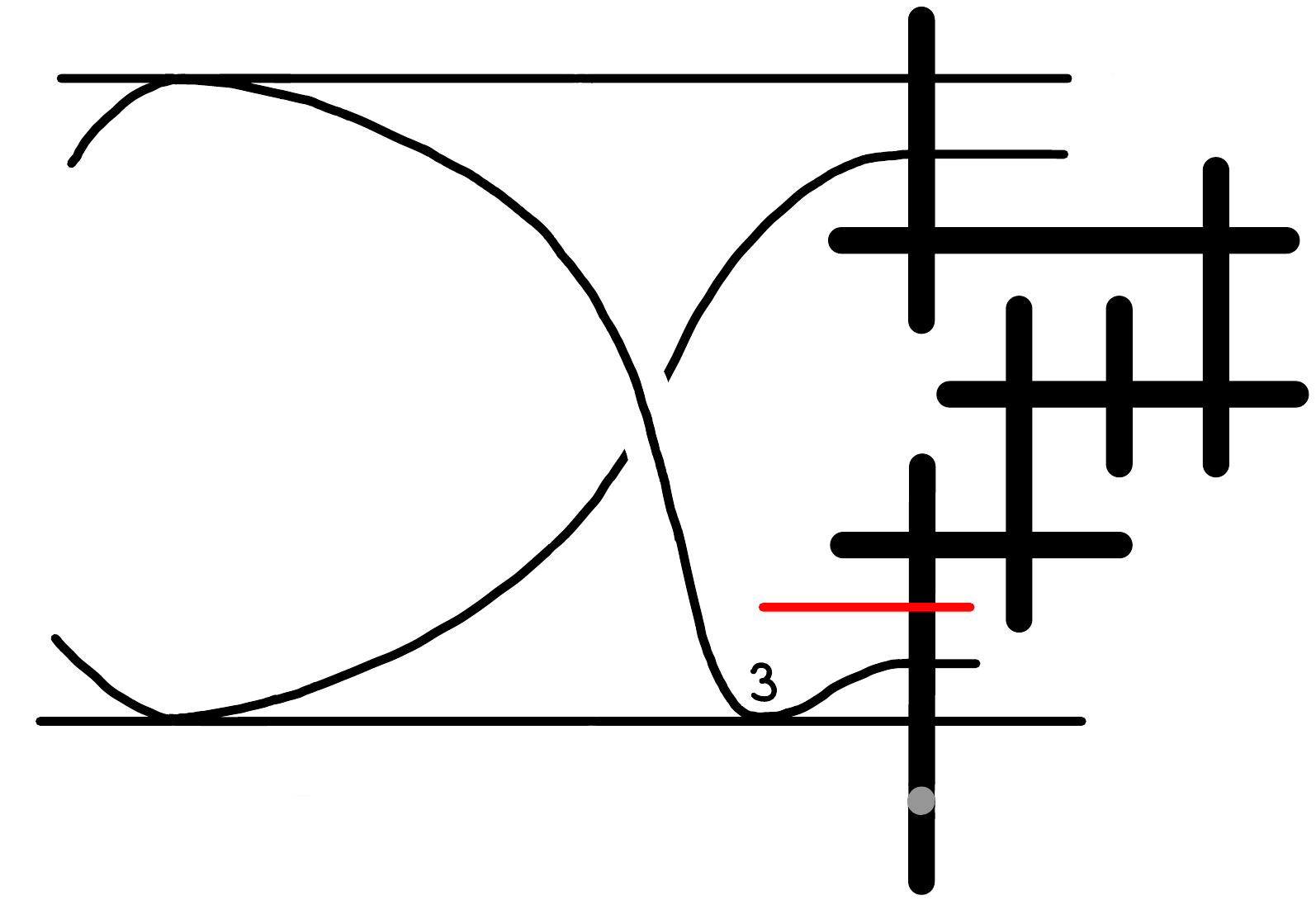} }\end{array}
 & \rightarrow
  & \begin{array}{c}\addstackgap[2pt]{\includegraphics[width=0.22\textwidth]{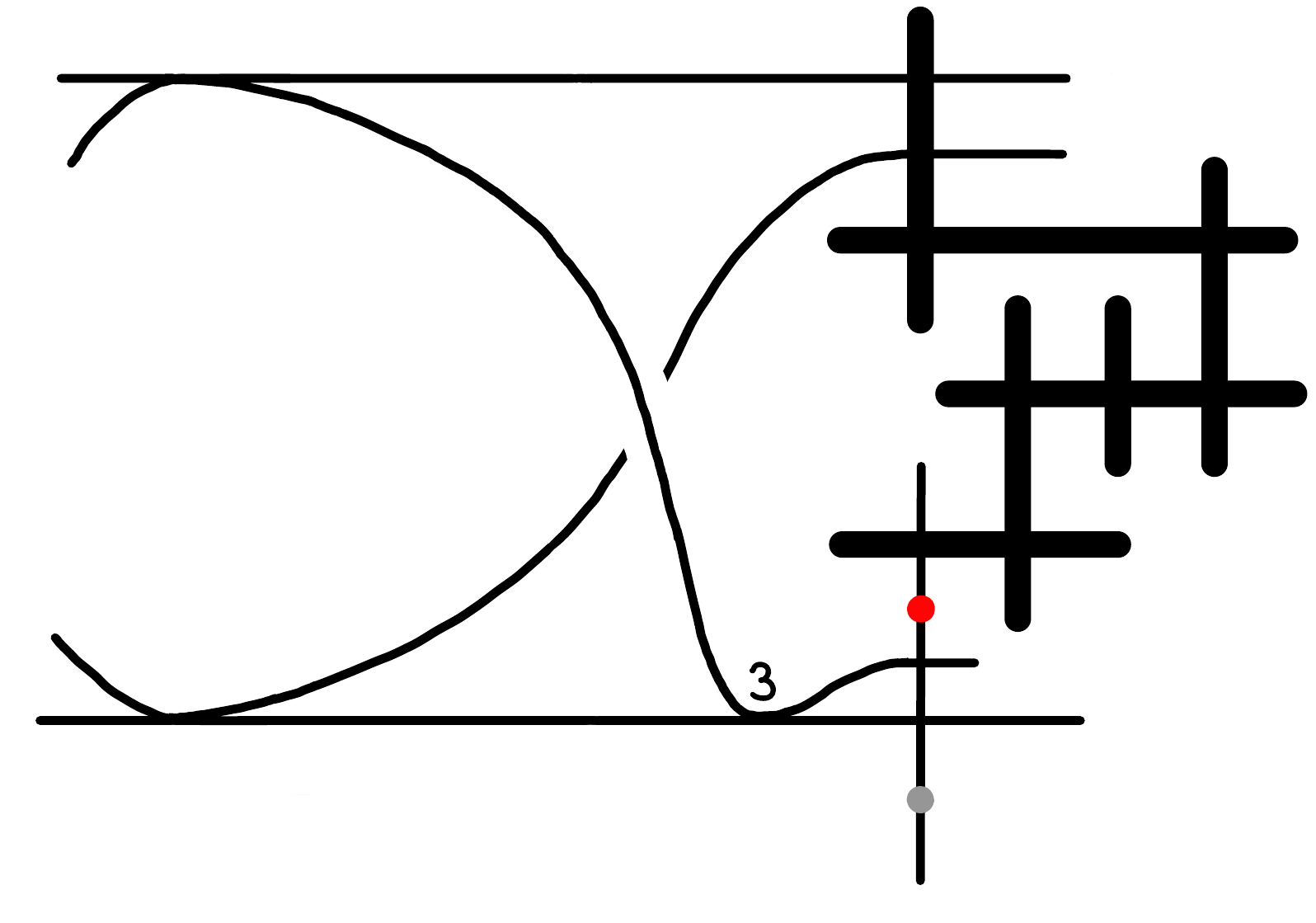} }\end{array}
  & \leftarrow
  & \begin{array}{c}\addstackgap[2pt]{\includegraphics[width=0.21\textwidth]{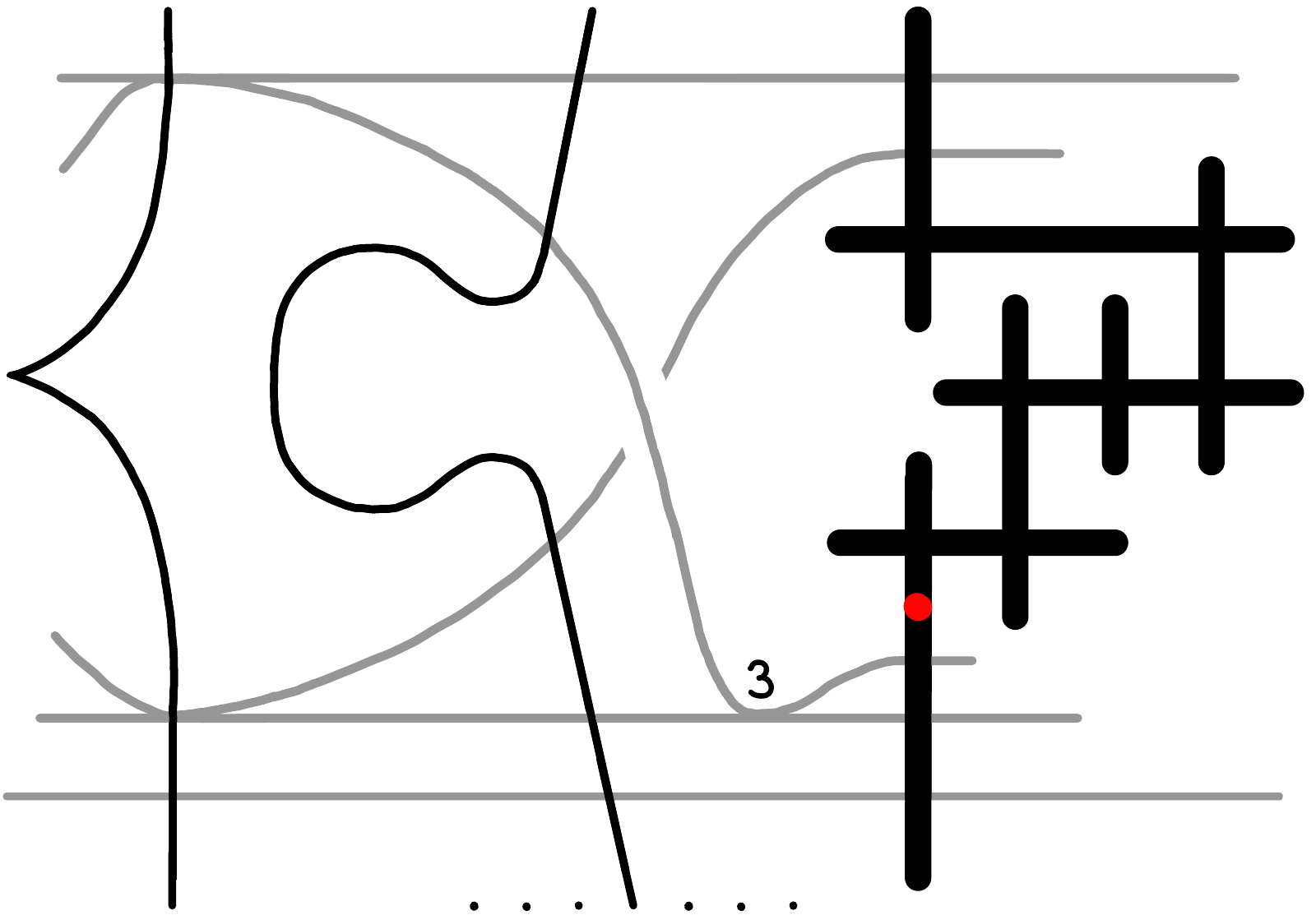} }\end{array}
  \\
  &
  & \downarrow
  &
  & \downarrow
  \\
  &
  & \begin{array}{c}\addstackgap[2pt]{\includegraphics[width=0.18\textwidth]{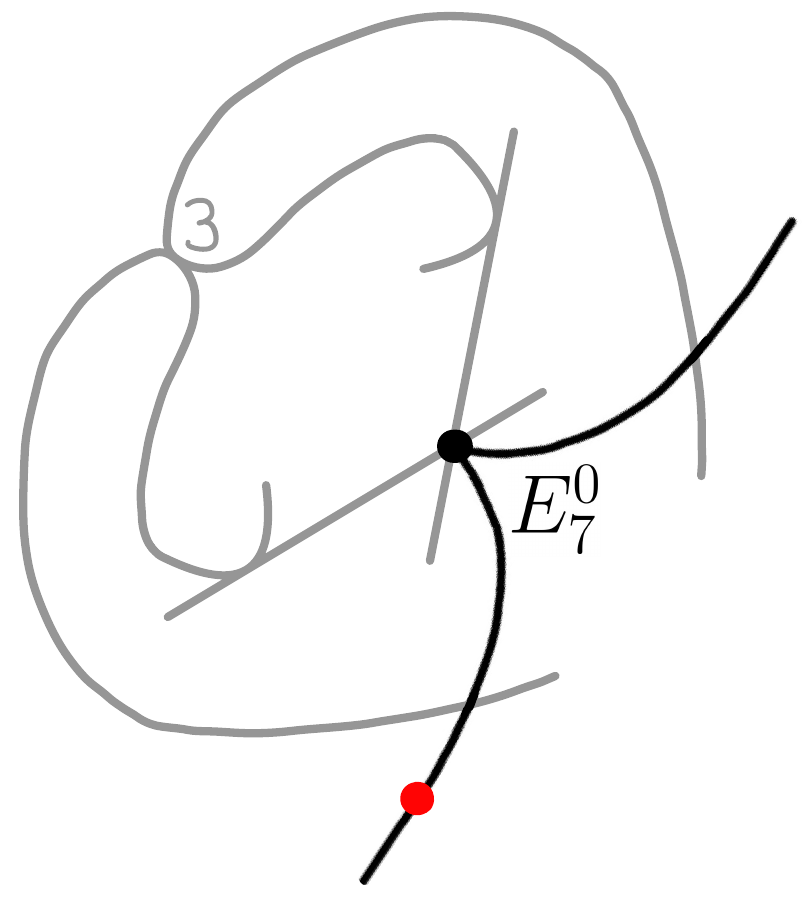}} \end{array}
  & \leftarrow
  & \begin{array}{c}\addstackgap[2pt]{ \includegraphics[width=0.20\textwidth]{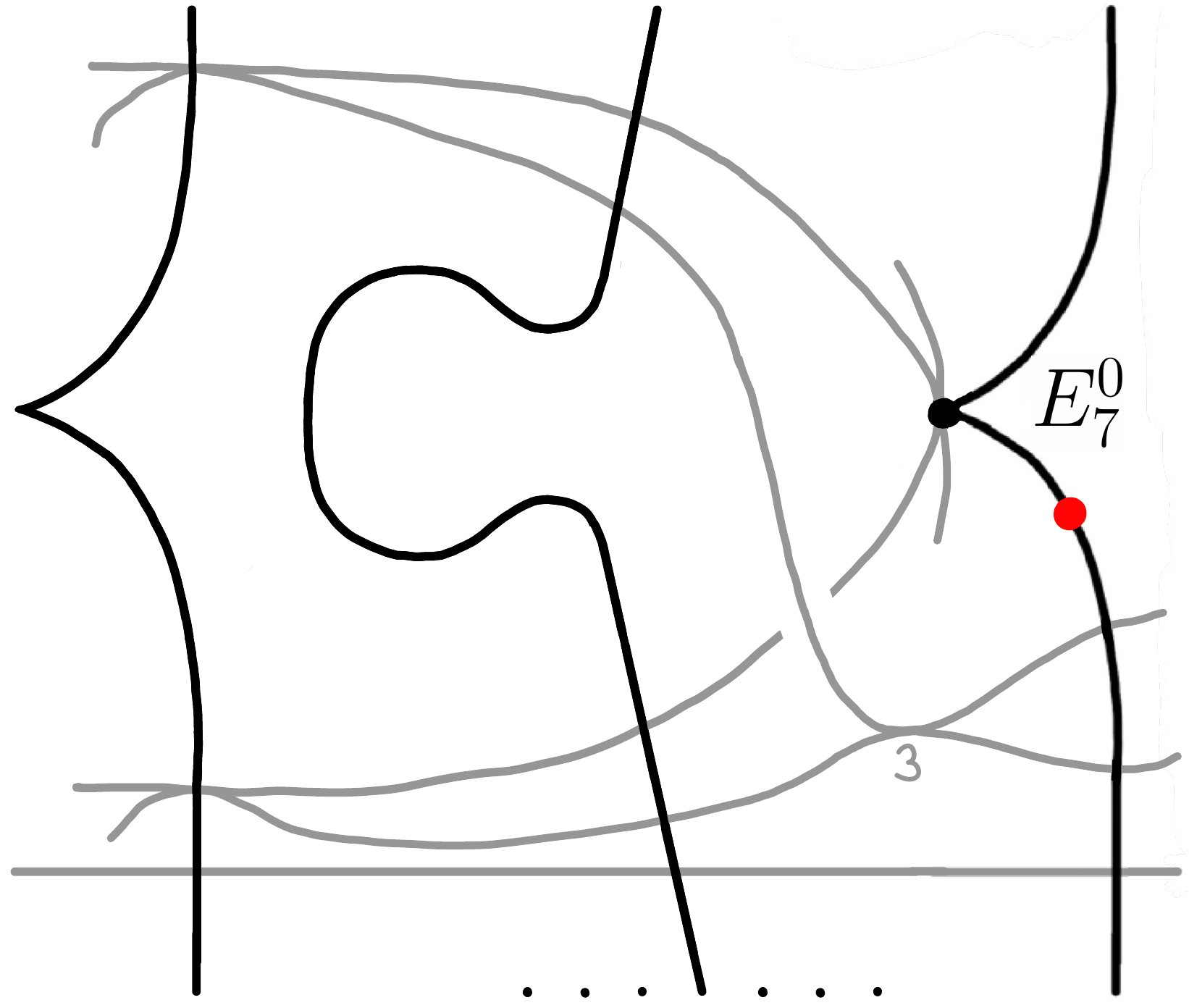}}  \end{array}
\end{array}
$
\captionof{figure}{Non-Jacobian and Jacobian fibrations with global vector fields originating from $\widetilde X$ of type \hyperref[Tab1F]{$1F$}, where the center is a general point on the $(-1)$-component of the ${\rm III}^*$-member}
\label{figure 1F general IIIstar configs jaco non-jaco}
\end{adjustbox}
\end{table}
\end{Discussion}
\begin{Corollary} \label{cor: 1F (-2)configs char3}
Let $\widetilde Z$ be one of surfaces of Corollary \ref{cor: 1F char3 nonJac revised}. Then one of the following holds.
\begin{enumerate}
\item\label{cor 1F configs multiple II}
$\widetilde Z$ contains nine $(-2)$-curves with dual graph of type $\widetilde E_8$ forming configuration ${\rm II}^*$. Moreover, $\widetilde Z$ is isomorphic to the unique surface in Corollaries \ref{cor: 1B non-Jac} and \ref{cor: 1B (-2)configs} in characteristic $3$, obtained from a weak del Pezzo surface of type \hyperref[Tab1B]{$1B$}.

\item\label{cor 1F configs special IIIstar}
$\widetilde Z$ contains ten $(-2)$-curves with dual graph of type $\widetilde E_7+\widetilde A_1$ forming configurations ${\rm III}^*$ and ${\rm I}_2$. Moreover, the ${\rm III}^*$-fiber is the unique multiple fiber and has multiplicity $m=3$. Furthermore, $\widetilde Z$ is isomorphic to the unique surface in Corollaries \ref{cor: 1C non-Jac} and \ref{cor: 1C (-2)configs} in characteristic $3$, obtained from a weak del Pezzo surface of type \hyperref[Tab1C]{$1C$}.

\item\label{cor 1F configs general IIIstar}
$\widetilde Z$ contains eight $(-2)$-curves with dual graph of type $\widetilde E_7$ forming configuration ${\rm III}^*$. Moreover, the ${\rm III}^*$-fiber is the unique multiple fiber and has multiplicity $m=3$, and such surfaces $\widetilde Z$ form a $1$-dimensional family.
\end{enumerate}
\end{Corollary}

\subsection{Case \hyperref[Tab1G]{$1G$}} \label{subsection 1G char3 revised} 
This case exists only if $\Char(k)=p=3$.

\begin{Proposition} \label{prop: 1G char3 fixed points revised}
Let $\widetilde X$ be of type \hyperref[Tab1G]{$1G$}. Then, $\widetilde{P}$ is admissible if and only if one of the following holds.
\begin{enumerate}
\item\label{prop 1G s0 revised} $C$ is of type ${\rm II}$ and $\widetilde P$ lies on one of the two $(-1)$-curves on $\widetilde X$ which are not contained in the anti-canonical member of type ${\rm I}_9$.
\item\label{prop 1G t0 revised} $C$ is of type ${\rm I}_9$ and $\widetilde P$ lies on the unique $(-1)$-curve contained in this anti-canonical member. Moreover, $\widetilde P$ is a torsion point of exact order $m$ on $C^0\cong \mathbb G_m$; in particular, $(m,3)=1$.
\end{enumerate}
Moreover, in both cases $({\rm Stab}_{\Aut_{\widetilde X}^0}(\widetilde P))^0\cong \mu_3$, and in case (\ref{prop 1G s0 revised}) we have $m=3$.
\end{Proposition}

\begin{proof}
By Proposition \ref{prop: char3 equations and liftable actions}, $X$ is given by
$y^2=x^3+s^2x^2+st^3x+t^6$, and the action of $\Aut_{\widetilde X}^0\cong \mu_3$ on $X$ is given by
$[s:t:x:y]\mapsto [s:\lambda t:x:y]$. To find admissible $P \in X$ (according to Strategy \ref{strategy of proof non-Jaco}), we distinguish the following cases:

\begin{enumerate}[leftmargin=0.8cm]
\item[(a)]
If $s\neq 0$, we may set $s=1$. Then the action is $[1:t:x:y]\mapsto [1:\lambda t:x:y]$. Thus a point in this chart has non-trivial connected stabilizer only if $t=0$. The point $[1:0:0:0]$ is the $A_8$-singularity and hence cannot be the image of $\widetilde P$. The remaining points on this fiber are the points $[1:0:x:y]$ with $x\neq 0$ and $y^2=x^3+x^2$. They lie on the smooth locus of the non-contracted component of the anti-canonical member over $t=0$. On $\widetilde X$, this component is the unique $(-1)$-curve contained in this member, and after blowing up the base point of $|-K_{\widetilde X}|$ it becomes the affine component of the associated Jacobian fiber of type ${\rm I}_9$. Its identity component is isomorphic to $\mathbb G_m$. Thus such a point is a valid Halphen point precisely if it is a torsion point of exact order $m$ in $\mathbb G_m$, and in characteristic $3$ this implies $(m,3)=1$. The connected stabilizer of each such point is all of $\mu_3$.

\item[(b)]
It remains to consider the chart $s=0$. Away from the base point, we may set $t=1$. After rescaling, the action becomes $[0:1:x:y]\mapsto [0:1:\lambda x:y]$. Hence the connected stabilizer is non-trivial precisely if $x=0$. The equation then gives $y^2=1$, so the only possible points are $P_+=[0:1:0:1]$ and $P_-=[0:1:0:-1]$. For both of these points, the connected stabilizer $({\rm Stab}_{\mu_3}(P_\pm))^0$ is all of $\mu_3$.

The anti-canonical curve containing these points is $C=X\cap\{s=0\}$. It is the cuspidal curve $y^2=x^3+1=(x+1)^3$ with cusp at $[0:1:-1:0]$, whereas the neutral element of $C^0$ is the point where the strict transform of $C$ meets the exceptional curve over the base point of $|-K_{\widetilde X}|$. Thus $P_+$ and $P_-$ lie in the smooth locus of $C$ and are distinct from the neutral element. Therefore $C^0\cong \mathbb G_a$, and since $\Char(k)=3$, both points have exact \mbox{order $3$.}
\end{enumerate}

It remains to identify the position of these points on the configuration of negative curves. The $(-1)$-curves on $\widetilde X$ are visible on the anti-canonical model. Indeed, the divisor $\{x=0\}$ on $X$ splits as the union of the two horizontal curves $S_\pm=\{x=0,\ y=\pm t^3\}$. Together with $B=\{t=0,\ y^2=x^3+s^2x^2\}$, the non-contracted component of the fiber over $t=0$, their strict transforms are the three $(-1)$-curves on the weak del Pezzo surface of type \hyperref[Tab1G]{$1G$}. Moreover, $P_+$ lies on $S_+$ and $P_-$ lies on $S_-$. Thus the two points on the curve of type ${\rm II}$ are exactly the two points described in (\ref{prop 1G s0 revised}), while the points on $B^0$ are exactly the points described in (\ref{prop 1G t0 revised}).
\end{proof}

\begin{Corollary} \label{cor: 1G char3 nonJac revised}
Let $\widetilde Z$ be arising from an $\widetilde X$ of type \hyperref[Tab1G]{$1G$} and assume that $h^0(\widetilde Z,T_{\widetilde Z})\neq 0$. Then $\Aut^0_{\widetilde Z}\cong \mu_3$, and the following hold.
\begin{enumerate}
\item[(1)]\label{cor 1G uniqueness(1) II}
If $\widetilde Z$ has one multiple fiber $3{\rm II}$, then $\widetilde{Z}$ is unique up to isomorphism.
\item[(2)]\label{cor 1G torsion I9}
Otherwise, $\widetilde Z$ has one multiple fiber $m{\rm I}_9$, where $m>1$ is the exact order of $\widetilde P$ on $C^0\cong\mathbb G_m$; in particular, $(m,3)=1$. Conversely, every such $m$ occurs. For fixed $m$, the resulting surfaces are parametrized by primitive $m$-th roots of unity modulo inversion. Hence there is one such surface if $m=2$, and $\varphi(m)/2$ such surfaces if $m>2$.
\end{enumerate}
\end{Corollary}

\begin{proof}
Everything except the uniqueness statement in (\hyperref[cor 1G uniqueness(1) II]{1}) and the isomorphism classification in (\hyperref[cor 1G torsion I9]{2}) follows by combining Corollary \ref{cor: approach for classification non-jacobian} with Proposition \ref{prop: 1G char3 fixed points revised}.

For (\hyperref[cor 1G uniqueness(1) II]{1}), recall from the proof of Proposition \ref{prop: 1G char3 fixed points revised} that the two possible images of $\widetilde P$ on $X$ are $P_+=[0:1:0:1]$ and $P_-=[0:1:0:-1]$. The involution
$\iota:[s:t:x:y]\mapsto [s:t:x:-y]$
preserves the equation $y^2=x^3+s^2x^2+st^3x+t^6$, fixes the unique singular point of $X$, and interchanges $P_+$ and $P_-$. Hence $\iota$ lifts to an automorphism of $\widetilde X$, and the two possible blow-ups are isomorphic.

For (\hyperref[cor 1G torsion I9]{2}), we have
$C^0\cong\mathbb G_m$, so $C^0(k)=k^*$. Since $k$ is algebraically
closed, for every $m>1$ with $(m,3)=1$ there exist points of exact
order $m$ on $C^0$. Conversely, a $k$-point of $\mathbb G_m$ has no
non-trivial $3$-torsion, since $
T^{3^r}-1=(T-1)^{3^r}$ for every $r\geq1$. Thus every torsion point of $C^0(k)$ has order
prime to $3$. For a point $\zeta\in C^0(k)$ of exact order $m$, denote
the corresponding blow-up by $\widetilde Z_\zeta$.

It remains to determine the orbits of these points under the full
automorphism group $\Aut(\widetilde X)$. Let $\Gamma$ be the
configuration of $(-2)$-curves on $\widetilde X$, which is of type
$A_8$. Since every automorphism of $\widetilde X$ preserves $\Gamma$,
there is a homomorphism $\rho\colon
\Aut(\widetilde X)
\to
\Aut(\Gamma)
\cong
\mathbb Z/2\mathbb Z$.
The involution $y \mapsto -y$
lifts to $\widetilde X$ and acts non-trivially on $\Gamma$. Hence
$\rho$ is surjective.
We claim that $\ker(\rho)$ is trivial. Indeed, an element of
$\ker(\rho)$ fixes every component of $\Gamma$. The classes of these
eight components together with $-K_{\widetilde X}$ generate
$\Pic(\widetilde X)_{\mathbb Q}$. Since every automorphism fixes the
canonical class, every element of $\ker(\rho)$ acts trivially on
$\Pic(\widetilde X)_{\mathbb Q}$ and hence, since
$\Pic(\widetilde X)$ is torsion-free, trivially on
$\Pic(\widetilde X)$. It is therefore cohomologically trivial.

Such an automorphism fixes the base point of
$|-K_{\widetilde X}|$ and hence lifts to a cohomologically trivial
automorphism of the associated Jacobian $\widetilde{Y}$.
By \cite[Table 3]{DolgachevMartin}, no such non-trivial automorphism
exists in characteristic $3$ (the only relevant entry containing an
$A_8$-configuration is Case 12 there, which occurs only $\Char(k) \neq 3$). Hence, $\ker(\rho)=\{1\}$ and $\Aut(\widetilde X)
\cong
\mathbb Z/2\mathbb Z$.
Under the identification $C^0\cong\mathbb G_m$, the involution
$\iota$ induces inversion. Therefore,
$
\widetilde Z_\zeta\cong\widetilde Z_{\zeta'}
$ if and only if $\zeta' \in \{\zeta, \zeta^{-1}\}$.
If $m=2$, the unique
primitive root $\zeta=-1$ is fixed by inversion, so there is a unique $\widetilde{Z}$. If $m>2$, inversion acts without fixed points on the
$\varphi(m)$ primitive $m$-th roots of unity, and hence there are $\varphi(m)/2$ isomorphism classes.
\end{proof}

\begin{Discussion} \label{discussion: 1G curves char3 revised}
The Jacobian surface $\widetilde Y$ associated with $\widetilde X$ has singular fibers ${\rm I}_9$ and ${\rm II}$. Indeed, the discriminant of $y^2=x^3+s^2x^2+st^3x+t^6$ is $-s^3t^9$, hence $\widetilde Y$ is elliptic and there are two singular fibers. Tate's algorithm gives a multiplicative fiber of type ${\rm I}_9$ over $t=0$, and the fiber over $s=0$ is the irreducible cuspidal fiber of type ${\rm II}$. Thus the reducible fiber contributes the root lattice $A_8$. By \cite{OguisoShioda}, the Mordell--Weil group of the Jacobian surface is $\mathbb Z/3\mathbb Z$. The two non-zero sections are visible on the anti-canonical model as $S_\pm=\{x=0,\ y=\pm t^3\}$. If the components of the ${\rm I}_9$-fiber are numbered cyclically, then the zero-section meets one component and $S_+$ and $S_-$ meet the components at distance $3$ and $6$.

To determine the number and configuration of $(-2)$-curves on the non-Jacobian surface $\widetilde Z$, we treat the cases (\hyperref[cor 1G uniqueness(1) II]{1}) and (\hyperref[cor 1G torsion I9]{2}) of Corollary \ref{cor: 1G char3 nonJac revised} separately.

\begin{enumerate}
\item[(1)]\label{discussion 1G multiple II} By the proof of Proposition \ref{prop: 1G char3 fixed points revised}(\ref{prop 1G s0 revised}), the point $\widetilde P$ is one of the two points on $\widetilde X$ lying on one of the two $(-1)$-curves $S_\pm$. By Corollary \ref{cor: 1G char3 nonJac revised}(\hyperref[cor 1G uniqueness(1) II]{1}), we may choose either of them and take the $\widetilde P$ corresponding to $P_+=[0:1:0:1]$ on $X$.

After blowing up this point, the strict transform of $S_+$ becomes a $(-2)$-curve on $\widetilde Z$. Together with the strict transforms of the old $A_8$-configuration on $\widetilde X$, it forms a fiber of type ${\rm II}^*$. The strict transform of the cuspidal curve $X\cap\{s=0\}$ is the reduced multiple fiber, so the multiple fiber is $3{\rm II}$. The rank bound of Lemma \ref{lemma (-2)curves on Ztilde} is saturated by the ${\rm II}^*$-fiber, hence there are no further reducible fibers.

The situation is summarized in Figure \ref{figure 1G multiple II configs jaco non-jaco} and Corollary \ref{cor: 1G (-2)configs char3}(\ref{cor 1G configs multiple II}). Using the known multiplicities of the components of Kodaira--N{\'e}ron fibers, we see that on $\widetilde Z$ all thin curves are $3$-sections.
\end{enumerate}

\begin{table}[H]
\begin{adjustbox}{center}
$
\begin{array}{ccccc}
 \begin{array}{c} \addstackgap[2pt]{\includegraphics[width=0.12\textwidth]{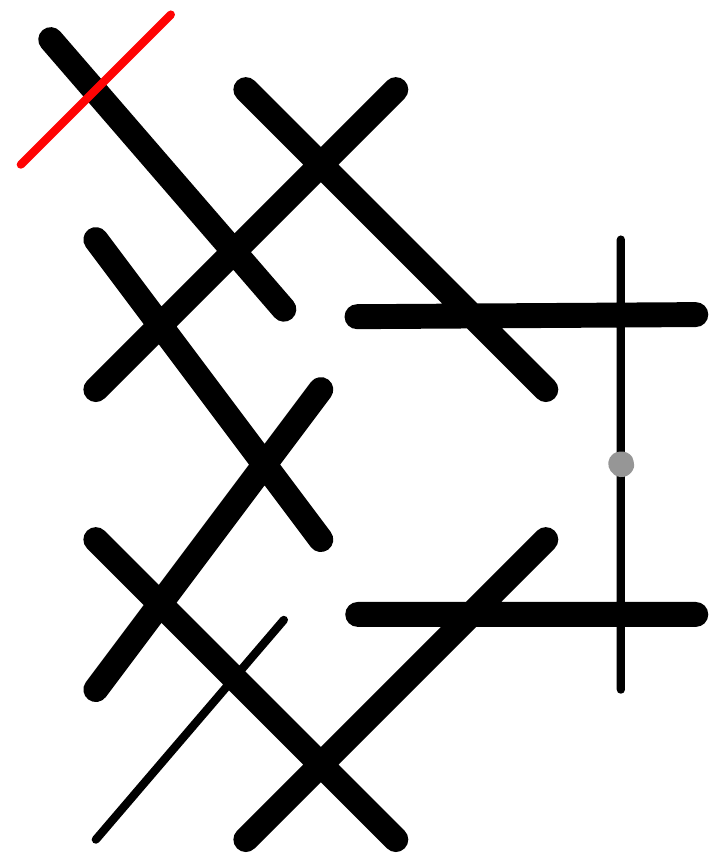} }\end{array}
 & \rightarrow
  & \begin{array}{c}\addstackgap[2pt]{\includegraphics[width=0.12\textwidth]{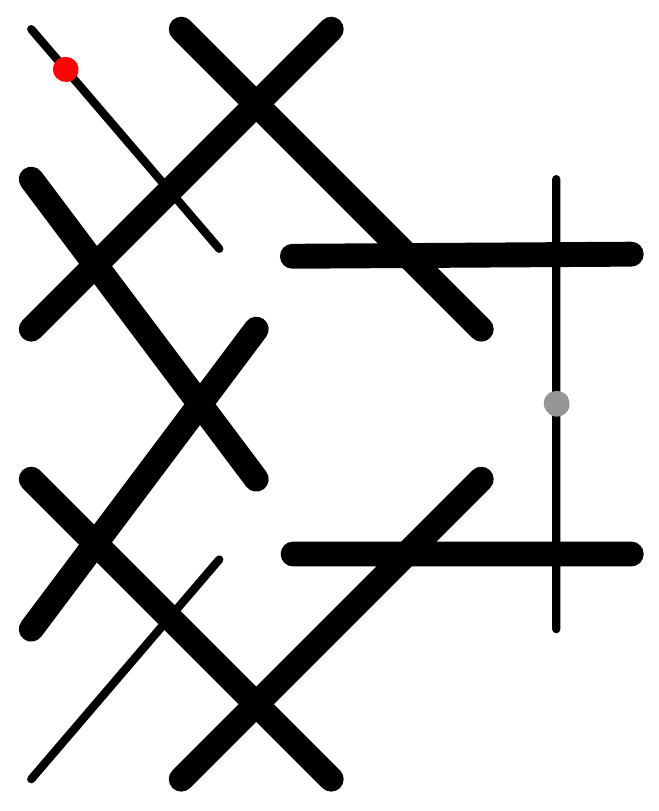} }\end{array}
  & \leftarrow
  & \begin{array}{c}\addstackgap[2pt]{\includegraphics[width=0.20\textwidth]{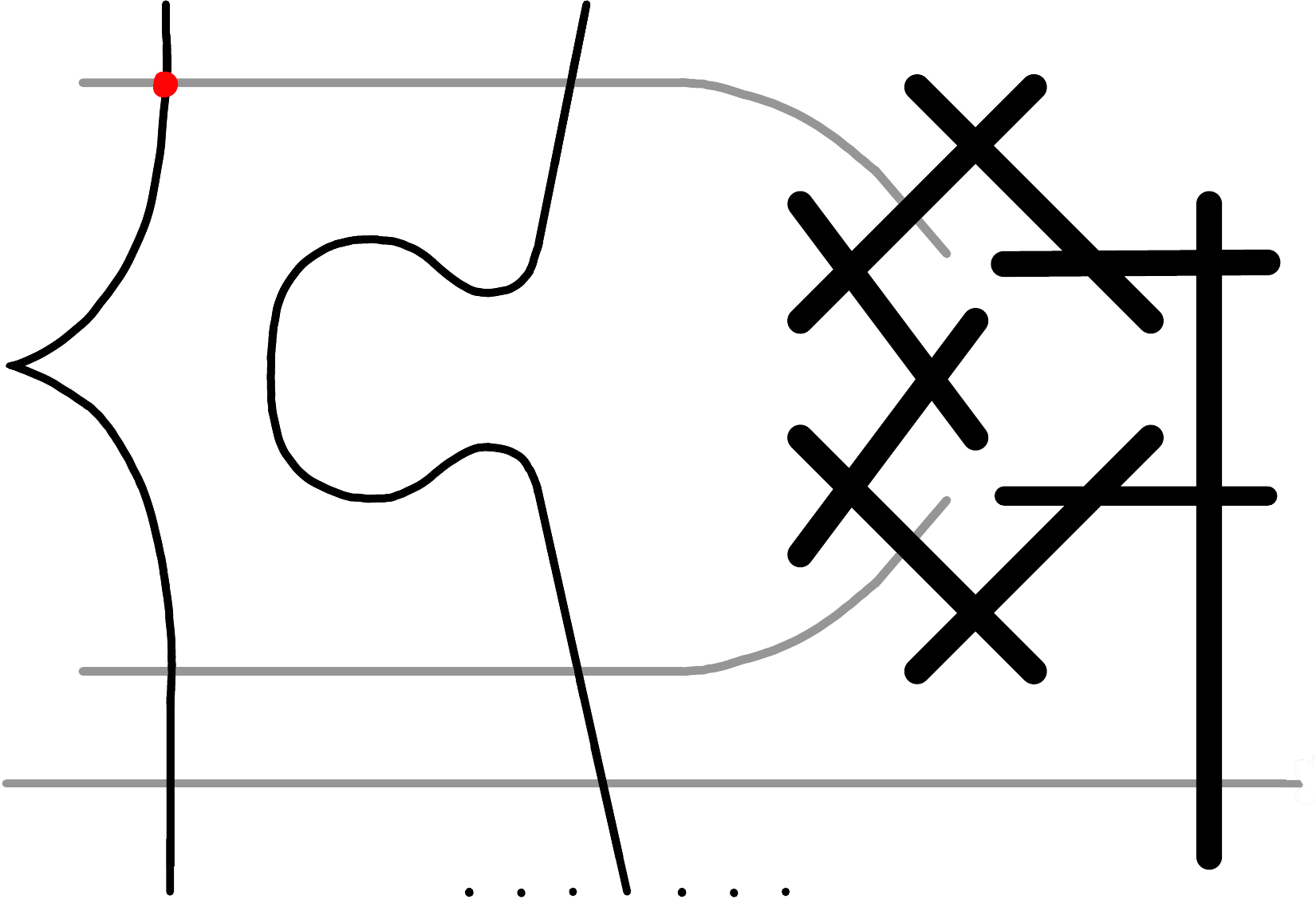} }\end{array}
  \\
  &
  & \downarrow
  &
  & \downarrow
  \\
  &
  & \begin{array}{c}\addstackgap[2pt]{\includegraphics[width=0.14\textwidth]{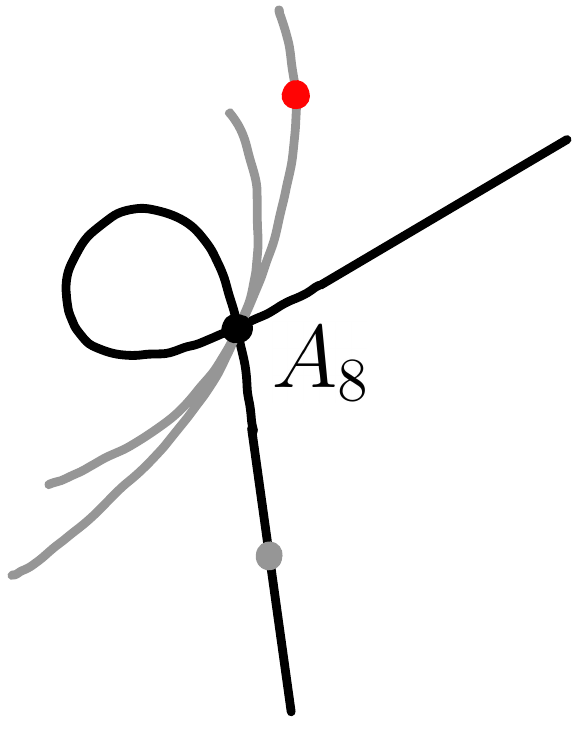}} \end{array}
  & \leftarrow
  & \begin{array}{c}\addstackgap[2pt]{ \includegraphics[width=0.20\textwidth]{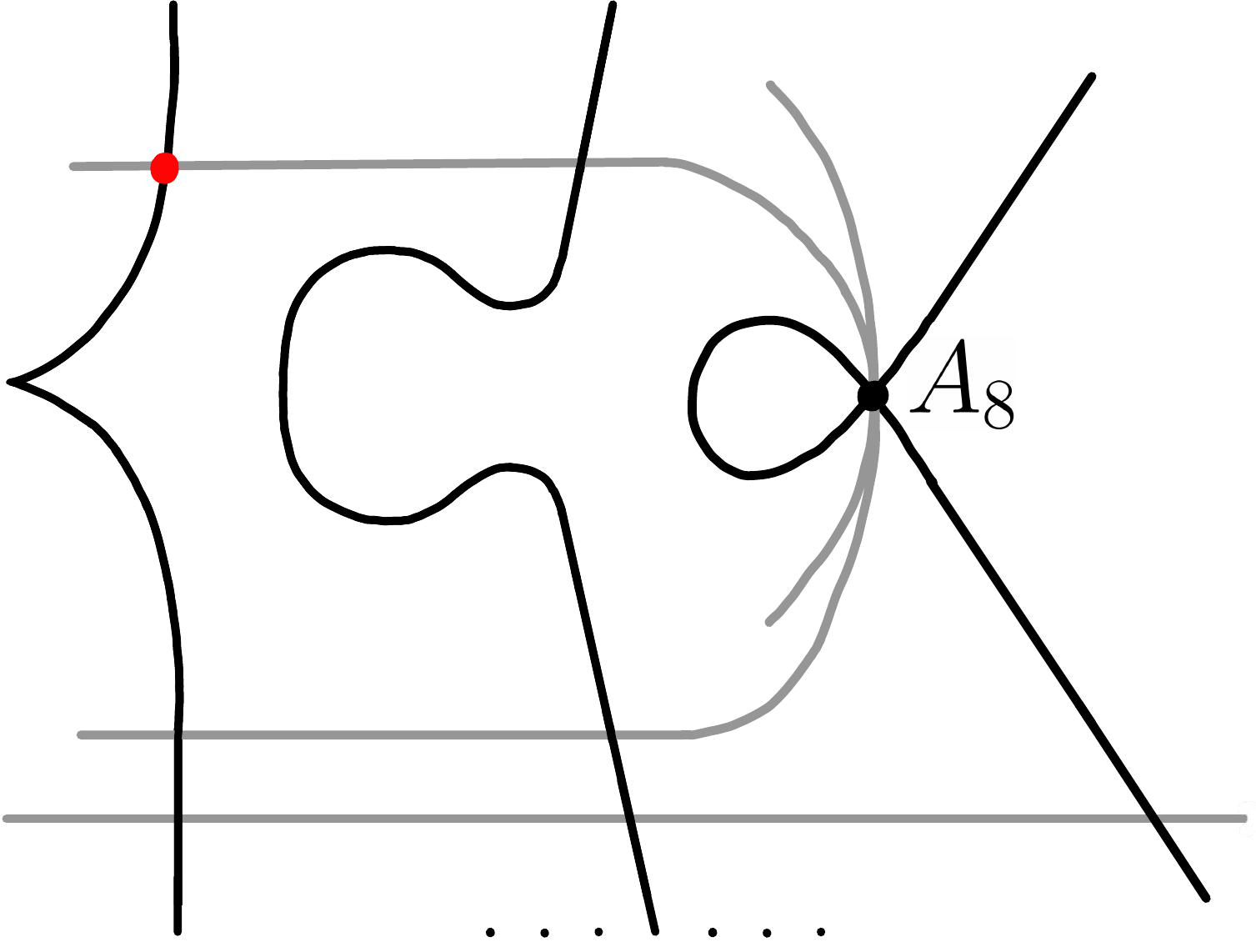}}  \end{array}
\end{array}
$
\captionof{figure}{Non-Jacobian and Jacobian fibrations with global vector fields originating from $\widetilde X$ of type \hyperref[Tab1G]{$1G$}, where the center lies on the type ${\rm II}$ curve}
\label{figure 1G multiple II configs jaco non-jaco}
\end{adjustbox}
\end{table}

\begin{enumerate}
\item[] \emph{This $\widetilde Z$ is not ``new'':}
In the left-hand diagram of Figure
\ref{figure 1E configs jaco non-jaco} in case
\hyperref[subsection 1E char3 revised]{$1E$}, we contract the thin black
$3$-section meeting the end component of the shortest of the three arms
of the ${\rm II}^*$-configuration. This yields another weak del Pezzo
surface of degree $1$ whose configuration of $(-2)$-curves is of type
$A_8$ and which still has global vector fields by Blanchard's Lemma.
By the classification in
\cite[Table 6]{WeakDelPezzoGlobalVectorFields}, this is the unique weak
del Pezzo surface of type \hyperref[Tab1G]{$1G$}, that we just blew up to produce $\widetilde{Z}$. Consequently, the surfaces obtained in Corollary
\ref{cor: 1G char3 nonJac revised}
(\hyperref[cor 1G uniqueness(1) II]{1}) and
Corollary \ref{cor: 1E char3 nonJac revised} are isomorphic.
\end{enumerate}
\end{Discussion}

\begin{enumerate}
\item[(2)]\label{discussion 1G multiple I9} By the proof of Proposition \ref{prop: 1G char3 fixed points revised}(\ref{prop 1G t0 revised}), the point $\widetilde P$ lies on the unique $(-1)$-component $B$ of the anti-canonical member of type ${\rm I}_9$ and has exact order $m$ on $C^0=B^0\cong\mathbb G_m$. Since $\Char(k)=3$, this forces $(m,3)=1$.

After blowing up such a point, the strict transform of $B$ becomes a $(-2)$-curve on $\widetilde Z$. Together with the old $A_8$-configuration, it forms the reduced multiple fiber of type ${\rm I}_9$. Hence the multiple fiber is $m{\rm I}_9$. The rank bound of Lemma \ref{lemma (-2)curves on Ztilde} is attained by this ${\rm I}_9$-fiber, hence there are no further reducible fibers.

The situation is summarized in the following Figure \ref{figure 1G multiple I9 configs jaco non-jaco} and Corollary \ref{cor: 1G (-2)configs char3}(\ref{cor 1G configs multiple I9}). Using the known multiplicities of the components of Kodaira--N{\'e}ron fibers, we observe that on $\widetilde Z$ all thin smooth curves in the picture are $m$-sections.
\end{enumerate}

\begin{table}[H]
\begin{adjustbox}{center}
$
\begin{array}{ccccc}
 \begin{array}{c} \addstackgap[2pt]{\includegraphics[width=0.12\textwidth]{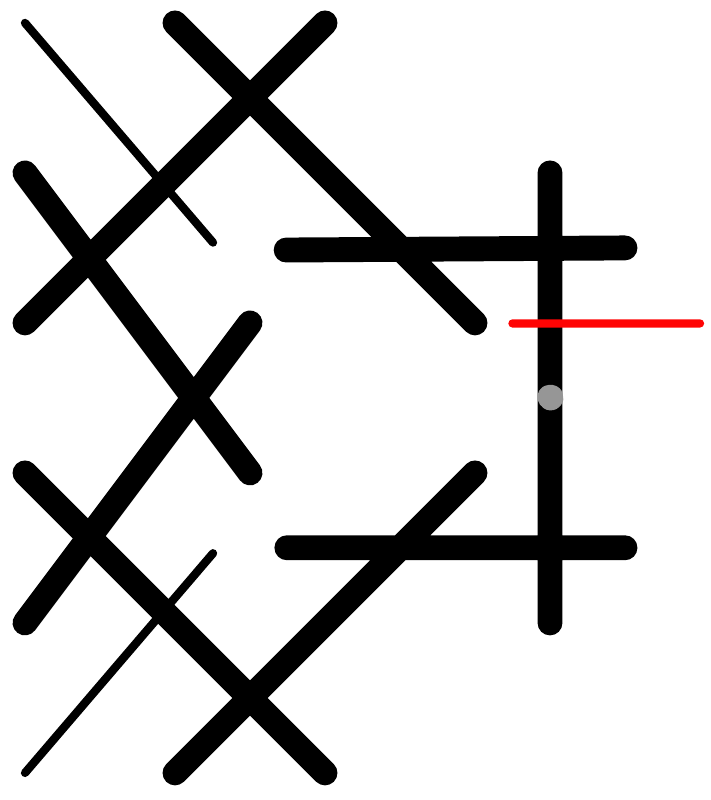} }\end{array}
 & \rightarrow
  & \begin{array}{c}\addstackgap[2pt]{\includegraphics[width=0.12\textwidth]{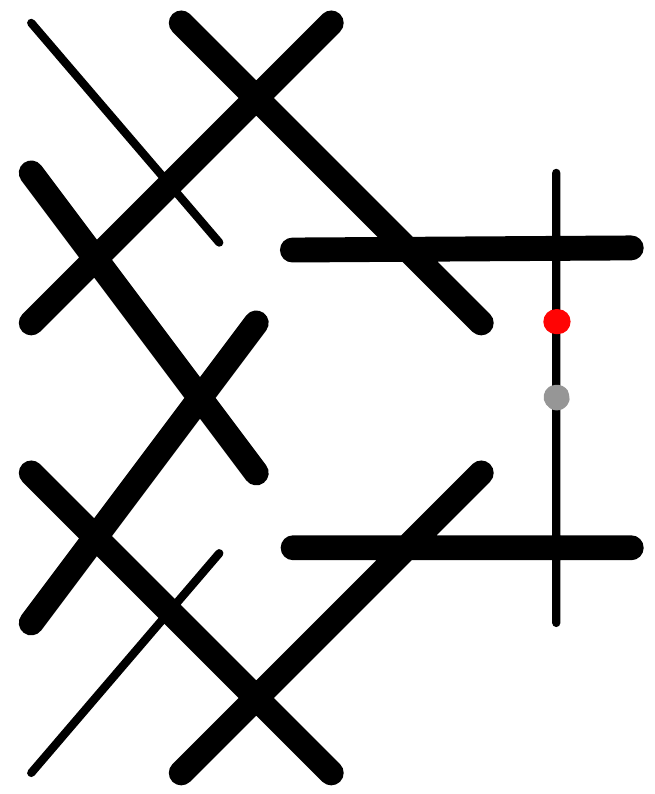} }\end{array}
  & \leftarrow
  & \begin{array}{c}\addstackgap[2pt]{\includegraphics[width=0.20\textwidth]{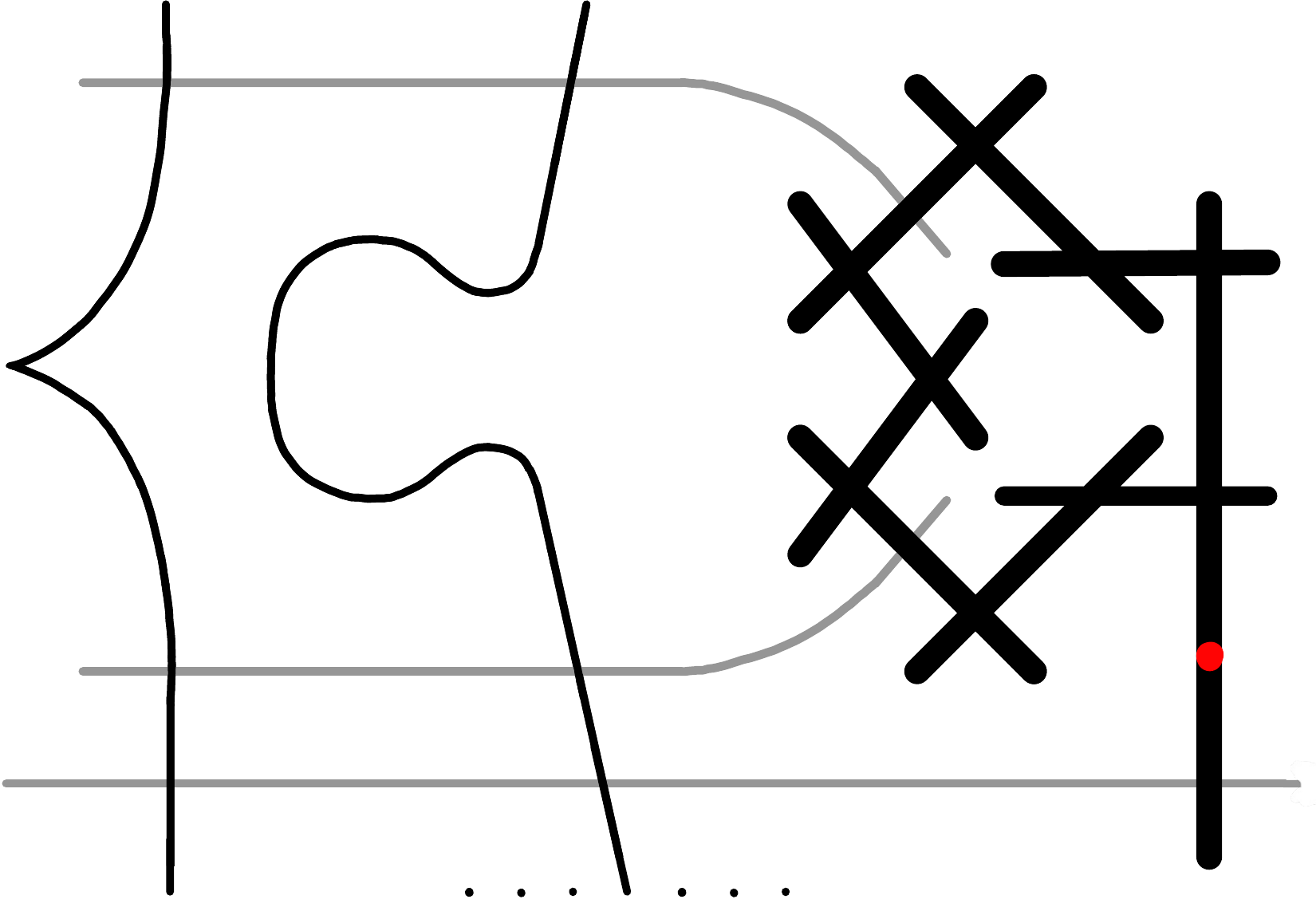} }\end{array}
  \\
  &
  & \downarrow
  &
  & \downarrow
  \\
  &
  & \begin{array}{c}\addstackgap[2pt]{\includegraphics[width=0.14\textwidth]{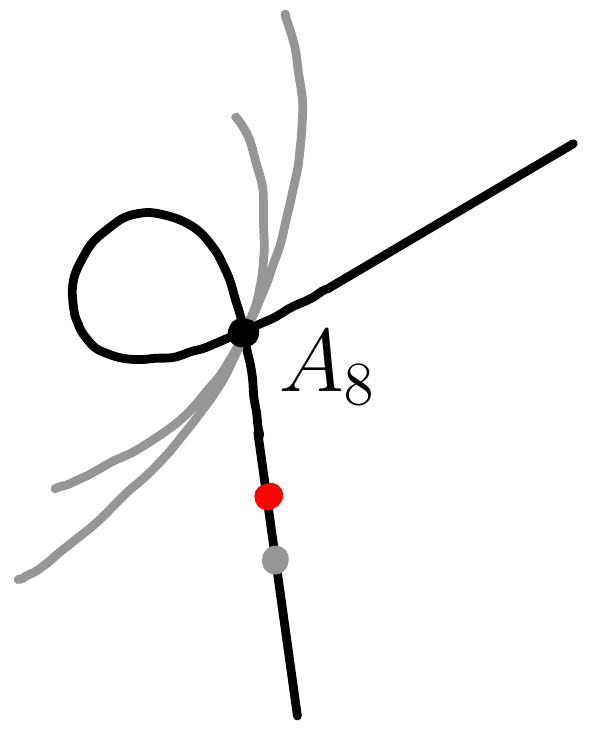}} \end{array}
  & \leftarrow
  & \begin{array}{c}\addstackgap[2pt]{ \includegraphics[width=0.20\textwidth]{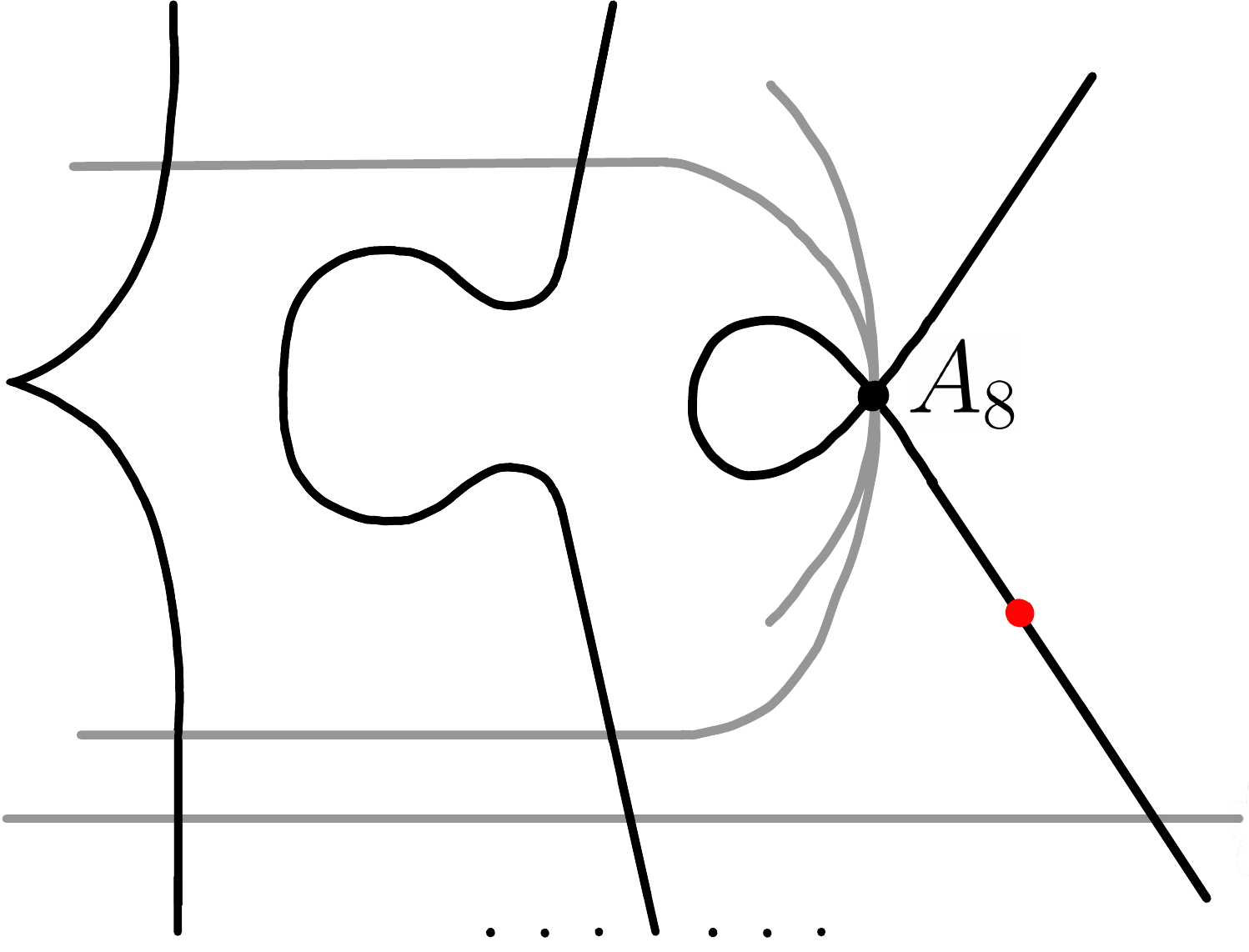}}  \end{array}
\end{array}
$
\captionof{figure}{Non-Jacobian and Jacobian fibrations with global vector fields originating from $\widetilde X$ of type \hyperref[Tab1G]{$1G$}, where the center is a torsion point of exact order $m$ on the $(-1)$-component of the ${\rm I}_9$-member}
\label{figure 1G multiple I9 configs jaco non-jaco}
\end{adjustbox}
\end{table}

\begin{Corollary} \label{cor: 1G (-2)configs char3}
Let $\widetilde Z$ be one of the surfaces of Corollary \ref{cor: 1G char3 nonJac revised}. Then one of the following holds.
\begin{enumerate}
\item\label{cor 1G configs multiple II}
$\widetilde Z$ contains nine $(-2)$-curves with dual graph of type $\widetilde E_8$ forming configuration ${\rm II}^*$. Moreover, the unique multiple fiber is $3{\rm II}$. Furthermore, $\widetilde{Z}$ is isomorphic to the unique surface in Corollaries \ref{cor: 1E char3 nonJac revised} and \ref{cor: 1E (-2)configs char3}, obtained from a weak del Pezzo surface of type \hyperref[Tab1E]{$1E$}.

\item\label{cor 1G configs multiple I9}
$\widetilde Z$ contains nine $(-2)$-curves with dual graph of type $\widetilde A_8$ forming configuration ${\rm I}_9$. Moreover, this ${\rm I}_9$-fiber is the unique multiple fiber and has multiplicity $m$, where $m>1$ is prime to $3$. If $m=2$, then $\widetilde Z$ is unique up to isomorphism. For every fixed $m>2$ prime to $3$, there are $\varphi(m)/2$ such surfaces $\widetilde Z$.
\end{enumerate}
\end{Corollary}

\subsection{Case \hyperref[Tab1H]{$1H$}} \label{subsection 1H char3 revised}
This case exists only if $\Char(k)=p=3$.

\begin{Proposition} \label{prop: 1H char3 fixed points revised}
Let $\widetilde X$ be of type \hyperref[Tab1H]{$1H$}. Then, $\widetilde{P}$ is admissible if and only if $C$ is of type ${\rm II}^*$ and $\widetilde P$ lies on the unique $(-1)$-curve contained in this anti-canonical member. Moreover, then
$({\rm Stab}_{\Aut^0_{\widetilde X}}(\widetilde P))^0\cong \mathbb G_a$
and $m=3$.
\end{Proposition}

\begin{proof}
By Proposition \ref{prop: char3 equations and liftable actions}, $X$ is given by
$y^2=x^3+s^4x+s^3t^3$, and the action of
$\Aut^0_{\widetilde X}\cong\mathbb G_a$ on $X$ is given by $[s:t:x:y]\mapsto
[s:t-(a^3+a)s:x+a^3s^2:y]$.
To find admissible $P \in X$ (according to Strategy \ref{strategy of proof non-Jaco}), we distinguish the following cases:

\begin{enumerate}[leftmargin=0.8cm]
    \item[(a)]
We first consider the locus $s=0$. Away from the base point, we may set $t=1$, and the action restricts to the identity. The point $[0:1:0:0]$ is the unique $E_8$-singularity of $X$, whereas all other points $P=[0:1:x:y]$ with $(x,y)\neq(0,0)$ and $y^2=x^3$ lie in the smooth locus of the cuspidal curve $C=X\cap\{s=0\}$. Thus $C^0\cong\mathbb G_a$, every such point is a non-zero point of exact order $3$, and, as observed above, is fixed by all of $\Aut^0_{\widetilde X}\cong\mathbb G_a$.

\item[(b)]
It remains to consider the locus $s\neq0$. We may set $s=1$, and the action becomes $[1:t:x:y]\mapsto [1:t-(a^3+a):x+a^3:y]$. Hence a point in this chart is fixed if and only if $a^3=0$ and $a^3+a=0$, which implies $a=0$. Thus every point in this chart has trivial stabilizer.
\end{enumerate}

Finally, the $(-1)$-curve on $\widetilde X$ is visible on the anti-canonical model as $B=\{s=0,\ y^2=x^3\}$, the non-contracted component of the anti-canonical member over $s=0$. Its strict transform is the unique $(-1)$-curve on the weak del Pezzo surface of type \hyperref[Tab1H]{$1H$}, and together with the old $E_8$-configuration it forms the anti-canonical member of type ${\rm II}^*$. This proves the claim.
\end{proof}

\begin{Corollary} \label{cor: 1H char3 nonJac revised}
Let $\widetilde Z$ be arising from an $\widetilde X$ of type \hyperref[Tab1H]{$1H$} and assume that $h^0(\widetilde Z,T_{\widetilde Z})\neq0$. Then, $\widetilde Z$ has one multiple fiber $3{\rm II}^*$, $\Aut^0_{\widetilde Z}\cong\mathbb G_a$ and such surfaces form a $1$-dimensional family.
\end{Corollary}

\begin{proof}
Everything except the moduli statement follows by combining Corollary \ref{cor: approach for classification non-jacobian} with Proposition \ref{prop: 1H char3 fixed points revised}, in whose proof we saw that the possible centers $\widetilde{P}$ form the $1$-dimensional locus $U=C^0\setminus\{0\}$ and that $\Aut_{\widetilde X}^0\cong\mathbb G_a$ fixes $C^0$ pointwise, hence also its closure $B$. Since $B$ is the unique $(-1)$-curve on $\widetilde X$, it is preserved by the entire automorphism group $\Aut(\widetilde{X})=\Aut_{\widetilde{X}}(k)$. Since $\Aut_{\widetilde X}$ is of finite type its component group is finite, and hence the $\Aut(\widetilde X)$-orbits in $U$ are finite.

We will now observe that every isomorphism identifying two such blow-ups descends to $\widetilde X$. Let $E$ be the exceptional curve of $\widetilde Z\to\widetilde X$, and let $\overline F\sim-K_{\widetilde Z}$ be the reduced fiber of type ${\rm II}^*$. If $D$ is a $(-1)$-curve on $\widetilde Z$, then $D.\overline F=-K_{\widetilde Z}.D=1$. Hence, by the known multiplicities of the components of a ${\rm II}^*$-fiber, $D$ meets the unique multiplicity-one component once and is disjoint from the eight components forming the $E_8$-configuration, just as $E$ is. Therefore $D-E$ is orthogonal to both $K_{\widetilde Z}$ and the $E_8$-configuration. Since the latter has full rank in $K_{\widetilde Z}^{\perp}/\mathbb Z K_{\widetilde Z}$, we have $D-E=nK_{\widetilde Z}$ for some $n\in\mathbb Z$. Then $0=(nK_{\widetilde Z})^2=(D-E)^2=-2-2\,D.E$, so $D.E=-1$, which forces $D=E$. Thus $E$ is the unique $(-1)$-curve on $\widetilde Z$.

Consequently, every isomorphism between two such blow-ups sends the exceptional curve to the exceptional curve and descends to an automorphism of $\widetilde X$ carrying one center to the other. Hence two centers in $U$ yield isomorphic surfaces only if they lie in the same finite $\Aut(\widetilde X)$-orbit. Therefore the resulting surfaces $\widetilde Z$ form a $1$-dimensional family up to isomorphism.
\end{proof}

\begin{Discussion} \label{discussion: 1H curves char3 revised}
The Jacobian surface $\widetilde Y$ associated with $\widetilde X$ has a unique singular fiber of type ${\rm II}^*$. Indeed, the discriminant of
$y^2=x^3+s^4x+s^3t^3$ is $-s^{12}$, hence $\widetilde Y$ is elliptic and its only singular fiber lies over $s=0$ and is of type ${\rm II}^*$. By \cite{OguisoShioda}, the Mordell--Weil group of the Jacobian surface is trivial.
The unique $(-1)$-curve on $\widetilde X$ is visible on the anti-canonical model as $B=\{s=0,\ y^2=x^3\}$.
Blowing up the gray base point on $B$ gives the zero-section of $\widetilde Y$, while the strict transform of $B$ becomes the affine component of the ${\rm II}^*$-fiber.

For the non-Jacobian surface $\widetilde Z$ of Corollary \ref{cor: 1H char3 nonJac revised}, we blow up a non-zero point $\widetilde P\in B^0\cong\mathbb G_a$. The strict transform of $B$ then becomes a $(-2)$-curve. Together with the old $E_8$-configuration, it forms the reduced multiple fiber of type ${\rm II}^*$. Since $\widetilde P$ has exact order $3$, the multiple fiber is $3{\rm II}^*$. The rank bound of Lemma \ref{lemma (-2)curves on Ztilde} is attained by this ${\rm II}^*$-fiber, hence there are no further reducible fibers.

The situation is summarized in the following Figure \ref{figure 1H configs jaco non-jaco} and Corollary \ref{cor: 1H (-2)configs char3}. Using the known multiplicities of the components of Kodaira--N{\'e}ron fibers, we observe that on $\widetilde Z$ the red exceptional curve is a $3$-section and we saw in the proof of the above Corollary \ref{cor: 1H char3 nonJac revised}, there is a unique $(-1)$-curve on $\widetilde{Z}$.

\begin{table}[H]
\begin{adjustbox}{center}
$
\begin{array}{ccccc}
 \begin{array}{c} \addstackgap[2pt]{\includegraphics[width=0.22\textwidth]{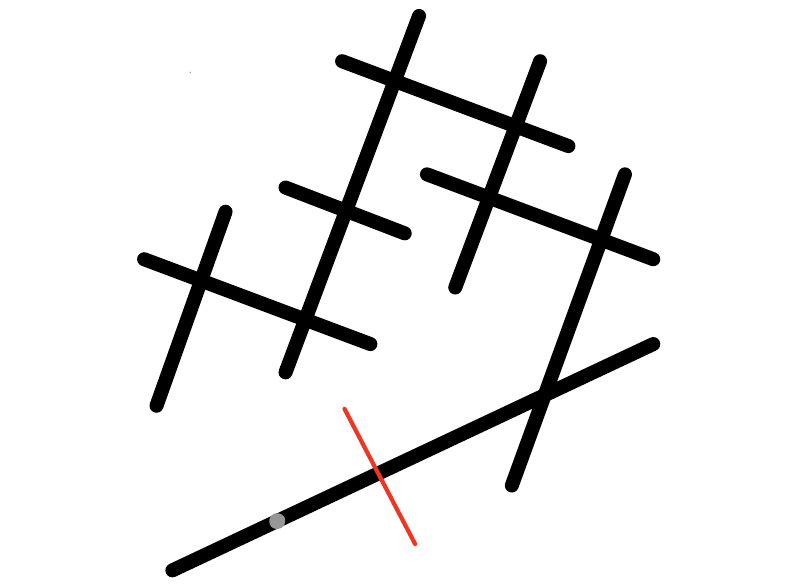} }\end{array}
 & \rightarrow
  & \begin{array}{c}\addstackgap[2pt]{\includegraphics[width=0.22\textwidth]{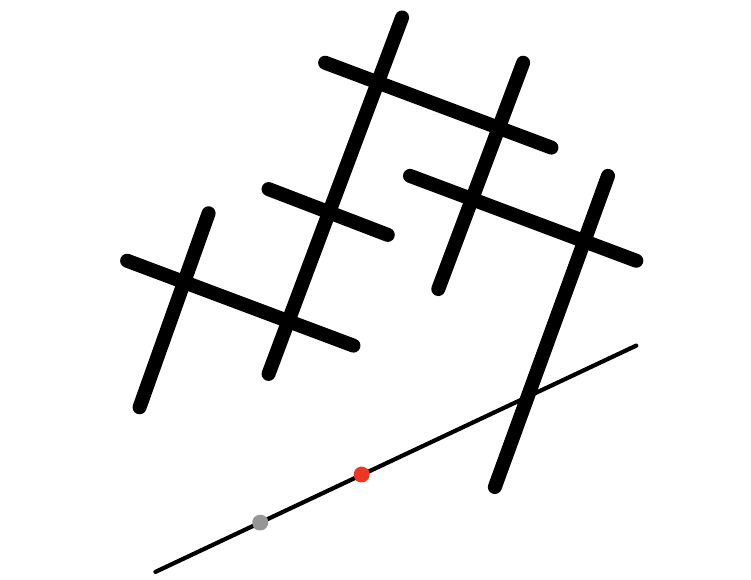} }\end{array}
  & \leftarrow
  & \begin{array}{c}\addstackgap[2pt]{\includegraphics[width=0.22\textwidth]{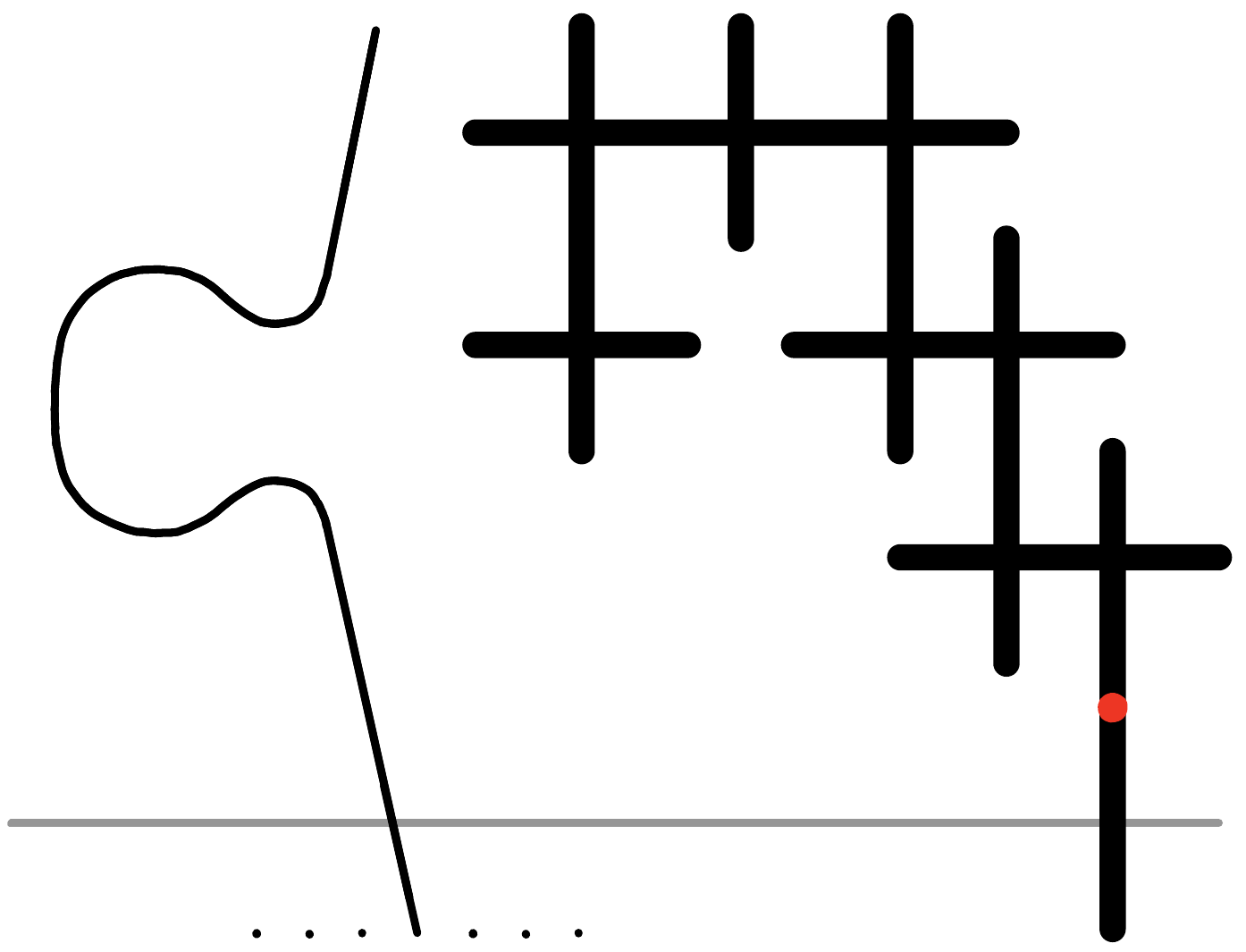} }\end{array}
  \\
  &
  & \downarrow
  &
  & \downarrow
  \\
  &
  & \begin{array}{c}\addstackgap[2pt]{\includegraphics[width=0.22\textwidth]{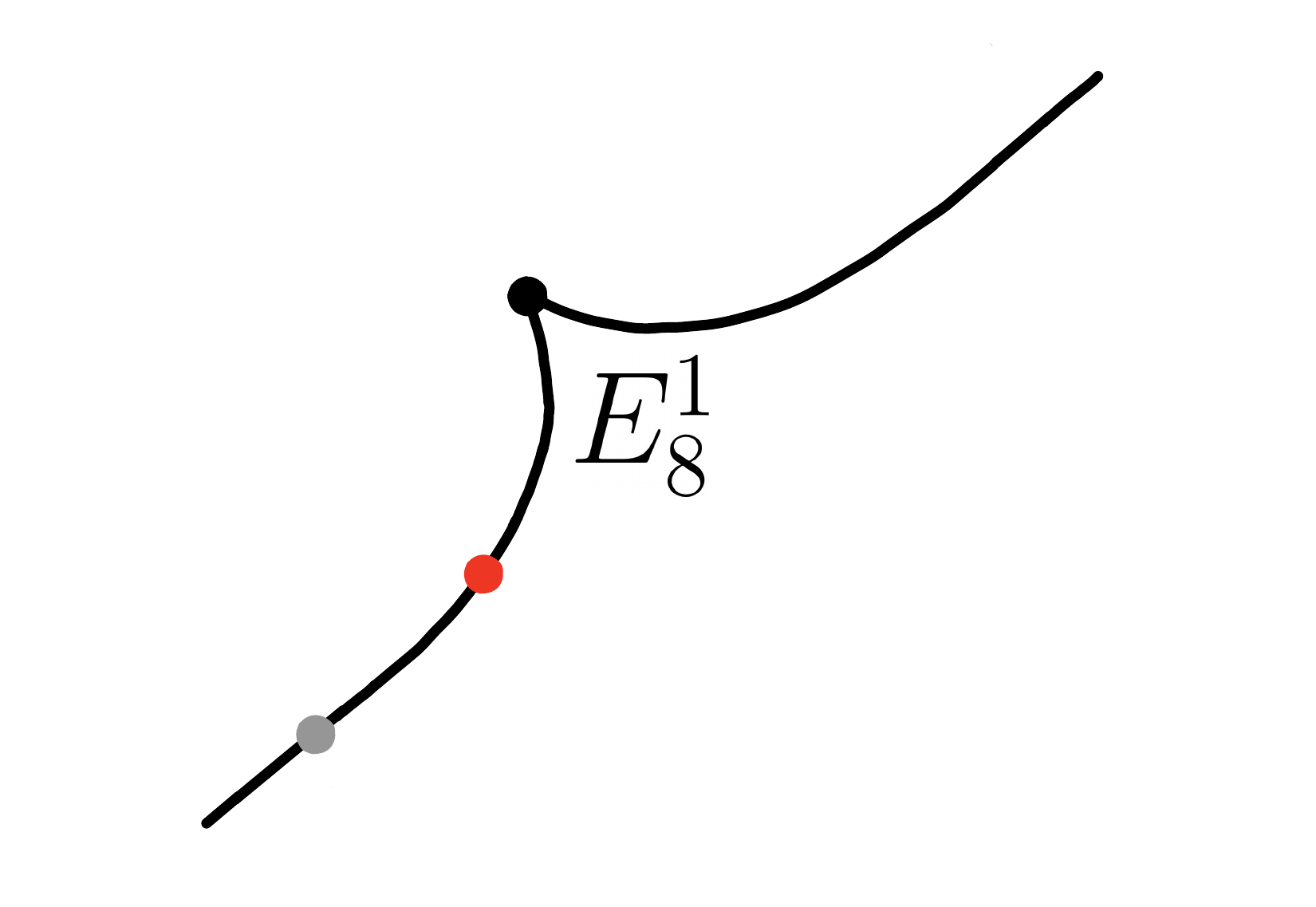}} \end{array}
  & \leftarrow
  & \begin{array}{c}\addstackgap[2pt]{ \includegraphics[width=0.22\textwidth]{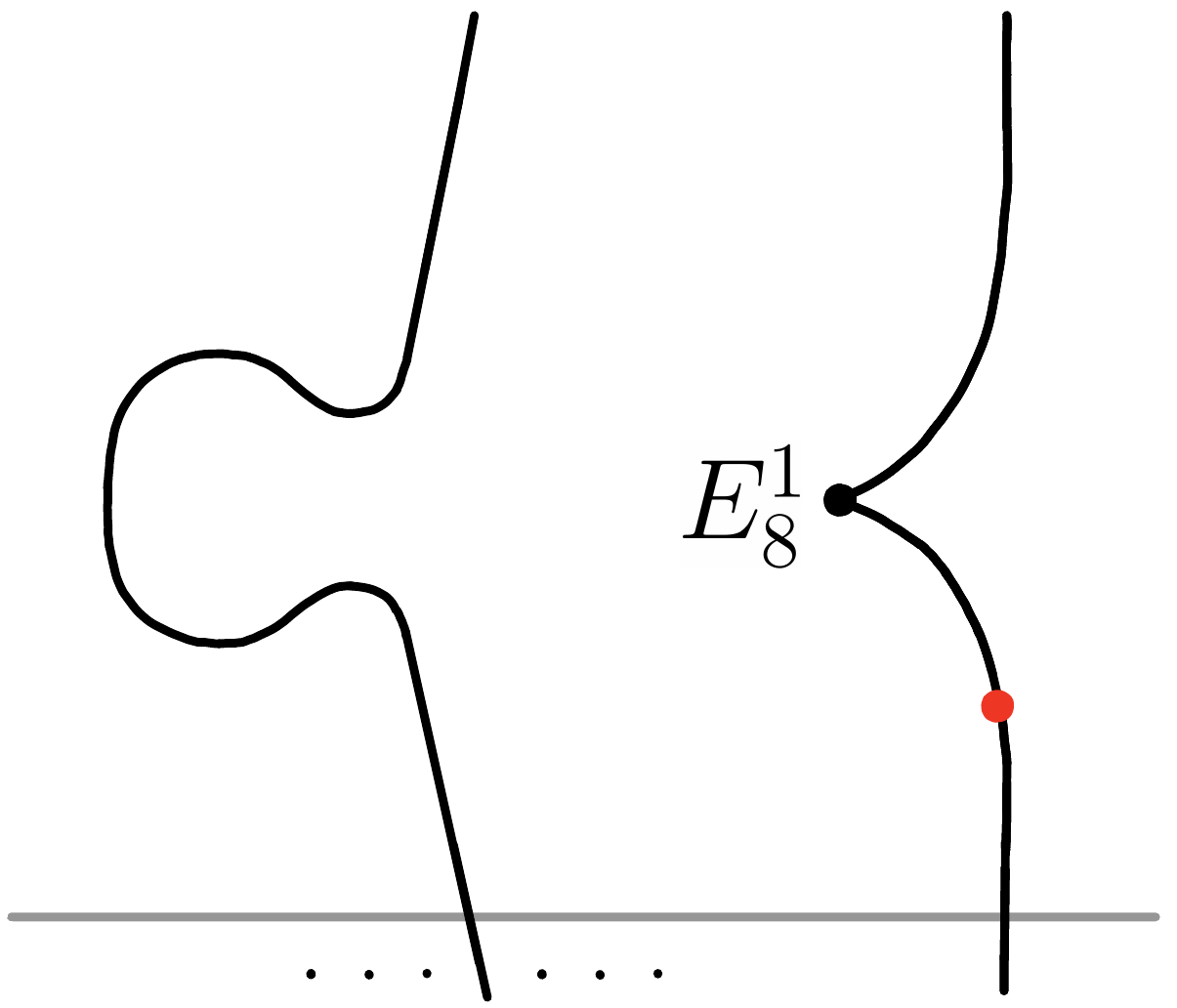}}  \end{array}
\end{array}
$
\captionof{figure}{Non-Jacobian and Jacobian fibrations with global vector fields originating from $\widetilde X$ of type \hyperref[Tab1H]{$1H$}, where the center is a non-zero point on the unique $(-1)$-component of the ${\rm II}^*$-member}
\label{figure 1H configs jaco non-jaco}
\end{adjustbox}
\end{table}
\end{Discussion}

\begin{Corollary} \label{cor: 1H (-2)configs char3}
Each of the $\widetilde Z$ of Corollary \ref{cor: 1H char3 nonJac revised} contains nine $(-2)$-curves with dual graph of type $\widetilde E_8$ forming configuration ${\rm II}^*$. Moreover, this ${\rm II}^*$-fiber is the unique multiple fiber and has multiplicity $3$. 
\end{Corollary}

 \subsection{Case \hyperref[Tab1I]{$1I$}} \label{subsection 1I char3 revised}
 This case exists only if $\Char(k)=p=3$.

\begin{Proposition} \label{prop: 1I char3 fixed points revised}
Let $\widetilde X$ be of type \hyperref[Tab1I]{$1I$}. Then, $\widetilde{P}$ is admissible if and only if one of the following holds.
\begin{enumerate}
\item\label{prop 1I IIstar revised} $C$ is of type ${\rm II}^*$ and $\widetilde P$ lies on the unique $(-1)$-curve contained in this anti-canonical member.
\item\label{prop 1I II revised} $C$ is of type ${\rm II}$ and $\widetilde P$ lies in the smooth locus of this irreducible cuspidal curve.
\end{enumerate}
Moreover, in case (\ref{prop 1I IIstar revised}), we have
$\bigl(\Stab_{\Aut_{\widetilde X}^0}(\widetilde P)\bigr)^0\cong\mathbb G_a$,
whereas in case (\ref{prop 1I II revised}), we have
$\bigl(\Stab_{\Aut_{\widetilde X}^0}(\widetilde P)\bigr)^0\cong\mu_3$.
In both cases, $m=3$.
\end{Proposition}

\begin{proof} 
By Proposition \ref{prop: char3 equations and liftable actions}, $X$ of $\widetilde X$ is given by
$
y^2=x^3+s^5t$,
and $\Aut_{\widetilde X}^0\cong\mathbb G_a\rtimes\mathbb G_m$ acts on $X$ by
$\mathbb G_a:\ [s:t:x:y]\mapsto[s:t-a^3s:x+as^2:y]
$
and
$\mathbb G_m:\ [s:t:x:y]\mapsto[\lambda s:\lambda^{-5}t:x:y]$. 
To find admissible $P \in X$ (according to Strategy \ref{strategy of proof non-Jaco}), we distinguish the following cases:

\begin{enumerate}[leftmargin=0.8cm]
\item[(a)]
We first consider the locus $s=0$. Away from the base point, we may set $t=1$. The additive subgroup $\mathbb G_a$ fixes this locus pointwise. The point $[0:1:0:0]$ is the unique $E_8$-singularity of $X$, whereas every other point $P=[0:1:x:y]$ with $(x,y) \neq(0,0)$ and $y^2=x^3$,
lies in the smooth locus of the cuspidal curve $C=X\cap\{s=0\}$. Hence $C^0\cong\mathbb G_a$, and all these points are non-zero points of exact order $3$. The $\mathbb G_m$-factor acts on this locus as $[0:1:x:y] \mapsto [0:1:\lambda^{10}x: \lambda^{15}y]$, hence an element $(a, \lambda) \in \mathbb{G}_a \rtimes \mathbb{G}_m$ has non-trivial stabilizer if and only if $\lambda^{10}=1=\lambda^{15}$. So, $\Stab{\mathbb G_a\rtimes\mathbb G_m}(P) = \mathbb{G}a \rtimes \mu_5$ and, since we are working in characteristic $3$, $(\Stab{\mathbb G_a\rtimes\mathbb G_m}(P))^0 = \mathbb{G}_a$.

\item[(b)]
It remains to consider the locus $s\neq0$. We may set $s=1$, in which case the combined action is $[1:t:x:y]\longmapsto
[1:\lambda^{-6}t-a^3:\lambda^{-2}x+a:\lambda^{-3}y]$.
Let $P=[1:t_0:x_0:y_0]$. Since $P$ lies in the smooth locus of the anti-canonical curve through it, we have $y_0\neq0$: the point with $y=0$ is the cusp of the curve $y^2=x^3+t_0$. The stabilizer equations are
\[
\lambda^{-3}y_0=y_0,
\qquad
\lambda^{-2}x_0+a=x_0,
\qquad
\lambda^{-6}t_0-a^3=t_0.
\]
The first equation is equivalent to $\lambda^3=1$, and the second then uniquely determines $a=(1-\lambda^{-2})x_0$ and so the third equation is automatic. Hence the stabilizer is $\mu_3$.
\end{enumerate}

Finally, the curve $B=\{s=0,\ y^2=x^3\}\subset X$ is the non-contracted component of the anti-canonical member over $s=0$. Its strict transform is the unique $(-1)$-curve on the weak del Pezzo surface of type \hyperref[Tab1I]{$1I$}; together with the old $E_8$-configuration, it forms the anti-canonical member of type ${\rm II}^*$. All anti-canonical curves meeting the chart $s=1$ are irreducible cuspidal curves and hence are of type ${\rm II}$. This proves the claim.
\end{proof}

\begin{Corollary} \label{cor: 1I char3 nonJac revised}
Let $\widetilde Z$ be arising from an $\widetilde X$ of type \hyperref[Tab1I]{$1I$}. Then the following hold.
\begin{enumerate}
\item[(1)]\label{cor 1I uniqueness IIstar}
If $C$ is of type ${\rm II}^*$, then $\widetilde Z$ has the unique multiple fiber $3{\rm II}^*$, $\Aut_{\widetilde Z}^0\cong\mathbb G_a$, and $\widetilde Z$ is unique up to isomorphism.
\item[(2)]\label{cor 1I uniqueness II}
If $C$ is of type ${\rm II}$, then $\widetilde Z$ has the unique multiple fiber $3{\rm II}$, $\Aut_{\widetilde Z}^0\cong\mu_3$, and $\widetilde Z$ is unique up to isomorphism.
\end{enumerate}
\end{Corollary}

\begin{proof}
Everything except the uniqueness statements follows by combining Corollary \ref{cor: approach for classification non-jacobian} with Proposition \ref{prop: 1I char3 fixed points revised}.

For (\hyperref[cor 1I uniqueness IIstar]{1}), the possible centers form $B^0\setminus\{0\}$, where $B$ denotes the strict transform on $\widetilde X$ of the curve $\{s=0,\ y^2=x^3\}\subset X$. On this locus, the $\mathbb G_m$-factor acts transitively, thus all corresponding blow-ups are isomorphic.

For (\hyperref[cor 1I uniqueness II]{2}), we work on the chart $s=1$. By Proposition \ref{prop: 1I char3 fixed points revised}, the admissible centers are the smooth points $(t,x,y)$ of the cuspidal curves $y^2=x^3+t$, that is, the points with $y\neq0$. They form a single $\mathbb G_a\rtimes\mathbb G_m$-orbit: choosing $\lambda$ such that $\lambda^{-3}y=1$ and then $a=-\lambda^{-2}x$ sends $(t,x,y)$ to $(1,0,1)$. Hence all corresponding blow-ups are isomorphic as well.
\end{proof}

\begin{Discussion} \label{discussion: 1I curves char3 revised}
The Jacobian surface $\widetilde Y$ associated with $\widetilde X$ is quasi-elliptic and has a unique reducible fiber of type ${\rm II}^*$. Indeed, for $y^2=x^3+s^5t$ the discriminant is identically zero, and the anti-canonical member over $s=0$ gives the only reducible fiber. Thus this fiber contributes the root lattice $E_8$, and the Mordell--Weil group is trivial.
The unique $(-1)$-curve on $\widetilde X$ is visible on the anti-canonical model as
$B=\{s=0,\ y^2=x^3\}$. Blowing up the base point on $B$ gives the zero-section of $\widetilde Y$, while the strict transform of $B$ becomes the affine component of the ${\rm II}^*$-fiber.

To determine the number and configuration of $(-2)$-curves on the non-Jacobian surface $\widetilde Z$, we treat the cases (\hyperref[cor 1I uniqueness IIstar]{1}) and (\hyperref[cor 1I uniqueness II]{2}) of Corollary \ref{cor: 1I char3 nonJac revised} separately.

\begin{enumerate}
\item[(1)]\label{discussion 1I multiple IIstar}
By the proof of Proposition \ref{prop: 1I char3 fixed points revised}(\ref{prop 1I IIstar revised}), the point $\widetilde P$ lies on the unique $(-1)$-curve $B$ contained in the anti-canonical member of type ${\rm II}^*$. By Corollary \ref{cor: 1I char3 nonJac revised}(\hyperref[cor 1I uniqueness IIstar]{1}), all such choices lead to isomorphic surfaces.

After blowing up such a point, the strict transform of $B$ becomes a $(-2)$-curve on $\widetilde Z$. Together with the old $E_8$-configuration, it forms the reduced multiple fiber of type ${\rm II}^*$. Since $\widetilde P$ has exact order $3$, the multiple fiber is $3{\rm II}^*$. The rank bound of Lemma \ref{lemma (-2)curves on Ztilde} is attained by this ${\rm II}^*$-fiber, hence there are no further reducible fibers.

The situation is summarized in Figure \ref{figure 1I multiple IIstar configs jaco non-jaco} and Corollary \ref{cor: 1I (-2)configs char3}(\ref{cor 1I configs multiple IIstar}). The red exceptional curve $E$ is a $(-1)$-curve and hence, by adjunction, a $3$-section, in accordance
with the intersection behavior shown in the figure.

Moreover, the surface obtained in this case is not isomorphic to any of
the surfaces of Corollary \ref{cor: 1H char3 nonJac revised}. Indeed,
the same lattice argument as in the proof of that corollary applies here
and shows that $E$ is the unique $(-1)$-curve on $\widetilde Z$. Hence any
isomorphism from $\widetilde Z$ to a surface arising from case
\hyperref[Tab1H]{$1H$} would therefore descend to an isomorphism between weak
del Pezzo surfaces of types \hyperref[Tab1I]{$1I$} and
\hyperref[Tab1H]{$1H$}, which is impossible (for example, by  their connected automorphism schemes being different).
\end{enumerate}

\begin{enumerate}
    \item[] \begin{Remark}
\label{rem: 1H elliptic versus 1I quasi-elliptic}
The distinction from the surfaces obtained in case
\hyperref[Tab1H]{$1H$} can also be seen directly from the geometric
generic fibers, namely every surface of Corollary
\ref{cor: 1H char3 nonJac revised} is elliptic, whereas the unique
surface in Corollary
\ref{cor: 1I char3 nonJac revised}
(\hyperref[cor 1I uniqueness IIstar]{1}) is quasi-elliptic.

Indeed, write the image of the blow-up center on the anti-canonical
model as $P_u=[0:1:u^2:u^3], u\in k^\times$.
The Halphen pencil on the resulting surface is induced on $X$ by the
two sections $s^3$ and $y-u^3t^3$
of $\mathcal O_X(3)$. Thus, if $\tau$ denotes the parameter of the pencil, its generic member is given by
\[
y-u^3t^3=\tau s^3.
\]
\begin{itemize}
    \item
First consider a surface arising from case
\hyperref[Tab1H]{$1H$}, whose anti-canonical model has equation
$y^2=x^3+s^4x+s^3t^3$.
Away from $s=0$, substitution of
$y=u^3t^3+\tau s^3$ gives an equation whose derivative with respect to
$x$ is $s^4$, and hence the generic member is smooth there. On the
chart $t=1$, set $\xi=x-u^2$ and $\eta=y-u^3$.
The equation becomes
\[
2u^3\eta+\eta^2
 =
\xi^3+s^3+u^2s^4+s^4\xi.
\]

Substituting $\eta=\tau s^3$ and using the blow-up chart
$\xi=sv$, the strict transform of the generic member is given by
\[
v^3+1-2u^3\tau+u^2s+s^2v-\tau^2s^3=0.
\]
For $s\neq0$, its derivative with respect to $v$ is $s^2\neq0$,
whereas along the exceptional divisor $s=0$, its derivative with
respect to $s$ is $u^2\neq0$. Hence, the geometric generic fiber is smooth, so every surface
obtained from case \hyperref[Tab1H]{$1H$} is elliptic.

\item
For case \hyperref[Tab1I]{$1I$}, the anti-canonical model is instead $y^2=x^3+s^5t$.
Away from $s=0$, after the same substitution, the derivative of the
resulting equation with respect to $t$ is $s^5$, so the generic member is again smooth there. On the chart $t=1$, the local equation at $P_u$
is
\[
2u^3\eta+\eta^2=\xi^3+s^5.
\]
After substituting $\eta=\tau s^3$ and passing to the blow-up chart
$\xi=sv$, the strict transform is given by
\[
v^3-2u^3\tau+s^2-\tau^2s^3=0.
\]
Over $\overline{k(\tau)}$, choose $v_0$ such that $v_0^3=2u^3\tau$ and set $w=v-v_0$. Since the characteristic is $3$, the equation then
becomes
\[
w^3+s^2-\tau^2s^3
=
w^3+s^2(1-\tau^2s)
=0.
\]
As $1-\tau^2s$ is a unit and $2$ is invertible, a formal change of the
$s$-coordinate transforms this equation into $w^3+(s')^2=0$.
Thus the geometric generic fiber has a cusp and is therefore a
cuspidal rational curve. Hence the unique surface obtained in case
\hyperref[Tab1I]{$1I$}
(\hyperref[cor 1I uniqueness IIstar]{1}) is quasi-elliptic.
\end{itemize}
\end{Remark}
\end{enumerate}

\begin{table}[H]
\begin{adjustbox}{center}
$
\begin{array}{ccccc}
 \begin{array}{c} \addstackgap[2pt]{\includegraphics[width=0.22\textwidth]{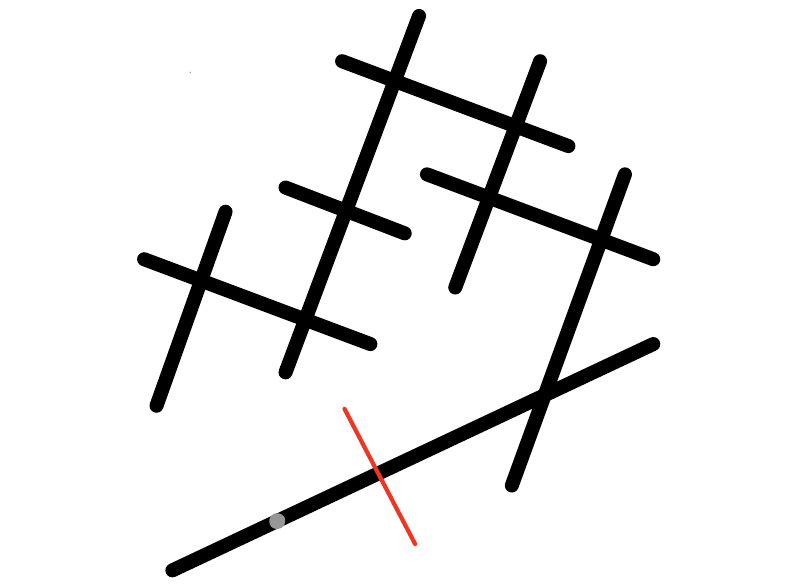} }\end{array}
 & \rightarrow
  & \begin{array}{c}\addstackgap[2pt]{\includegraphics[width=0.22\textwidth]{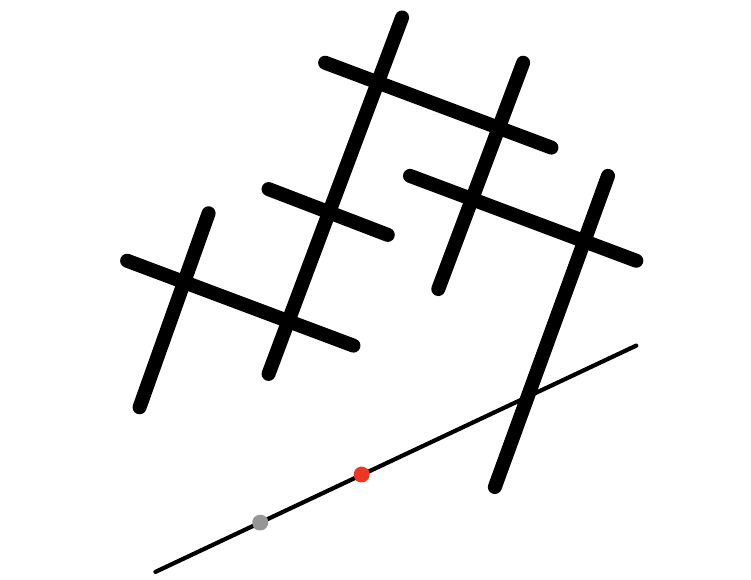} }\end{array}
  & \leftarrow
  & \begin{array}{c}\addstackgap[2pt]{\includegraphics[width=0.22\textwidth]{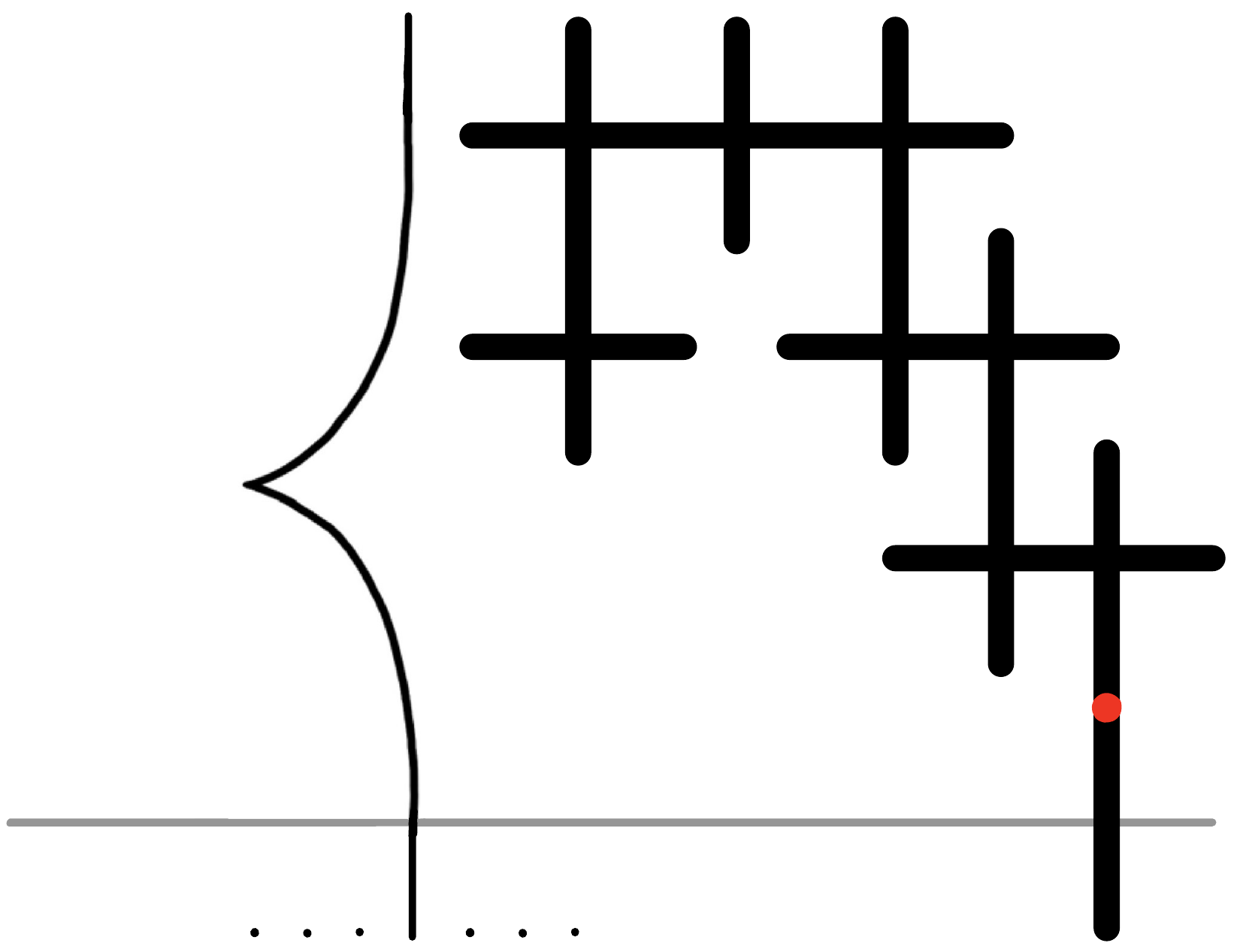} }\end{array}
  \\
  &
  & \downarrow
  &
  & \downarrow
  \\
  &
  & \begin{array}{c}\addstackgap[2pt]{\includegraphics[width=0.22\textwidth]{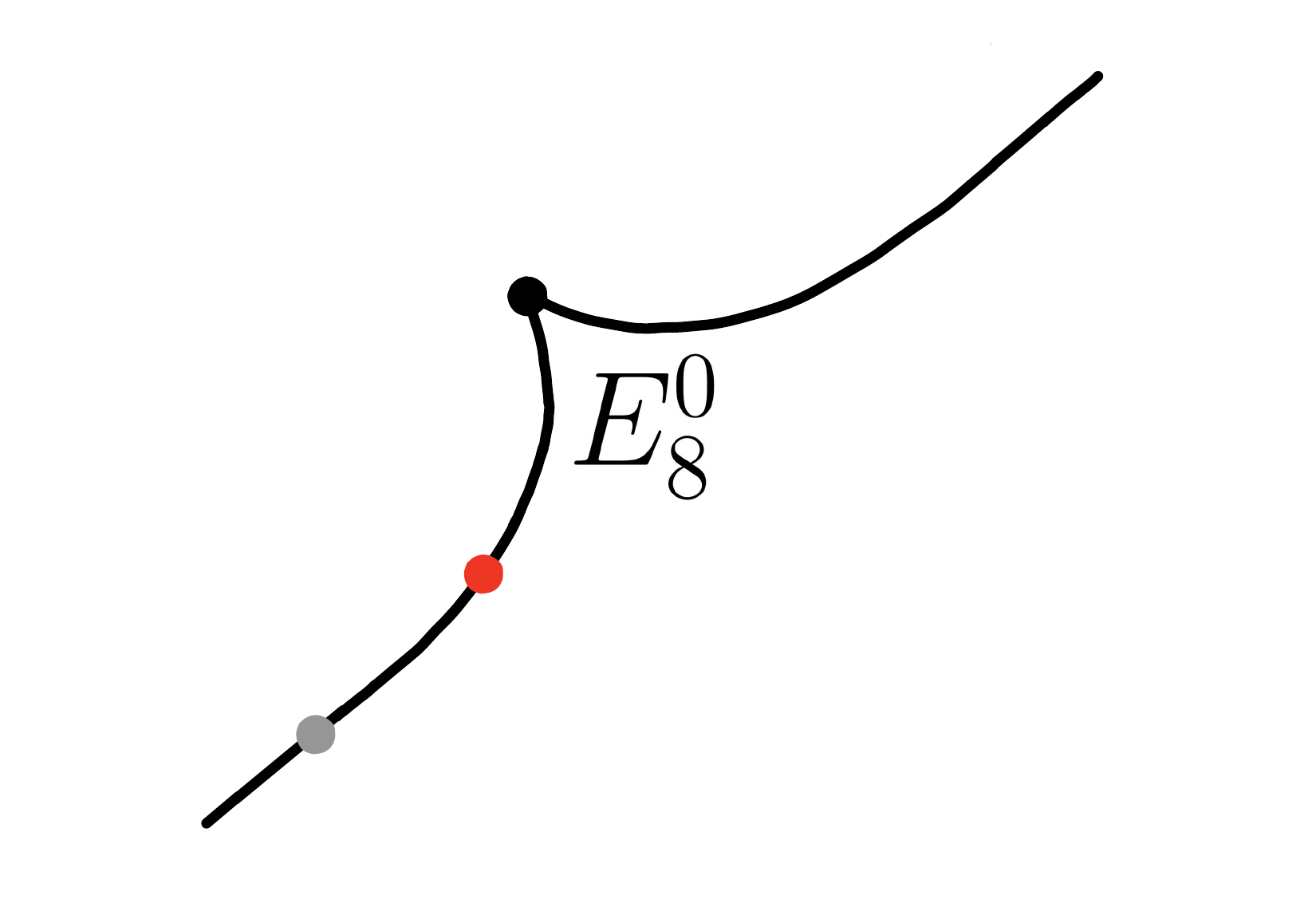}} \end{array}
  & \leftarrow
  & \begin{array}{c}\addstackgap[2pt]{ \includegraphics[width=0.22\textwidth]{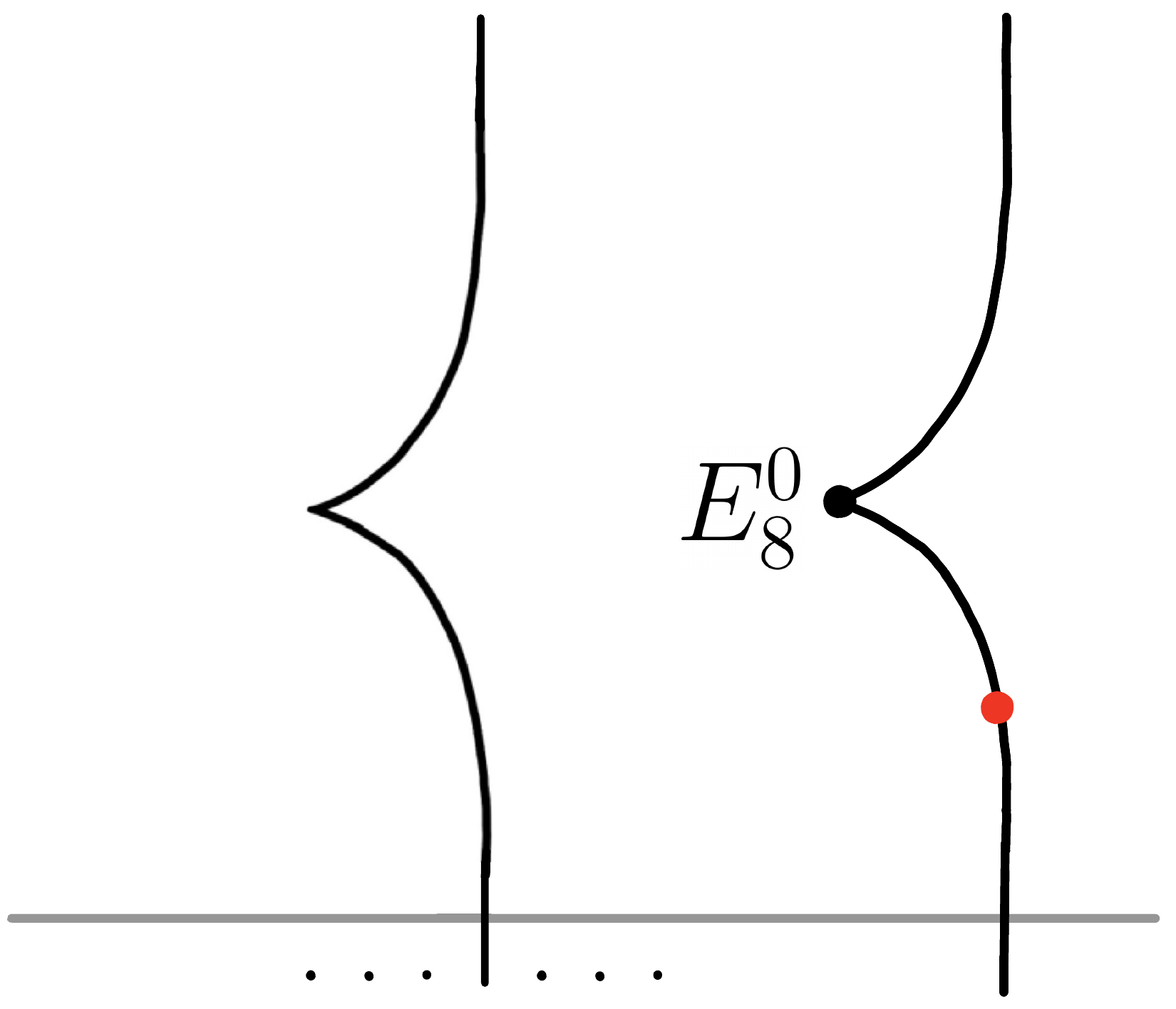}}  \end{array}
\end{array}
$
\captionof{figure}{Non-Jacobian and Jacobian fibrations with global vector fields originating from $\widetilde X$ of type \hyperref[Tab1I]{$1I$}, where the center lies on the $(-1)$-component of the ${\rm II}^*$-member}
\label{figure 1I multiple IIstar configs jaco non-jaco}
\end{adjustbox}
\end{table}

\begin{enumerate}
\item[(2)]\label{discussion 1I multiple II}
By the proof of Proposition \ref{prop: 1I char3 fixed points revised}(\ref{prop 1I II revised}), the point $\widetilde P$ is a smooth point of an irreducible cuspidal anti-canonical curve of type ${\rm II}$. By Corollary \ref{cor: 1I char3 nonJac revised}(\hyperref[cor 1I uniqueness II]{2}), we may choose the point corresponding to $[1:1:0:1]$ on $X$.

After blowing up this point, the strict transform of the type ${\rm II}$ curve becomes the reduced multiple fiber, so the multiple fiber is $3{\rm II}$. The old $E_8$-configuration remains vertical on $\widetilde Z$ and, being connected, is contained in a single fiber. The Kodaira--N{\'e}ron classification shows that the only fiber type containing an $E_8$-configuration of $(-2)$-curves is ${\rm II}^*$. So there has to be one further $(-2)$-component, not visible as a negative curve on $\widetilde X$ completing ${\rm II}^*$, and since it attains the rank bound of Lemma \ref{lemma (-2)curves on Ztilde}, hence there are no further reducible fibers. 

In order to understand how this new $(-2)$-component on $\widetilde Z$ (drawn then in purple in Figure
\ref{figure 1I multiple II configs jaco non-jaco} below) intersects the exceptional divisor $E$ of the blow-up of $\widetilde P$ and the strict
transform $B$ of the unique $(-1)$-curve on
$\widetilde X$, let us observe the following.
\end{enumerate}

\begin{enumerate}
\item[] \emph{This $\widetilde Z$ is not ``new'':}
By adjunction, both $E$ and $B$ are $3$-sections. Since the center
$\widetilde P$ does not lie on the old $E_8$-configuration, the exceptional
divisor $E$ is disjoint from its strict transform. The multiplicities of
the components of the ${\rm II}^*$-fiber therefore show that $E$ meets the
new $(-2)$-component with total intersection multiplicity $3$.
Moreover, the component of the old $E_8$-configuration met by $B$ has
multiplicity $2$ in the ${\rm II}^*$-fiber. Since $B$ is a $3$-section,
it follows that $B$ has one additional intersection, of multiplicity $1$,
with the new $(-2)$-component, even though at this point we do not yet know
the precise position of this intersection.

Nevertheless, contracting $B$ yields another weak del Pezzo surface of
degree $1$ with an \linebreak $E_7$-configuration of $(-2)$-curves and, by
Blanchard's Lemma, still with global vector fields. By the classification in
\cite[Table 6]{WeakDelPezzoGlobalVectorFields}, this is the unique surface
of type \hyperref[Tab1F]{$1F$}. Since $\widetilde Z$ has the multiple
fiber $3{\rm II}$ by Corollary
\ref{cor: 1I char3 nonJac revised}(\hyperref[cor 1I uniqueness II]{2}),
Corollary
\ref{cor: 1F char3 nonJac revised}(\hyperref[cor 1F uniqueness(1) solo]{1})
shows that $\widetilde Z$ is the unique surface in the first subcase of
case \hyperref[Tab1F]{$1F$}. By comparison with Figure
\ref{figure 1F multiple II configs jaco non-jaco}, we see that the
intersection behavior summarized in Figure
\ref{figure 1I multiple II configs jaco non-jaco} below is correct.
\end{enumerate}

\begin{table}[H]
\begin{adjustbox}{center}
$
\begin{array}{ccccc}
 \begin{array}{c} \addstackgap[2pt]{\includegraphics[width=0.22\textwidth]{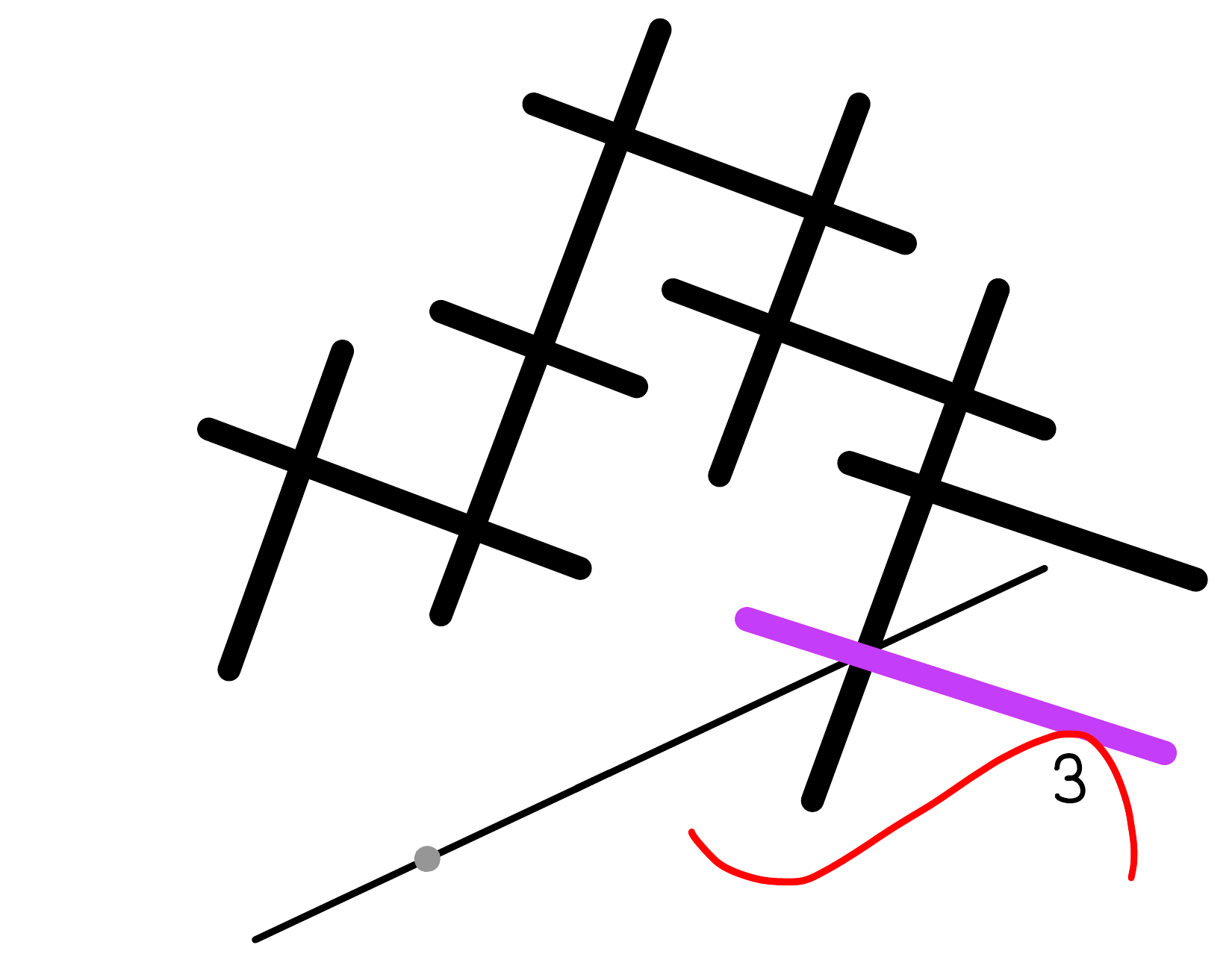} }\end{array}
 & \rightarrow
  & \begin{array}{c}\addstackgap[2pt]{\includegraphics[width=0.22\textwidth]{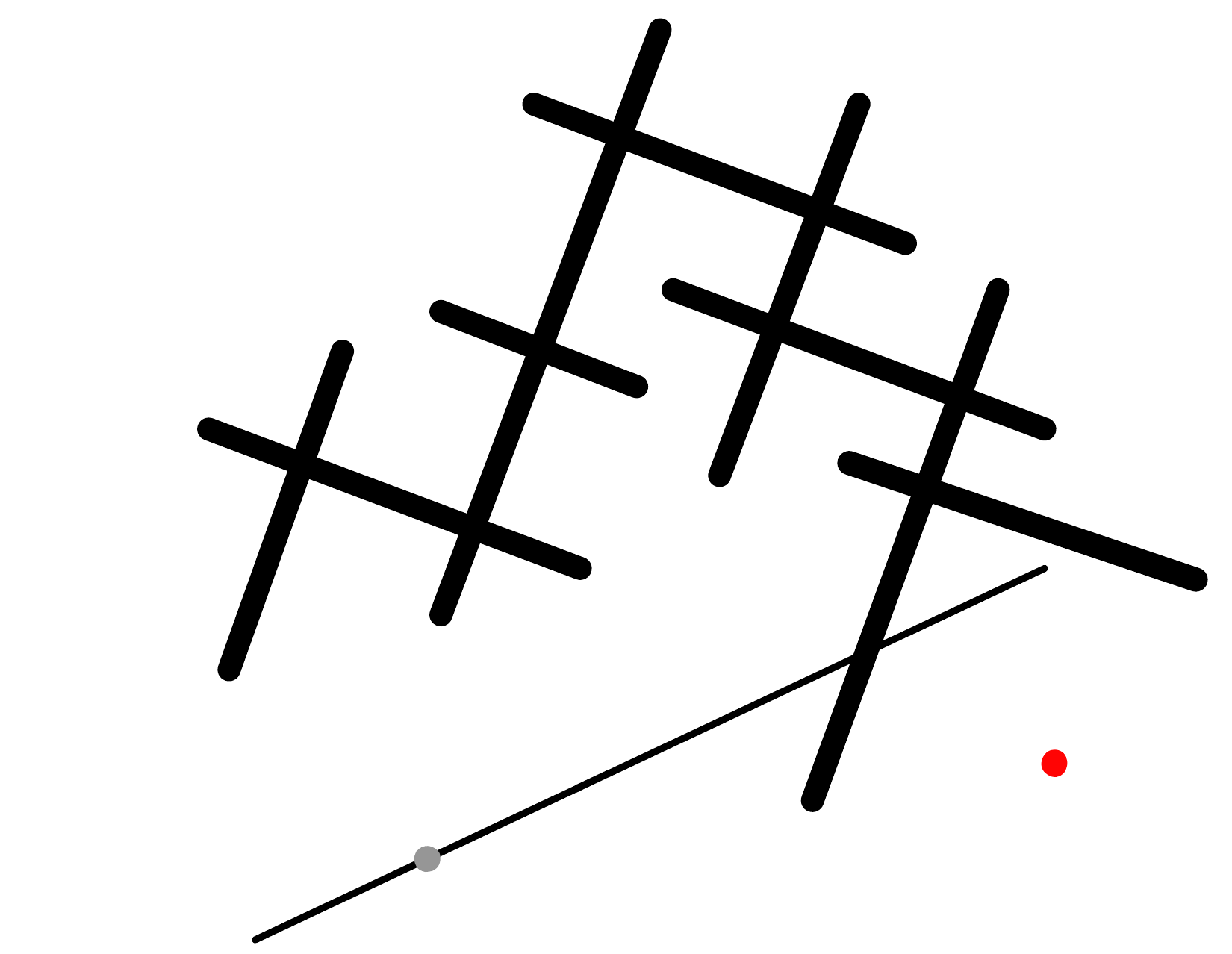} }\end{array}
  & \leftarrow
  & \begin{array}{c}\addstackgap[2pt]{\includegraphics[width=0.22\textwidth]{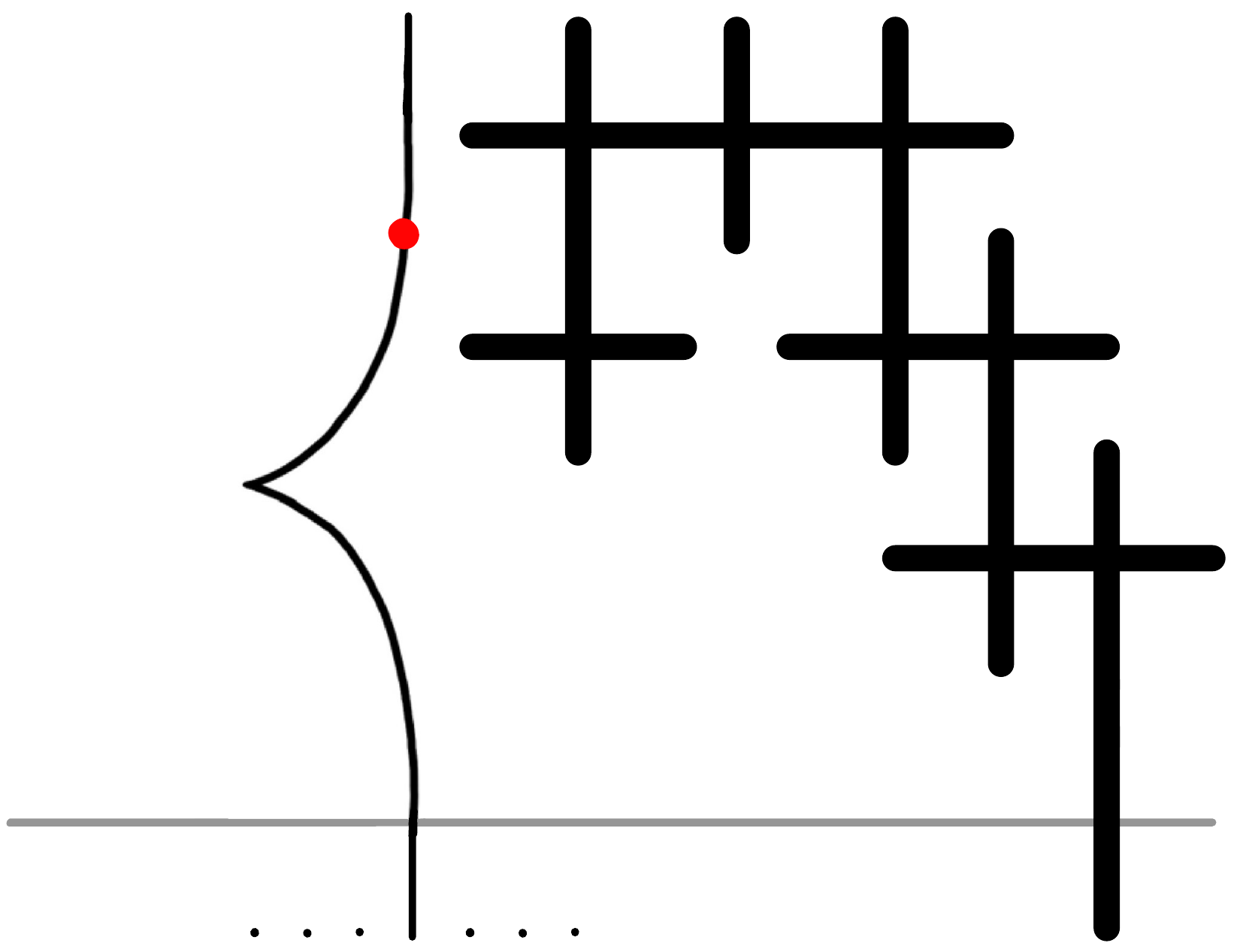} }\end{array}
  \\
  &
  & \downarrow
  &
  & \downarrow
  \\
  &
  & \begin{array}{c}\addstackgap[2pt]{\includegraphics[width=0.22\textwidth]{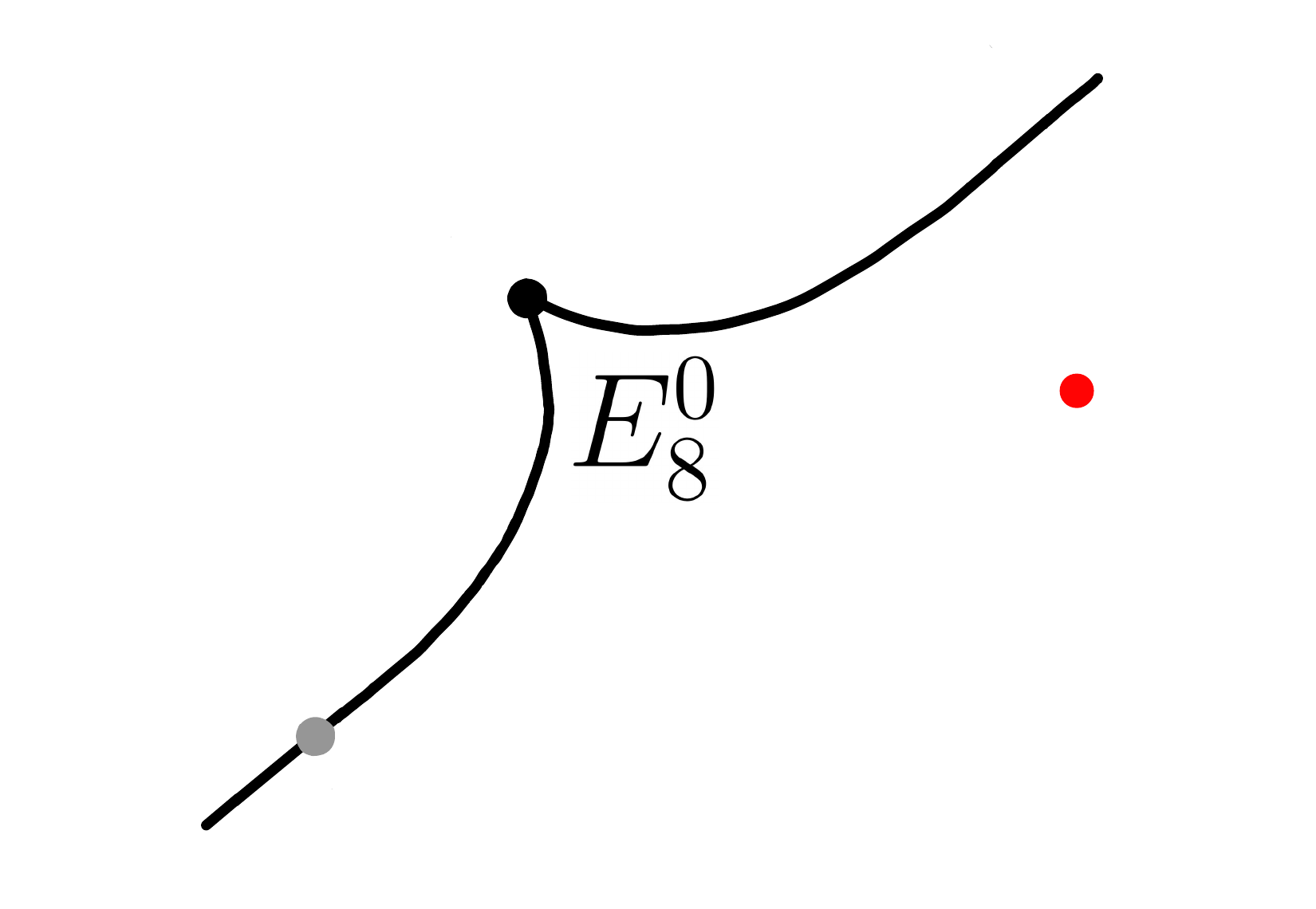}} \end{array}
  & \leftarrow
  & \begin{array}{c}\addstackgap[2pt]{ \includegraphics[width=0.22\textwidth]{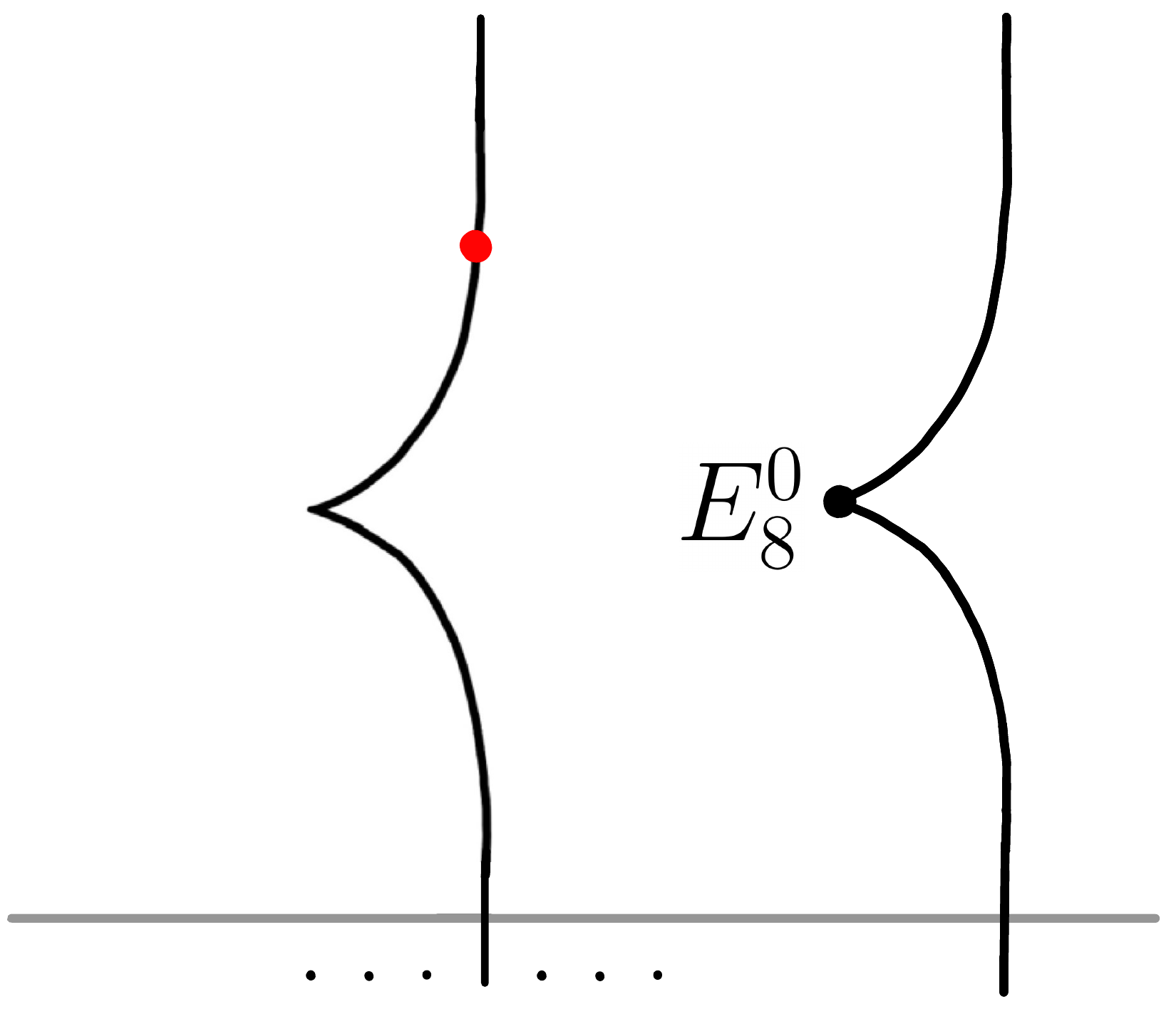}}  \end{array}
\end{array}
$
\captionof{figure}{Non-Jacobian and Jacobian fibrations with global vector fields originating from $\widetilde X$ of type \hyperref[Tab1I]{$1I$}, where the center lies on a type ${\rm II}$ curve}
\label{figure 1I multiple II configs jaco non-jaco}
\end{adjustbox}
\end{table}
\end{Discussion}

\begin{Corollary} \label{cor: 1I (-2)configs char3}
Let $\widetilde Z$ be one of the surfaces of Corollary \ref{cor: 1I char3 nonJac revised}. Then one of the following holds.
\begin{enumerate}
\item\label{cor 1I configs multiple IIstar}
$\widetilde Z$ contains nine $(-2)$-curves with dual graph of type $\widetilde E_8$ forming configuration ${\rm II}^*$. Moreover, this ${\rm II}^*$-fiber is the unique multiple fiber and has multiplicity $3$.

\item\label{cor 1I configs multiple II}
$\widetilde Z$ contains nine $(-2)$-curves with dual graph of type $\widetilde E_8$ forming configuration ${\rm II}^*$. Moreover, the unique multiple fiber is $3{\rm II}$. Furthermore, $\widetilde Z$ is isomorphic to the unique surface in Corollaries \ref{cor: 1F char3 nonJac revised}(\hyperref[cor 1F uniqueness(1) solo]{1}) and \ref{cor: 1F (-2)configs char3}(\ref{cor 1F configs multiple II}), obtained from a weak del Pezzo surface of type \hyperref[Tab1F]{$1F$}.
\end{enumerate}
\end{Corollary}

\hfill $\blacksquare$

\newpage

\bibliographystyle{alpha}
\bibliography{quellipt}

\end{document}